\documentclass[10pt,letterpaper]{article}
\usepackage{amssymb}
\usepackage{amsmath}
\usepackage{amsrefs}
\usepackage{amsthm}
\usepackage{array}
\usepackage{rotating}
\usepackage{verbatim}
\usepackage{enumitem}
 
\theoremstyle{plain}
\newtheorem{theorem}[equation]{Theorem}
\newtheorem{corollary}[equation]{Corollary}
\newtheorem{lemma}[equation]{Lemma}
\newtheorem{conjecture}[equation]{Conjecture}
\newtheorem{proposition}[equation]{Proposition}

\theoremstyle{definition}
\newtheorem{algorithm}[equation]{Algorithm}
\theoremstyle{remark}

\newtheoremstyle{indenteddefinition}{\topsep}{\topsep}{\addtolength{\leftskip}{2.0em}}{-0em}{\bfseries}{.}{
}{} 
\theoremstyle{indenteddefinition}
\newtheorem{definition}[equation]{Definition}
\newtheorem{example}[equation]{Example}

\DeclareMathOperator\Ad{Ad}

\DeclareMathOperator\Aut{Aut}
\DeclareMathOperator\Norm{Norm}
\DeclareMathOperator\Hom{Hom}

\DeclareMathOperator\Cent{Cent}

\DeclareMathOperator\Ind{Ind}
\DeclareMathOperator\Int{Int}

\DeclareMathOperator\End{End}

\DeclareMathOperator\tr{tr}

\DeclareMathOperator\op{op}

\DeclareMathOperator\sig{sig}
\DeclareMathOperator\diag{diag}

\renewcommand{\theequation}{\thesection.\arabic{equation}}

\begin{document}
\title{Parameters for twisted representations}
\author{Jeffrey D. Adams\thanks{The first author was supported in part
    by NSF grant DMS-1317523.} \\Department of
  Mathematics \\ University of Maryland
\and David A. Vogan, Jr.\thanks{The second author was supported in part by NSF grant DMS-0967272.}\\E17-442, Department of Mathematics\\ MIT,
  Cambridge, MA 02139}

\date{\today}
\maketitle
\newcommand{\X}{\mathcal X}
\newcommand{\inv}{^{-1}}
\newcommand{\Gext}{{}^\Delta G}
\newcommand{\Kext}{{}^{\delta_0}K}
\renewcommand{\int}{\mathrm{int}}
\newcommand\wt{\widetilde}
\newcommand\wh{\widehat}
\newcommand{\Q}{\mathbb Q}
\newcommand{\Z}{\mathbb Z}
\newcommand{\g}{\mathfrak g}
\newcommand{\sgn}{\mathrm{sgn}}
\newcommand{\ch}[1]{{}^\vee\negthinspace#1}
\newcommand{\ID}{\mathcal I\mathcal D}

\section{Introduction}\label{sec:intro}

One of the central problems in representation theory is understanding
irreducible unitary representations. The reason is that in many
applications of linear algebra (like those of representation theory to
harmonic analysis) the notion of {\em length} of vectors is
fundamentally important. Unitary representations are exactly those
preserving a good notion of length.

The paper \cite{herm} provides an algorithm for calculating the
irreducible unitary representations of a real reductive Lie group. The
purpose of this paper is to address a problem arising in the
implementation of this algorithm. In order to explain the problem, we
need to describe briefly (or at least more briefly than \cite{herm})
the nature of the algorithm.

In order to minimize technicalities, we will provide in the
introduction complete details only for {\em finite-dimensional}
representations. For a real reductive Lie group, the theory of Harish-Chandra
modules provides a complete way to deal with the
complications attached to infinite-dimensional representations.

To study unitary representations it is natural to study the larger class of 
representations with invariant Hermitian forms.
Here is the underlying formalism.

\begin{definition}\label{def:herm} Suppose $V$ and $W$ are complex
  vector spaces. A {\em sesquilinear pairing} is a map
$$\langle \cdot,\cdot\rangle \colon V\times W \rightarrow {\mathbb
  C}$$
that is linear in $V$ and conjugate-linear in $W$:
$$\langle av_1 + bv_2,w\rangle = a\langle v_1,w\rangle + b\langle
v_2,w\rangle, \quad \langle v,cw_1 + dw_2\rangle = \overline{c}\langle
v,w_1\rangle  + \overline{d}\langle v,w_2\rangle.$$
In case $V=W$, the pairing is called {\em Hermitian} if in addition 
$$\langle v_1,v_2 \rangle = \overline{\langle v_2,v_1\rangle}. $$

If $\langle,\rangle$ is a {\em nondegenerate} Hermitian pairing on a
{\em finite-dimensional} vector space $V$, then there is a one-to-one
correspondence between linear maps $A\in \Hom(V,V)$ and sesquilinear
pairings $\langle,\rangle_A$ on $V$, defined by
$$\langle v,w\rangle_A = \langle v,Aw\rangle.$$
In this correspondence, $\langle,\rangle_A$ is Hermitian if and only
if $A$ is self-adjoint with respect to $\langle,\rangle$.
\end{definition}

\begin{definition}\label{def:invform} Suppose $(\pi,V)$ is a
  representation of a group $G_1$ on a finite-dimensional
  complex vector 
  space $V$. An {\em invariant Hermitian form} on $V$ is a
  Hermitian pairing
$$\langle\cdot,\cdot\rangle \colon V\times V \rightarrow {\mathbb C}$$
with the property that
$$\langle \pi(g)v,\pi(g)w\rangle = \langle v,w\rangle \qquad (v,w\in
V, \ g\in G_1).$$
The representation $\pi$ is {\em Hermitian} if it is endowed with a
nondegenerate invariant Hermitian form, and {\em unitary} if in
addition this form is positive definite.

If $G_1$ is a connected real Lie group with Lie algebra
${\mathfrak g}_1$, then $\pi$ is determined by its
differential (still called $\pi$)
$$\pi\colon {\mathfrak g}_1 \rightarrow \End(V),$$
a Lie algebra representation.  The condition for the Hermitian form to
be invariant is equivalent to 
$$\langle \pi(X)v,w\rangle + \langle v,\pi(X)w\rangle = 0 \qquad (v,w\in
V, \ X\in {\mathfrak g}_1);$$
that is, that the real Lie algebra ${\mathfrak g}_1$ acts
by skew-Hermitian operators.
\end{definition}

A Hermitian form on a finite-dimensional vector space $V$ has a {\em
  signature} which for us will be a triple $(p,q,z) \in {\mathbb
  N}^3$: here $p$ is the dimension of a maximal positive-definite
subspace of $V$, $q$ is the dimension of a maximal negative-definite
subspace, and $z$ is the dimension of the radical. Sylvester's law of
inertia says that $p$, $q$, and $z$ are well-defined, and that
\begin{equation}\label{e:sig}
p + q + z = \dim(V).
\end{equation}

\begin{proposition}[Schur's Lemma]\label{prop:schur} Suppose $(\pi,V) \in
  (\widehat{G_1})_{\text{fin}}$ (notation
  \eqref{e:Ghatfin}). Then any two non-zero invariant Hermitian forms
  on $V$ are nondegenerate, and 
  differ by a real nonzero scalar. In particular, the signature
  $(p(\pi),q(\pi))$ is well-defined up to interchanging $p$ and $q$.
\end{proposition}

Here is an outline of the algorithm in \cite{herm} for determining the unitary
irreducible representations of a real reductive group.

\begin{algorithm}\label{alg:unitarity}
Suppose $G_1$ is the group of real points of a complex connected
reductive algebraic group.
\begin{enumerate}
\item {\em List} all the irreducible representations of $G_1$
  admitting a nonzero invariant Hermitian form.
\item For each such irreducible $\pi$, {\em choose} a nonzero invariant
  form $\langle,\rangle_\pi$.
\item For each form $\langle,\rangle_\pi$, {\em calculate} the
  signature $(p(\pi),q(\pi))$.
\item {\em Check} whether one of $p(\pi)$ and $q(\pi)$ is zero; in
  this case, $\pi$ is an irreducible unitary representation.
\end{enumerate}\end{algorithm}

We have explained this algorithm in the case of finite-dimensional
representations.  For infinite-dimensional representations step 1 is
the Langlands classification, and what it means to calculate the
signature of an invariant form on an infinite-dimensional
representation is discussed in \cite{herm}. 

Of these steps, (1) was carried out by Knapp and Zuckerman about 1976;
there is an account in \cite{overview}*{Chapter 16}. Their argument
was a reduction of the problem to the special case of {\em
  representations with real infinitesimal character}.  We will not
recall the precise definition (see \cite{herm}*{Definition 5.5} or
\cite{green}*{Definition 5.4.11}). The nature of the reduction provided
at the same time a reduction of (2)--(4): the entire problem of
understanding unitary irreducible representations was reduced to the
case of real infinitesimal character.  We will therefore concentrate
henceforth on this case. (If $G_1$ is real semisimple, then every
finite-dimensional representation of $G_1$ has real infinitesimal
character; so the reduction is invisible on the level of
finite-dimensional representations.)

\begin{subequations}\label{se:complexify}
Before we look at an example, one more general idea is useful. A
fundamental idea in the representation theory of a real Lie group (or
Lie algebra) is to
{\em complexify} the group (or Lie algebra), and take advantage of the
(simpler and stronger) structural results available for complex Lie
algebras and groups. This is particularly easy for Lie algebras: any
real Lie algebra ${\mathfrak g}_1$ has a natural
complexification
\begin{equation}\label{e:complexifyLiealg}
{\mathfrak g} =_{\text{def}} {\mathfrak g}_1 \otimes_{\mathbb R}
{\mathbb C} = {\mathfrak
  g}_1 \oplus i{\mathfrak g}_1.
\end{equation}
(The distinction between ${\mathfrak g}_1$ and ${\mathfrak g}$ in this
notation seems a little obscure and hard to remember. In the body of
the paper, $G_1$ will usually be something like $G({\mathbb R})$, and
${\mathfrak g}_1$ will be ${\mathfrak g}({\mathbb R})$.)
The extra structure on the complex Lie algebra ${\mathfrak g}$ that
remembers ${\mathfrak g}_1$ is a {\em real form}: a
conjugate-linear real Lie algebra automorphism of order two
\begin{equation}\label{e:complexconj}
\sigma_1\colon {\mathfrak g} \rightarrow {\mathfrak g}, \qquad
\sigma_1(X+iY) = X - iY \quad (X, Y\in {\mathfrak g}_1).
\end{equation}

Any real Lie algebra representation $\pi_{\mathbb R}$ of ${\mathfrak
  g}_1$ on a complex vector space $V$ gives rise to a {\em
  complex} Lie algebra representation 
\begin{equation}\label{e:complexifyrep}
\pi_{\mathbb C}(X + iY) = \pi_{\mathbb R}(X) + i\pi_{\mathbb R}(Y)
\qquad (X,Y \in {\mathfrak g}_1);
\end{equation}
and of course $\pi_{\mathbb R}$ can be recovered from $\pi_{\mathbb
  C}$ by restriction. This is so elementary and fundamental that it
usually goes unsaid, and the subscripts ${\mathbb R}$ and ${\mathbb
  C}$ on $\pi$ are not used.

The reason we make this explicit now is that an invariant Hermitian
form for $\pi_{\mathbb R}$ is almost {\em never} invariant for $\pi_{\mathbb
  C}$; if $\langle \pi_{\mathbb R}(X)v,w\rangle+\langle v,\pi_{\mathbb
  R}(X)w\rangle=0$
$(x\in {\mathfrak g}_1)$, then 
$$
\langle \pi_{\mathbb C}(iX)v,w\rangle+\langle v,\pi_{\mathbb C}(iX)w\rangle=
i(\langle \pi_{\mathbb R}(X)v,w\rangle-\langle v,\pi_{\mathbb R}(X)w\rangle)
$$
and there is no reason for this to be $0$. 
 What is true is that {\em the Hermitian
  form $\langle\cdot,\cdot\rangle$ on $V$ is $\pi_{\mathbb
    R}$-invariant (see \ref{def:invform}) if and only if}
\begin{equation}
\langle \pi_{\mathbb C}(Z)v,w\rangle + \langle v,\pi_{\mathbb
  C}(\sigma_1(Z))w \rangle = 0 \qquad (v,w\in V,\ Z\in {\mathfrak g}).
\end{equation}
That is, we require that {\em $\pi_{\mathbb C}$ should carry the complex
conjugation on ${\mathfrak g}$ to minus Hermitian transpose on
operators}. In this case we call $\langle\cdot,\cdot\rangle$ a {\em
  $\sigma_1$-invariant form} for the representation $\pi_{\mathbb C}$ of
${\mathfrak g}$. 

The point to remember is that the definition of invariant Hermitian form on a
complex representation of a complex Lie algebra ${\mathfrak g}$ {\em
  requires} a choice of real form on ${\mathfrak g}$. Changing
the real form changes everything: whether an invariant form exists,
and what its signature is. 
\end{subequations} 

\begin{example}\label{ex:SL3} Suppose $G_1=SL(3,{\mathbb R})$. The
  finite-dimensional representations of $G_1$ are precisely
  those of the 
  complex Lie algebra ${\mathfrak s}{\mathfrak l}(3,{\mathbb C})$. The
  corresponding real form of ${\mathfrak s}{\mathfrak l}(3,{\mathbb
    C})$ is
$$\sigma_1(Z) = \overline Z \qquad (Z\in {\mathfrak s}{\mathfrak
  l}(3,{\mathbb C})),$$
complex conjugation of matrices.

Irreducible finite-dimensional representations of ${\mathfrak
  s}{\mathfrak l}(3,{\mathbb C})$ are indexed by highest weights
$$\lambda = (\lambda_1,\lambda_2,\lambda_3), \quad
\lambda_1+\lambda_2 + \lambda_3 = 0, \quad \lambda_p - \lambda_q\in
{\mathbb Z},\quad \lambda_1 \ge \lambda_2 \ge \lambda_3.$$
For example, $E_{(2/3,-1/3,-1/3)}$ is the tautological representation on
${\mathbb C}^3$, and $E_{(1,0,-1)}$ is the 8-dimensional adjoint
representation. It turns out that the only representations with a
non-zero invariant $\sigma_1$-invariant Hermitian form are the ``Cartan
powers of the adjoint representation:''
$$E_{(m,0,-m)} = \text{irreducible representation of dimension
  $(m+1)^3$.}$$
(This follows from the Knapp-Zuckerman result explained in
\cite{overview}*{Chapter 16}, but for finite-dimensional
representations is probably much older.)

We would like to understand $\sigma_1$-invariant Hermitian forms on
$E_{(m,0,-m)}$. According
to the program described after Proposition \ref{prop:schur}, we need
first to {\em choose} one of the two 
possible forms. For this (and for much more!) we will use the
restriction of representations of $G_1$ to the maximal
compact subgroup 
$$K_1 = SO(3).$$
Because each irreducible representation of a compact group has a
positive-definite invariant Hermitian form, the positive and negative
parts of an invariant form for $G_1$ may be understood not just as
vector spaces (with dimensions) but as representations of $K_1$ (sums
of irreducible representations with multiplicity). It 
turns out that $E_{(m,0,-m)}$ contains {\em either} the trivial
representation $F_1$ of $K_1$ (if $m$ is even), {\em or}
the tautological three-dimensional representation $F_3$ (if $m$ is
odd), but not both. This representation appears with multiplicity one, so any 
invariant Hermitian form is either positive or negative definite on
the subspace $F_1$ or $F_3$. We fix our choice of $\sigma_1$-invariant
form on $E_{(m,0,-m)}$ by requiring
$$\text{form is {\em positive} on $F_1$ and {\em negative} on $F_3$.}$$
For example, the
adjoint representation $E_{(1,0,-1)} \simeq {\mathfrak g}$ has a Cartan
decomposition (more precisely, the complexification of the Cartan
decomposition of ${\mathfrak g}_1$)
$${\mathfrak g} = {\mathfrak k} \oplus {\mathfrak p} = F_3 \oplus F_5$$
(skew-symmetric and symmetric traceless matrices), the sum of
irreducible representations of $K_1$ of dimensions $3$ and
$5$. We can 
choose for our invariant Hermitian form the trace form
$$\langle X, Y\rangle = \tr(XY^*) \qquad (X, Y \in {\mathfrak
  s}{\mathfrak l}(3,{\mathbb C})).$$
This form is easily seen to be positive definite on the space of real symmetric matrices (since these have real eigenvalues), and
negative definite on the real skew-symmetric matrices (since
these have purely imaginary eigenvalues). In particular, it is
negative on $F_3$. In this way we 
see that the form on $E_{(1,0,-1)}$ has signature $(5,3)$; even better,
the signature is $(F_5,F_3)$ as a representation of $SO(3)$.

Here are a few more signatures. We are for the moment simply claiming
that these formulas are correct, not explaining where they come
from. Always we write $F_{2k+1}$ for the unique irreducible
representation of $SO(3)$ of dimension $2k+1$, endowed with a
positive-definite invariant Hermitian form.
$$\text{signature of $E_{(3,0,-3)}$} = ([F_{13} + F_9 + F_7] +
F_5, [F_{11} + F_9 + F_7] + F_3);$$
that is,
$$\begin{aligned}
\sig(E_{(3,0,-3)}) &= ([F_{13} + F_9 + F_7],[F_{11} + F_9
+ F_7]) + \sig(E_{(1,0,-1)})\\
\sig(E_{(5,0,-5)}) &= ([F_{21} + F_{17} + F_{15} + F_{13}
+ F_{11}],\\
& \qquad [F_{19} + F_{17} + F_{15} + F_{13} + F_{11}]) +
\sig(E_{(3,0,-3)}).\end{aligned}$$
At this point perhaps the pattern is evident: we get the signature for
$E_{(2m+1,0,-2m-1)}$ from that for $E_{(2m-1,0,-2m+1)}$ by adding to the
positive and negative parts sums of $2m+1$ irreducible representations
of $K_1$. The two added strings are identical except for the first
terms, which differ in dimension by two. (The pattern applies even to
getting the signature of $E_{(1,0,-1)}$ from that of the (zero)
representation $E_{(-1,0,1)}$.)  

In the same way, it turns out that
$$\begin{aligned}
\sig(E_{(0,0,0)}) &= (F_1,0)\\
\sig(E_{(2,0,-2)}) &= ([F_9+F_5],[F_7 + F_5]) +
\sig(E_{(0,0,0)})\\
\sig(E_{(4,0,-4)}) &= ([F_{17}+F_{13} + F_{11} +
F_9],\\
& \qquad [F_{15} + F_{13} + F_{11} + F_9]) + \sig(E_{(2,0,-2)}).
\end{aligned}$$
The pattern is essentially the same as in the odd case: we get the
signature for $E_{(2m+2,0,-2m-2)}$ from that for $E_{(2m,0,-2m)}$ by
adding to the 
positive and negative parts sums of $2m+2$ irreducible representations
of $K_1$. The two added strings are identical except for the first
terms, which differ in dimension by two.

As a consequence of this inductive description of the signature as a
representation of $K_1$, or more directly, one can show that 
$$\sig(E_{(m,0,-m)}) = ([(m+1)^3 + (m+1)]/2,[(m+1)^3 - (m+1)]/2).$$
In particular, $E_{(m,0,-m)}$ is unitary if and only if $m=0$: the
trivial representation is the only finite-dimensional unitary
representation of $SL(3,{\mathbb R})$.  (The last statement of course
has many extremely short proofs; the point of explaining this long
argument is that the ideas apply to infinite-dimensional
representations of general real reductive groups.)
\end{example}

This is the shape of the calculation made possible by \cite{herm}: we
find enormous detail about the precise signatures of invariant
Hermitian forms, and then (for the purposes of questions about
unitarity) throw almost all of this information away.  Of course we
would be very happy to learn what interesting questions this discarded
information is actually addressing.

We now describe how the calculation of signatures is related to a more
classical representation-theoretic problem of Clifford theory: how to
extend an irreducible representation of a normal subgroup.

\begin{subequations}\label{se:exnotation}
We do not wish to use all of the somewhat complicated
and delicate hypotheses under which we finally work (involving real
algebraic groups and $L$-groups just as a point of departure). On the
other hand, we would like to use notation that is close to being
consistent with that of the body of the paper.  Here is a
compromise.

A key object to consider will be a group extension
\begin{equation}\label{e:extgroup}
1\rightarrow G \rightarrow {}^{\text{ex}} G \rightarrow \{1,\delta\}
\rightarrow 1,
\end{equation}
which we call an {\em extended group for $G$}. In the body of the
paper, $G$ will very often be a complex connected reductive algebraic
group. Perhaps the most familar example of such an extension, and one
that we will certainly use (often behind the scenes), is
Langlands L-group \eqref{e:Lgroup}; there the role of $G$ is played
by a (complex connected reductive algebraic) dual group, and
$\{1,\delta\}$ is the Galois group of ${\mathbb C}/{\mathbb R}$. 

Here is a concrete way to construct such a group extension. Begin with
an automorphism $\theta \in \Aut(G)$, with the property that
$\theta^2$ is inner:
\begin{equation}
\theta^2 = \Int(g_0) \qquad (g\in G).
\end{equation}
(Only the coset $g_0Z(G)$ is determined by $\theta$.) Then we can define
${}^{\text{ex}}G$ by generators and relations, as the group generated
by $G$ and a single additional element $h_0$, satisfying
\begin{equation}
h_0^2 = g_0,\qquad h_0 g h_0^{-1} = \theta(g) \quad (g\in G).
\end{equation}
(This presentation {\em does} depend on the choice of representative
$g_0$ for the coset $g_0Z(G)$.) Two automorphisms $\theta$ and
$\theta'$ of $G$ are said 
to be {\em inner} to each other if $\theta'\circ\theta^{-1}$ is an
inner automorphism. Now it is clear that
\begin{equation}\label{e:innerclass}
\{\Int(h)|_G \mid h\in {}^{\text{ex}}G - G\} \quad\text{is an inner
  class in $\Aut(G)$.}
\end{equation}
This is how we will use extended groups: as a place to keep track of
and compare representatives of various automorphisms of $G$.

An {\em extended subgroup} of ${}^{\text{ex}}G$ is a subbgroup
${}^{\text{ex}}G_1$ mapping surjectively to $\{1,\delta\}$. In this
case we define $G_1 = G \cap {}^{\text{ex}}G_1$, so that
\begin{equation}\label{e:extsubgroup}
1\rightarrow G_1 \rightarrow {}^{\text{ex}} G_1 \rightarrow \{1,\delta\}
\rightarrow 1.
\end{equation}
Often we will consider several such subgroups ${}^{\text{ex}}G_1$,
${}^{\text{ex}}G_2$, and so on. 

In the body of the paper, the
(complex) group $G$ will be a very useful tool, but we will be
interested actually in representations only of a real form $G({\mathbb
  R})$; then this will be a typical $G_1$. A little more precisely, if
${}^{\text{ex}}G$ is a complex Lie group, then a {\em real form} means
an antiholomorphic automorphism of order two of real
extended Lie groups
\begin{equation}
\sigma_1 \colon {}^{\text{ex}}G \rightarrow {}^{\text{ex}}G, \quad
\sigma_1(G) = G, \qquad {}^{\text{ex}}G_1 =_{\text{def}}
[{}^{\text{ex}}G]^{\sigma_1}. 
\end{equation}
(``Antiholomorphic'' means that if $f$ is a local
holomorphic function on $G$, then $\overline{f\circ \sigma_1}$ is
holomorphic as well.)
\end{subequations} 

Our main results are about the problems of Clifford theory
(Proposition \ref{prop:Clifford}): the relationship between
representations of $G$ and of ${}^{\text{ex}}G$.  Here is a classical
statement of Clifford theory for finite-dimensional representations;
in the world of Harish-Chandra modules, the extension to
infinite-dimensional representations is easy. Write
\begin{equation}\label{e:Ghatfin}
(\widehat G)_{\text{fin}} = \text{equiv. classes of
  fin.-diml. irreducible representations.} 
\end{equation}

\begin{proposition}[Clifford]\label{prop:Clifford} Suppose $G\subset
  {}^{\text{ex}} G$ is an extended group \eqref{e:extgroup}. 
\begin{enumerate}
\item The quotient ${}^{\text{ex}}G/G = \{1,\delta\}$
  acts on $\widehat G_{\text{fin}}$, by
$$\delta\cdot(\pi,E) = (\pi^h,E), \qquad \pi^h(g) =
\pi(hgh^{-1});$$
here $h$ is any fixed element of ${}^{\text{ex}}G - G$ (the
non-identity coset of $G$ in ${}^{\text{ex}}G$).  The equivalence
class of $\pi^h$ is independent of the choice of $h$.

\item Define 
$$\epsilon \colon G \rightarrow \{\pm 1\},\quad  \epsilon (h) = \begin{cases}
  1 & (h\in G) \\ -1 & (h \notin G). \end{cases}.$$
Then the group of characters $\{1,\epsilon\}$ of 
${}^{\text{ex}} G/G$ acts on $(\widehat{{}^{\text{ex}}G})_{\text{fin}}$ by
$$\epsilon\cdot \Pi = \Pi\otimes \epsilon.$$
\item 
Because these are actions of two-element groups, we have 
$$\text{$\pi \in \widehat G_{\text{fin}}$ {\bf either} is fixed by
  $\delta$, {\bf or} has a two-element orbit $\{\pi,\delta\cdot
  \pi\}$.}$$
Similarly, 
$$\text{$\Pi \in (\widehat{{}^{\text{ex}}G})_{\text{fin}}$ {\bf either} has a
  two-element orbit $\{\Pi,\epsilon\cdot \Pi\}$, {\bf or} is fixed by
  $\epsilon$.}$$

\item The two-element orbits of $\delta$ on $\widehat G_{\text{fin}}$ are in
one-to-one correspondence (by induction 
from $G$ to ${}^{\text{ex}} G$) with the $\epsilon$-fixed elements of
$(\widehat{{}^{\text{ex}} G})_{\text{fin}}$.  These are the representations
$\Pi$ of ${}^{\text{ex}} G$ whose characters vanish on
${}^{\text{ex}}G - G$; their characters on $G$ are 
the sum of the two corresponding characters of $G$.

\item The two-element orbits of $\epsilon$ on
$(\widehat{{}^{\text{ex}}G})_{\text{fin}}$ are 
in one-to-one correspondence (by 
restriction to $G$) with the $\delta$-fixed  elements of $\widehat
G_{\text{fin}}$. The two extensions $\Pi$ and 
$\Pi'$ of such a $\pi$ have characters on ${}^{\text{ex}} G-G$ differing by sign;
their characters on $G$ agree with that of $\pi$.

\item Suppose $(\pi,E)$ is a $\delta$-fixed element of $\widehat
  G_{\text{fin}}$. Then an extension of $\pi$ to
  ${}^{\text{ex}}G$ may be constructed as follows. Fix any element
  $h_0\in {}^{\text{ex}}G - G$. Let $A_\pi$ be a nonzero intertwining
  operator from $\pi$ to $\pi^h$:
$$A_\pi \pi(g) = \pi(h_0gh_0^{-1})A_\pi.$$
This requirement determines $A_\pi$ up to a multiplicative
scalar. Write $g_0 = h_0^2 \in G$. 
After modifying $A_\pi$ by an appropriate scalar, we may
arrange
$$A_\pi^2 = \pi(g_0);$$
with this additional condition, $A_\pi$ is determined up to
multiplication by $\pm 1$. Each choice of $A_\pi$ determines an
extension $\Pi$ of $\pi$, by the requirement
$$\Pi(h_0) = A_\pi.$$
\end{enumerate}
\end{proposition}

Suppose we understand the character theory of the smaller group $G$.
Because of this proposition, in order to understand the character
theory of ${}^{\text{ex}}G$, we must understand, for each
$\delta$-fixed irreducible representation of $G$, the character of
{\em some extension} of it on ${}^{\text{ex}}G-G$.  There are always
exactly two such extensions, whose characters on ${}^{\text{ex}}G-G$
differ by sign; our task will be to find a way (for the particular
groups of interest) to specify one of these two extensions.

\begin{subequations}\label{se:disconnected}
Suppose now (as we will for the body of this paper) that $G$ is a
complex connected reductive algebraic group, and that
\begin{equation}\label{e:Gcomplexconj}
\sigma\colon G \rightarrow G
\end{equation}
is a real form: an antiholomorphic automorphism of order two of real
Lie groups. (``Antiholomorphic'' means that if $f$ is a local
holomorphic function on $G$, then $\overline{f\circ \sigma}$ is
holomorphic as well. This implies in particular that the differential
of $\sigma$ (still denoted $\sigma$) is a real form of ${\mathfrak g}$
in the sense of \eqref{e:complexconj}).  The corresponding real
reductive algebraic group is
\begin{equation}\label{e:GR}
G({\mathbb R},\sigma) =  G({\mathbb R}) =_{\text{def}} G^\sigma,
\end{equation}
a real Lie group with Lie algebra
\begin{equation}\label{e:gR}
{\mathfrak g}({\mathbb R}) =_{\text{def}} {\mathfrak g}^\sigma.
\end{equation}
Now $G$ has one particularly interesting (conjugacy class of) real
form(s), the {\em compact real form} $\sigma_c$. It is characterized
up to conjugation by $G$ by the requirement that
\begin{equation}\label{e:Gcpt}
\text{$G({\mathbb R},\sigma_c)$ is compact.}
\end{equation}
Elie Cartan showed that $\sigma_c$ may be chosen to {\em commute} with
the real form $\sigma$, and that this requirement determines
$\sigma_c$ up to conjugation by $G({\mathbb R},\sigma)$. Because of
the commutativity, the composition
\begin{equation}
\theta = \sigma\circ \sigma_c = \sigma_c\circ \sigma
\end{equation}
is an algebraic involution of $G$ of order two, called the {\em Cartan
  involution}; it is determined by $\sigma$ up to conjugation by
$G({\mathbb R},\sigma)$. The group of fixed points
\begin{equation}\label{e:K}
K = G^\theta
\end{equation}
is a (possibly disconnected) complex reductive algebraic subgroup of
$G$. The two real forms $\sigma$ and $\sigma_c$ of $G$ both preserve
$K$, and act the same way there; the corresponding real form
\begin{equation}\label{e:KR}
K({\mathbb R}) = G({\mathbb R},\sigma)\cap K = G({\mathbb
  R},\sigma)\cap G({\mathbb R},\sigma_c) = G({\mathbb
  R},\sigma_c)\cap K 
\end{equation}
is a maximal compact subgroup of $G({\mathbb R},\sigma)$ and a maximal
compact subgroup (the compact real form) of $K$.
\end{subequations} 

We wish to understand $\sigma$-invariant Hermitian forms on
representations of $G({\mathbb R})$. The next proposition recalls the
classical solution (by Cartan and Weyl) of a related problem, and then
relates the two problems.

\begin{proposition}\label{prop:cinvtforms}
Suppose we are in the setting \eqref{se:disconnected}.
\begin{enumerate}
\item Finite-dimensional algebraic representations of $K$ may be
  identified with finite-dimensional continuous representations of
  the compact real form $K({\mathbb R})$.
\item Finite-dimensional algebraic representations of $G$ may be
  identified with finite-dimensional continuous representations of the
  compact real form $G({\mathbb R},\sigma_c)$.
\item Every finite-dimensional irreducible algebraic representation
  $(\pi,E)$ of $G$ admits a positive-definite $\sigma_c$-invariant
  Hermitian form $\langle\cdot,\cdot\rangle_c$, unique up to positive
  scalar multiple:
$$\langle \pi(g)v,w\rangle_c = \langle v,\pi(\sigma_c(g))^{-1}w\rangle_c$$
\item The finite-dimensional irreducible algebraic representation
  $(\pi,E)$ of $G$ admits a $\sigma$-invariant Hermitian form
  $\langle\cdot,\cdot\rangle$ if and only if there is a nonzero linear
  operator, self-adjoint with respect to the form $\langle,\rangle_c$, 
$$A_\pi\colon E \rightarrow E, \qquad A_\pi^* = A_\pi$$
with the property that 
$$A_\pi \pi(g) = \pi(\theta(g))A_\pi \qquad (g\in G).$$
In particular, $A_\pi$ commutes with the action of $K$. These
requirements determine $A_\pi$ up to a real multiplicative
scalar. In this case the $\sigma$-invariant form is
$$\langle v,w\rangle = \langle v,A_\pi w\rangle_c $$

\item In the setting of (4), $A_\pi^2$ must commute with the
  action of $\pi$, and so must be a nonzero scalar. Because
  $A_\pi$ is self-adjoint with respect to the positive Hermitian
  form $\langle,\rangle_c$, the scalar is necessarily positive real:
$$A_\pi^2 = r_\pi I_E, \qquad r_\pi\in {\mathbb R}^{+,\times}.$$

\item 
The signature of this $\sigma$-invariant form is
$$\sig(E) = (\text{$+1$ eigenspace of
  $(r_\pi^{-1/2})A_\pi))$},
\text{$-1$ eigenspace of $(r_\pi^{-1/2})A_\pi$});$$
here the positive and negative parts are representations of $K$.

\end{enumerate}
\end{proposition}

The conditions on $A_\pi$ in Proposition \ref{prop:cinvtforms} look like
conditions in Proposition \ref{prop:Clifford} for defining a
representation of an extended group.  We deduce easily 

\begin{corollary}\label{cor:extgrouprep}
In the setting \eqref{se:disconnected}, suppose also that $G$ is
part of an extended group as in \eqref{se:exnotation}; and that 
$$\theta = \Int(\xi), \quad \xi \in {}^{\text{ex}} G - G, \quad \xi^2 = z\in
Z(G).$$
Then a finite-dimensional irreducible algebraic representation $(\pi,E)$ of
$G$ admits a $\sigma$-invariant Hermitian form if and only if $\pi$
has an extension $\Pi$ to ${}^{\text{ex}} G$. In that case define a
nonzero complex scalar $z_\pi$ so that
$$\pi(z) = z_\pi I_E,$$
and choose a square root $\omega_\pi$ of $z_\pi$. Then the
$\sigma$-invariant Hermitian form on $E$ may be taken to be
$$\langle v,w\rangle = \langle v,\omega_\pi^{-1}\Pi(\xi) w\rangle_c.$$
The signature of this $\sigma$-invariant form is
$$\sig(E) = (\text{$+1$ eigenspace of $\omega_\pi^{-1}\Pi(\xi)$},
\text{$-1$ eigenspace of $\omega_\pi^{-1}\Pi(\xi)$}).$$
In particular, the difference between the dimensions of the positive and negative parts
is equal to $\omega_\pi^{-1}$ times the character value $\tr(\Pi(\xi))$. 
\end{corollary}

There is no difficulty in finding the extended group needed in this
Corollary: one can use for example
\begin{equation}\label{e:simpleextgroup}
{}^{\text{ex}} G = G\rtimes \{1,\xi\},
\end{equation}
with $\xi$ acting on $G$ by the Cartan involution $\theta$. In this
case $z=1$, so the statement of the corollary simplifies a bit. We
allow for more general extended groups because those will turn out to
be useful for the bookkeeping we want to do. 

Corollary \ref{cor:extgrouprep} says that understanding the existence
and signatures of 
$\sigma$-invariant Hermitian forms on algebraic representations of $G$
is equivalent to understanding the algebraic representations of the
disconnected (complex reductive algebraic) group ${}^{\text{ex}} G$. Here is
what happens in the case of $SL(3)$.

\begin{proposition}\label{prop:discSL3}
Suppose $G=SL(3)$, with the real form $G({\mathbb R},\sigma) =
SL(3,{\mathbb R})$ given by complex conjugation of matrices. Then a
compact real form of $SL(3)$ is $SU(3)$, with complex conjugation
given by inverse Hermitian transpose:
$$\sigma_c(g) = {}^t\overline g ^{-1} \qquad (g \in SL(3,{\mathbb
  C})).$$
Then $\sigma_c$ and $\sigma$ commute, so the Cartan involution is
$$\theta(g) = {}^t g^{-1}, \quad K = SO(3,{\mathbb C}), \quad
K({\mathbb R}) = SO(3).$$
Twisting by $\theta$ carries the representation of highest weight
$(\lambda_1,\lambda_2,\lambda_3)$ (see Example \ref{ex:SL3}) to the one
of highest weight $(-\lambda_3,-\lambda_2,-\lambda_1)$. In particular,
the only representations fixed are the various $(\pi_m,E_{(m,0,-m)})$.  

For such a representation, we can therefore find an operator
$$A_m\colon E_{(m,0,-m)} \rightarrow E_{(m,0,-m)}, \quad
A_m\pi_m(g) = \pi_m({}^t g^{-1}) \quad (g\in SL(3)).$$
This requirement specifies $A_m$ up to a scalar; we can specify
it precisely by requiring that $A_m$ act by $+1$ on the unique
largest $SO(3)$ representation $F_{4m+1}$ inside $E_{(m,0,-m)}$. 

With this choice, we can extend $\pi_m$ to a representation $\Pi_m$ of the
disconnected group ${}^{\text{ex}} SL(3)$ of \eqref{e:simpleextgroup},
by defining 
$$\Pi_m(\xi) = A_m.$$

The semisimple conjugacy classes of $SL(3)$ on the non-identity
component of ${}^{\text{ex}} SL(3)$ are represented by 
$$\left\{ h(z) = \begin{pmatrix} 0 & 0 & z\\ 0 & 1 & 0 \\ -z^{-1} & 0 &
    0\end{pmatrix}\xi \right\};$$
here $h(z)$ is conjugate to $h(z^{-1})$. By a version of the Weyl
character formula (for example \cite{Orange}*{Theorem 1.43}), the trace
of this element in the extended representation $\Pi_m$ is
$$\tr(\Pi_m(h(z))) = (-1)^m(z^{2m+2} - z^{-2m-2})/(z^2 - z^{-2})$$
for $z$ not a fourth root of $1$. The element $\xi$ is conjugate to
$h(\pm i)$; so by L'H\^opital's rule,
$$\tr(\Pi_m(\xi)) = \tr(\Pi_m(h(i))) = m+1.$$
The difference between the positive and negative parts of the
signature of the $\sigma$-invariant Hermitian form defined by $\Pi_m$
is therefore $m+1$, so this is the same form described in Example
\ref{ex:SL3}. 
\end{proposition}

Perhaps the most challenging part of proving this proposition is to
verify that $\xi$ is conjugate to $h(\pm i)$, but this can be done.

Kazhdan-Lusztig
theory for computing irreducible characters typically takes place in
a free ${\mathbb Z}[q]$-module with basis indexed by the
irreducible characters of interest. In something like the setting of
Proposition \ref{prop:Clifford} (where we already understand
characters on $G$, and so wish to understand just characters on
${}^{\text{ex}}G - G$) this suggests that we will be 
interested in a free ${\mathbb Z}[q]$-module having a basis
$\{m_\Pi\}$ indexed by {\em one} irreducible representation  $\Pi$ from each pair $\Pi
\ne \Pi' = \Pi\otimes \epsilon$.  In this module, we will think of 
$$m_{\Pi'} = - m_{\Pi}$$
(corresponding to the fact that the characters of $\Pi$ and $\Pi'$ sum
to zero on  ${}^{\text{ex}}G - G$). All of this is explained more
precisely in Section \ref{sec:hecke}. 

The computational problem in implementing the Kazhdan-Lusztig
algorithm is that we know precisely how to parametrize the
$\delta$-fixed irreducibles $\pi$ of the smaller group $G$; but a
$\delta$-fixed irreducible corresponds only to a pair $\{\Pi,\Pi'\}$,
and so only to a basis vector of the Hecke module {\it defined up to
  sign}. We need  
an equally precise parametrization of irreducibles of ${}^{\text{ex}}
G$; that is, of 
how to specify one of the two possible extensions of $\pi$ to
${}^{\text{ex}} G$. In 
Proposition \ref{prop:discSL3} this happened with the requirement that
$A_m$ act by $+1$ on $F_{4m+1}$.  This amounts to a condition
involving the action of a particular element of the larger group
${}^{\text{ex}} G$ on a highest weight vector. 

Corollary \ref{cor:extgrouprep} shows that (in the setting
\eqref{se:disconnected}) understanding $\sigma$-invariant Hermitian
forms on finite-dimensional representations is closely related to
understanding the extensions to ${}^{\text{ex}} G({\mathbb R},\sigma)$
of irreducible representations of $G({\mathbb R},\sigma)$.

An important special case is when 
$\xi$ of \eqref{e:simpleextgroup} acts by an inner
involution of $G$. In this case write $\xi(g)=xgx\inv$ for some $x\in G$. Then the map 
$\xi\rightarrow (x,\epsilon)$ induces an isomorphism
\begin{equation}
\label{e:exequalrank}
{}^{\text{ex}} G = G\rtimes \{1,\xi\}\simeq G\times\Z/2\Z=G\times \{1,\epsilon\}.
\end{equation}
In this, the {\it equal rank case}, there is no essential new
information in the representation theory of ${}^{\text{ex}}G$, and it
is enough to work with $G$ itself.

With the appropriate generalizations, this can be made to to work for
infinite-dimensional representations as well.  This is discussed in
detail in \cite{herm}. Just as in the case of finite-dimensional
representations, it is not necesary to use the extended group in the
case of an equal rank group.
See \cite{herm}*{Section 11}. 

\begin{subequations}\label{se:xiPi}
  In the unequal rank case, this requires (at least implicitly)
  understanding the analogues of highest weights---Lie algebra
  cohomology for maximal nilpotent subalgebras ${\mathfrak n}$---by
  which infinite-dimensional representations $(\pi,E)$ of real
  reductive groups are classified. A little more precisely, one looks
  at the normalizer $G_{\mathfrak n}$ of ${\mathfrak n}$ in $G$. This
  group acts by a character $\chi_\pi$ on a Lie algebra cohomology
  space $H^*({\mathfrak n},E)$, and the character $\chi_\pi$
  determines the representation $\pi$.  To specify an extension
  $(\Pi,E)$ of $\pi$ to ${}^{\text{ex}} G$, one needs an extension
  $\chi_\Pi$ of $\chi_\pi$ to the normalizer ${}^{\text{ex}}
  G_{\mathfrak n}$ of ${\mathfrak n}$ in ${}^{\text{ex}} G$.  To get
  that, we can fix any element
\begin{equation}\label{e:xiPi}
h_{\mathfrak n} \in {}^{\text{ex}} G_{\mathfrak n} - G_{\mathfrak n}.
\end{equation}
Necessarily
\begin{equation}\label{e:hn}
h_{\mathfrak n}^2 = g_{\mathfrak n} \in G_{\mathfrak n}.
\end{equation}
{\it An extension $\Pi$ of of $\pi$ to ${}^{\text{ex}} G$ is specified by specifying
  the single character value $\chi_\Pi(h_{\mathfrak n})$, which may be either
  square root of $\chi_\pi(g_{\mathfrak n})$}:
\begin{equation}\label{e:Pipi}
\text{extension $\Pi$ of $\pi$}\quad \longleftrightarrow\quad \text{square root
  $\chi_\Pi(h_{\mathfrak n})$ of $\chi_\pi(g_{\mathfrak n})$.}
\end{equation}

What makes matters difficult is that the cohomology classes needed for
different representations involve different maximal nilpotent
subalgebras, and (as it turns out) necessarily different elements
$h_{\mathfrak n}$. Even worse, for a single ${\mathfrak n}$, there may be no
preferred choice of $h_{\mathfrak n}$.  We need to have a way to keep
track of choices of these elements $h_{\mathfrak n}$, and of the square
roots $\chi_\Pi(h_{\mathfrak n})$.

A natural way to reduce choices would be to try to arrange for
$h_{\mathfrak n}$ to have order two; in that case $g_{\mathfrak n}=1$, so
$\chi_\pi(g_{\mathfrak n}) = 1$, and the choice $\chi_\Pi(h_{\mathfrak n})$ must
be $\pm 1$. This is more or less what happened in Proposition
\ref{prop:discSL3}, and we were then able to make the ``natural''
choice $\chi_\Pi(h_{\mathfrak n}) = 1$. But in general we
cannot always arrange for $h_{\mathfrak n}$ to have order 2. It turns out
that there is behavior like the example of $G = {\mathbb Z}/4{\mathbb
  Z} = \{\pm 1,\pm i\}$ sitting inside the quaternion group
${}^{\text{ex}} G$ of 
order $8$: every element $\{\pm j, \pm k\}$ of the non-identity coset
has order exactly $4$. Once we are forced to consider a case when
$\chi_\pi(g_{\mathfrak n}) = -1$, it is easy to believe that there can be
no preferred choice of square root.
\end{subequations}

This gives a hint at the difficulties we face. To explain in more
detail their resolution, we begin with the
extension of the Cartan-Weyl highest weight theory to parametrize
representations. This is provided by the Langlands classification,
which is phrased in terms of the complex reductive dual
group. Langlands' results in their original form parametrize not
individual representations but ``L-packets,'' which are collections of
finite sets of irreducible representations for each of several different real
forms of $G$. To use the construction of \eqref{e:simpleextgroup}
would require introducing a different extended group for each of
these different real forms. This is inconvenient at best, and is
inconsistent with the cleanest formulation of the Langlands
classification. 

A glimpse of this inconvenience is the description of conjugacy
classes in the extended group given in Proposition \ref{prop:discSL3}. What
is good about the elements $h(z)$ defined there is that they normalize
the standard Borel subgroup (consisting of upper triangular matrices)
in $SL(3)$; the element $\xi$ does not. It is this good property that
allows one to write a nice Weyl character formula for the elements
$h(z)$. We recall next the notion of {\it pinning} for a
reductive algebraic group, and the derived notion of {\it
  distinguished automorphism}; these are required for the
formulation of the Langlands classification made in Section
\ref{sec:atlasparameters}.

\begin{definition}\label{def:pindist}
Suppose $G$ is a complex connected reductive algebraic group. A {\em
  pinning} of $G$ consists of
\begin{enumerate}
\item a Borel subgroup $B\subset G$;
\item a maximal torus $H\subset B$; and
\item for each simple root $\alpha$, a choice of basis vector
  $X_\alpha \in {\mathfrak g}_\alpha$.
\end{enumerate}
The pair $H\subset B$ is determined by $\{X_\alpha\}$, so we can just
write $(G,\{X_\alpha\})$ for the pinning. 

An algebraic automorphism $\delta_0$ of $G$ is called {\em
  distinguished} (with respect to this pinning) if the differential of
$\delta_0$ permutes the chosen simple root vectors $X_\alpha$. (As a
consequence, $\delta_0$ must preserve $H$ and $B$.)

If $\theta_0$ is a distinguished automorphism of order one or two, we
define the {\em distinguished extended group} to be the algebraic
group ${}^\Gamma G$ generated by $G$ and one more element $\xi_0$,
subject to the relations
$$\xi_0^2 = 1, \qquad \xi_0 g  = \theta_0(g) \xi_0 \quad (\g\in G).$$
\end{definition} 

Recall that two automorphisms $\delta$ and $\delta'$ of $G$ are said to be
{\em inner} to each other if $\delta'\circ \delta^{-1}$ is an inner
automorphism. 

\begin{proposition}[\cite{Spr}*{Corollary 2.14}]\label{prop:pindist}
Suppose $(G,\{X_\alpha\})$ is a complex connected reductive algebraic
group with a pinning. Then any automorphism $\delta$ of $G$ is inner
to a {\em unique} distinguished automorphism $\delta_0$. Necessarily
the order of $\delta_0$ divides the order of $\delta$ (where we make
the conventions that any nonzero natural number divides infinity, and
infinity divides itself). If in addition $\delta$ is semisimple (for
example, if $\delta$ has finite order), then $\delta$ is conjugate by
$G$ to an automorphism $\Ad(h)\delta_0$, for some 
(usually not unique) $h\in H$.
\end{proposition}

In case $\theta_0$ has order one or two, the proposition says that
{\em every automorphism $\theta$ of $G$ inner to $\theta_0$ may be realized by
  the conjugation action of an element $\xi$ of the nonidentity coset
  $G\xi_0$ of the corresponding distinguished extended group}. The
difference from \eqref{e:simpleextgroup} is that, even if
$\theta^2=1$, the element $\xi^2$ may be a {\em nontrivial} element of
$Z(G)$. This turns out to be a small price to pay for having a single
extended group to work with (as $\theta$ varies over an inner class).

The Cartan involution $\theta$ (and therefore the extended group
${}^{\text{ex}} G$) 
is playing a double role in Corollary \ref{cor:extgrouprep}: first,
specifying the real form $G({\mathbb R})$; and second, specifying an
automorphism of $G({\mathbb R})$ by which we wish to twist
representations. It will be convenient to separate these two roles: to
study the twisting of representations of $G({\mathbb R})$ by a second
automorphism $\delta$. 

Section \ref{sec:setting} establishes the
required notation for ``doubly extended groups,'' and recalls also
Langlands' L-group. Section \ref{sec:atlasparameters} recalls from
\cite{algorithms} a formulation of the Langlands classification
well-suited to calculation.  Section \ref{sec:twist} computes the
twisting action of $\delta$ on representations. The idea here (exactly
as in the original work of Knapp and Zuckerman recorded in
\cite{overview}) is that this is a fairly elementary inspection of the
twisting action on {\em parameters} for representations. 

Section \ref{sec:extparam} describes a way to add information to a
$\delta$-fixed 
parameter---essentially choices of elements
$h_{\mathfrak n}$ and $\chi_\Pi(h_{\mathfrak n})$ discussed in
\eqref{se:xiPi}---to specify a representation of the
corresponding extended group ${}^\delta G({\mathbb R})$. Particularly
because the extended group element $h_{\mathfrak n}$ is not unique,
the question of when two of these extended representations are
equivalent is a bit subtle; Section
\ref{sec:equiv} answers this question.

In this way we are able to write explicitly a basis (not just a basis
defined up to sign) for the Hecke module considered in
\cite{LVq}. Precise formulas for the action of Hecke algebra
generators on the basis are written in Section \ref{sec:hecke}. Each
such formula involves one to four basis vectors in the module. In
\cite{LVq} it was shown that these one to four basis vectors could be
chosen so that the action of the generator was given by a specified
matrix (of size one to four). A typical example (the {\em only}
example in the original paper \cite{KL}) is
\begin{equation}\label{e:Ts}
\begin{pmatrix} 0 & q \\ 1 & q-1 \end{pmatrix}
\end{equation}

The technical problem that led to this paper is that {\it these nice
  choices of basis vectors cannot be made consistently as the Hecke
  algebra generator varies}.  The result is that if we fix a single
choice of basis for the Hecke module, then the actions of some of the
Hecke algebra generators will be given by matrices made of blocks not
only like \eqref{e:Ts}, but also by conjugates of such a matrix by a
diagonal matrix with entries $\pm 1$. A typical example is
\begin{equation}\label{e:Tssign}
\begin{pmatrix} 0 & -q \\ -1 & q-1 \end{pmatrix}
\end{equation}
The point of the formulas in Section \ref{sec:hecke} is to say
precisely where the minus signs must go. In order to do this, one
needs to say how to manipulate our extended parameters to get the nice
basis vectors discussed in \cite{LVq}.  There are two cases where this
manipulation is somewhat more complicated, and they are described in
detail in Sections \ref{sec:2i12} and \ref{sec:2Ci}.
Ultimately this gives an explicit algorithm for computing the polynomials 
of \cite{LVq}, which is being implemented in the {\tt atlas of Lie groups and Representations} software \cite{atlas}.
The application to our computation of Hermitian forms is 
\cite{herm}*{Theorem 19.4}. 

A guiding principle in formulating these results is the fundamental
duality theorem originating in \cite{KL}*{Theorem 3.1}, and extended
to Harish-Chandra modules in \cite{IC4}. Section  \ref{sec:dual}
describes how to prove for this for the Hecke modules in the twisted setting.
The heart of the
proof in every case is that a ``transpose'' of one Hecke algebra
action is equal to another Hecke algebra action; explicitly, that the
transpose of the matrix giving an action of a generator is equal to
the matrix giving the action of the same generator on a different
module. That such a statement is true up to signs was clear from
\cite{LVq}; with the specification of the signs in this paper we are
able to prove it completely.

\section{Setting}\label{sec:setting}
\setcounter{equation}{0}

Our first goal is to understand which representations are fixed by a
given outer automorphism, and how to 
to write down the corresponding representations of the extended group.
We begin by setting up some notation in this section, discuss the {\tt
  atlas} parametrization of representations in Section
\ref{sec:atlasparameters}, and the action of twisting on these
parameters in Section \ref{sec:twist}.

We start  with a connected complex reductive algebraic group $G$,
equipped with a pinning (Definition \ref{def:pindist}). 
\begin{subequations}\label{se:notation}
Acting on this we have two commuting distinguished involutive automorphisms
\begin{equation}\label{e:auts}
\xi_0\colon (G,B,H) \rightarrow (G,B,H), \qquad \delta_0\colon (G,B,H)
\rightarrow (G,B,H), 
\end{equation}
satisfying
\begin{equation}\label{e:autreq}
\xi_0(X_\alpha) = X_{\xi_0(\alpha)},\quad \delta_0(X_\alpha) =
X_{\delta_0(\alpha)}  \qquad (\alpha,\ \xi(\alpha),\ \delta_0(\alpha) \in \Pi).
\end{equation}
See Definition \ref{def:pindist} and
\cite{ABV}*{p. 34 or p. 51}. 

The automorphism $\xi_0$ defines the inner class of real forms under
consideration; it is the unique Cartan involution in the inner class
which is distinguished, and is the Cartan involution of the  ``most
compact''  real form in the inner class. The automorphism $\delta_0$ defines the
twisting of representations that we will consider. Since any
automorphism is inner to a distinguished one there is no loss in
assuming $\delta_0$ is distinguished. 

We will abuse notation and use these automorphisms to define a
semidirect product of $G$ with the Klein 4-group $({\mathbb
  Z}/2{\mathbb Z})^2$:
\begin{equation}\label{e:doublyextended}
{}^\Delta G = G\rtimes \{1,\xi_0,\delta_0,\xi_0\delta_0\}.
\end{equation}
The superscript
$\Delta$ is supposed to suggest ``double.''  The abuse of notation is
that from now on $\xi_0$ may denote an element of ${}^\Delta G$ (which
by definition is never the identity) {\em or} an automorphism of $G$
(which is the identity exactly when $\xi_0$ defines the equal rank
inner class).

It is helpful to use also the corresponding {\em large}
(\cite{ABV}*{p. 51}) involutive automorphism. As in  \cite{ABV}  we write
\begin{equation}\label{e:exp}
e\colon {\mathfrak h} \rightarrow H, \qquad e(X) = \exp(2\pi i X); 
\end{equation}
this is a surjective homomorphism from the Lie algebra onto $H$, with
kernel equal to $X_*(H)$. Also we write
\begin{equation}
\rho = \frac{1}{2} \sum_{\beta \in R^+(G,H)} \beta, \quad \rho^\vee = \frac{1}{2} \sum_{\beta \in R^+(G,H)} \beta^\vee.
\end{equation} 
Then $\alpha(e(\rho^\vee/2)) = -1$ for every simple root $\alpha$; so
if we define 
\begin{equation}\label{e:large}
\xi_1 = e(\rho^\vee/2)\xi_0 \in H\xi_0, 
\end{equation} 
then this element of ${}^\Delta G$ acts on $G$ as an involutive
automorphism satisfying
\begin{equation}
\xi_1|_H = \xi_0|H, \quad \xi_1(X_\alpha) = -X_{\xi(\alpha)}, 
\qquad (\alpha \in \Pi).
\end{equation}
This element satisfies
\begin{equation}\label{e:xi1squared}
\xi_1^2 =  e(\rho^\vee) =_{\text{def}} z(\rho^\vee) \in Z(G),
\end{equation}
a central element of order (one or) two.
\end{subequations} 

\begin{subequations}\label{se:Lgroup}
Our torus $H\subset G$ has a well-defined (that is, uniquely defined
up to unique isomorphism) dual torus
\begin{equation}\label{e:dualtorus}
{}^\vee H = X^*(H) \otimes_{\mathbb Z} {\mathbb C}^\times.
\end{equation}
The characters and cocharacters of ${}^\vee H$ are naturally
identified with the
cocharacters and characters of $H$:
\begin{equation}\label{e:duallattices}
X^*({}^\vee H) \simeq X_*(H), \qquad X_*({}^\vee H) \simeq X^*(H).
\end{equation}
The isomorphisms here are canonical, and respect the pairings into
${\mathbb Z}$. 

The automorphisms $\xi_0$ and $\delta_0$ of $H$ (cf. \eqref{se:notation}) 
define automorphisms ${}^t\xi_0$ and ${}^t\delta_0$ of $X^*(H)$, and
therefore
\begin{equation}\label{e:dualHauts}
{}^\vee\xi_0 =_{\text{def}} -w_0{}^t\xi_0, \qquad {}^\vee\delta_0
=_{\text{def}} {}^t\delta_0
\end{equation}
of $X^*(H)$ and of ${}^\vee H$. Here we write 
\begin{equation}\label{e:long}
w_0 \in W(G,H) \simeq W({}^\vee G,{}^\vee H)
\end{equation}
for the unique longest element, which carries $R^+(G,H)$ to
$-R^+(G,H)$. {\em Notice the presence of a minus
sign in the definition of ${}^\vee \xi_0$ (partly ``corrected'' by the
factor of $w_0$) and its absence in the
definition of ${}^\vee \delta_0$.} This is the way things are. One way
to understand it is that $\xi$ is related to the Cartan involution for
$G$, which is less fundamental and natural than the Galois action for a
real form. The Cartan involution acts on the root datum (with respect
to a real $\theta$-stable Cartan) by the {\em negative} of the Galois
action on the root datum; and it is this minus sign which accounts for
the minus sign in \eqref{e:dualHauts}.

Now we construct a dual group ${}^\vee G \supset {}^\vee H$, whose
root datum is dual to that of $G$:
\begin{equation}\label{e:veeG}
{}^\vee G \supset {}^\vee B ={}^\vee H {}^\vee N, \quad R^+({}^\vee
G,{}^\vee H) = \{\beta^\vee \mid \beta \in R^+(G,H)\}.
\end{equation}
We choose also a pinning: nonzero root vectors
\begin{equation}\label{e:veeGpin}
\{X_{\alpha^\vee} \mid \alpha^\vee \in \Pi^\vee\} \subset
{}^\vee{\mathfrak n}.
\end{equation}
Such a choice of dual group and pinning is unique up to unique
isomorphism. Because the automorphisms ${}^\vee\xi_0$  and $^\vee
 \delta_0$ 
respect the based root datum, they extend uniquely to (distinguished)
automorphisms
\begin{equation}\label{e:dualGauts}\begin{aligned}
{}^\vee\xi_0\colon ({}^\vee G,{}^\vee G,{}^\vee H)
&\rightarrow ({}^\vee G,{}^\vee B,{}^\vee H), \quad
{}^\vee\xi_0(X_{\alpha^\vee}) = X_{-w_0\xi_0(\alpha)^\vee}\\
{}^\vee\delta_0\colon ({}^\vee G,{}^\vee B,{}^\vee H)
&\rightarrow ({}^\vee G,{}^\vee B,{}^\vee H), \quad
{}^\vee\delta_0(X_{\alpha^\vee}) = X_{\delta_0(\alpha)^\vee},
\end{aligned}\end{equation}
Automatically ${}^\vee\delta_0$ and ${}^\vee\xi_0$ commute. By
definition the {\em L-group of $G$} is the semidirect product
\begin{equation}\label{e:Lgroup}
{}^L G = {}^\vee G \rtimes \{1,{}^\vee\xi_0\}. 
\end{equation}
(A little more
precisely, it is this group endowed with the ${}^\vee G$-conjugacy
class of $({}^\vee B,\{X_{\alpha^\vee}\},{}^\vee\xi_0)$.)

Just as for $G$, it is convenient to have in hand also the {\em large}
representative
\begin{equation}\label{e:duallarge}
{}^\vee \xi_1 = e(\rho/2)\xi_0, \qquad
{}^\vee\xi_1(X_{\alpha^\vee}) = -X_{-w_0\xi_0(\alpha)^\vee}. 
\end{equation}
Again this element satisfies
\begin{equation}\label{e:dualz}
{}^\vee\xi_1^2 = e(\rho) =_{\text{def}} z(\rho) \in Z({}^\vee G),
\end{equation}
a central element of order (one or) two.

We say a little more about the identification of Weyl groups in
\eqref{e:long}. Define
\begin{equation}\label{e:WGH}\begin{aligned}
s_\alpha&\in \Aut(X_*(H)),  \quad s_\alpha(t)= t -
\langle\alpha,t\rangle\alpha^\vee\\
W(G,H) &= \langle s_\alpha\mid \alpha\in \Pi\rangle \subset \Aut(X_*(H)).
\end{aligned}
\end{equation}
Then the identification 
$$ \Aut(X_*(H)) \supset W(G,H)\simeq W({}^\vee G,{}^\vee H) \subset
\Aut(X^*(H))$$ 
is given by
\begin{equation}
s_\alpha \mapsto s_{\alpha^\vee},\qquad w\mapsto {}^t w^{-1}.
\end{equation}
\end{subequations}

\section{Atlas  Parameters}
\label{sec:atlasparameters}

The basic reference for this section is \cite{algorithms}.

\begin{subequations}
\label{se:strongreal}
As explained after Proposition \ref{prop:pindist}, we are going to
represent involutive automorphisms of $G$ (briefly, 
{\em involutions}) by the conjugation action of elements of
$G\xi_0$. For this purpose we introduce the set of {\em strong
  involutions}:
\begin{equation}\label{e:calI}
{\mathcal I} = \{\xi \in G\xi_0 \mid \xi^2 \in Z(G)\}.
\end{equation}
If $\xi\in\mathcal I$ then 
\begin{equation}\label{e:thetaK}
\theta_\xi=\int(\xi), \qquad 
K_\xi=G^{\theta_\xi} = \Cent_G(\xi).
\end{equation}
is an involutive automorphism of $G$, in the inner class of $\xi_0$;
and every such involutive automorphism arises this way. We need to
allow $\xi^2 \in Z(G)$ (and not merely $\xi^2=1$) because not every
involution in the inner class of $\xi_0$ arises from an element $\xi$
of order $2$. 
(But we {\em can} easily arrange for $\xi$ to have order a power of $2$.)  The
central element
\begin{equation}\label{e:centralcocharxi}
z = \xi^2 \in Z(G)
\end{equation}
is called the {\em central cocharacter} of the strong involution $\xi$.

A {\em strong real form} of $G$ is a $G$-conjugacy class $\mathcal
C\subset \mathcal I$. The central cocharacter is constant on
${\mathcal C}$, so we may write it as
\begin{equation}\label{e:centralcocharC}
z({\mathcal C}) = \xi^2 \in Z(G) \qquad (\xi \in {\mathcal C}).
\end{equation}
The various involutions $\{\theta_\xi\mid \xi\in {\mathcal C}\}$ form
a single $G$-conjugacy class of involutive automorphisms of $G$, so
the subgroups $\{K_\xi \mid \xi\in {\mathcal C}\}$ are a single
$G$-conjugacy class as well. If $G$ is adjoint, then these three
$G$-conjugacy classes (strong involutions, involutions, and fixed
point subgroups) are identified by the natural maps
$$\xi \rightarrow \theta_\xi \rightarrow K_\xi.$$
If $G$ is not adjoint, however, the first of these maps need not be
one-to-one: choosing a strong involution is more restrictive than
choosing an involution.
\end{subequations} 

Here is the reason that strong involutions and strong real forms are useful.
\begin{proposition}\label{prop:switchstrong} Suppose $\xi$ and $\xi'$
  are strong involutions in the same strong real form---that is,
  conjugate by $G$ (\eqref{se:strongreal}). Then there is a {\em canonical}
  bijection from equivalence classes of irreducible $({\mathfrak
    g},K_\xi)$-modules to equivalence classes of irreducible
  $({\mathfrak g},K_{\xi'})$-modules.
\end{proposition}
\begin{proof}
Suppose $g\in G$ conjugates $\xi$ to $\xi'$. Then twisting by $g$
carries $({\mathfrak g},K_\xi)$-modules to $({\mathfrak
  g},K_{\xi'})$-modules. So far this would have worked using just
involutive automorphisms $\theta$ and $\theta'$. What is special about
strong involutions is that {\em the stabilizer of $\xi$ in $G$ is
  precisely $K_\xi$} (whereas the stabilizer of $\theta_\xi$ can be bigger). 
 This means that the coset $gK_\xi$ is uniquely
determined. Because twisting by $K_\xi$ acts trivially on equivalence
classes of $({\mathfrak g},K_\xi)$-modules, it follows that the
bijection we have defined is unique. 
\end{proof}

Using these unique bijections, one can make a well-defined set of
equivalence classes of irreducible modules attached to each strong
real form ${\mathcal C}$.  These equivalence classes are what we will
study.

In classical representation theory, one fixes once and for all a
Cartan involution $\theta$ of $G$, defining a single symmetric
subgroup $K = G^\theta$. The theory of $({\mathfrak g},K)$-modules
proceeds by defining and studying (for example)
various maximal tori preserved by $\theta$. A central idea in the
{\tt atlas} algorithms is instead to fix the maximal torus $H \subset
G$, and to study various Cartan involutions preserving it. There are
hints of this idea in the classical theory. For example, it is common
in introductory texts to describe the principal series
representations of $SL(2,{\mathbb R})$, because these are closely
related to the standard (diagonal) split maximal torus. When
discussing the discrete series, it is common to consider instead the
(isomorphic) real group $SU(1,1)$, because the discrete series are
closely related to the standard (diagonal) compact maximal torus of
$SU(1,1)$.

\begin{subequations}\label{se:X}
In order to pursue this idea, we need to single out the strong
involutions preserving our fixed $H$. These are
\begin{equation}
\label{e:X}
\begin{aligned}
\wt \X&=\mathcal I\cap\Norm_{G\xi_0}(H)=\{\xi\in \Norm_{G\xi_0}(H)\mid
\xi^2\in Z(G)\}\\ 
\X&=\wt\X/H\quad\text{(quotient by conjugation action of $H$)}\\
\end{aligned}
\end{equation}
If $z \in Z(G)$, we write
\begin{equation}\label{e:Xz}
\begin{aligned}
\wt \X_z &=\{\xi\in \Norm_{G\xi_0}(H)\mid \xi^2 = z\}\\ 
\X_z&=\wt\X_z/H\quad\text{(quotient by conjugation action of $H$)}\\
\end{aligned}
\end{equation}
for the subset of elements of central cocharacter $z$.

Write $p:\wt\X\rightarrow\X$ for the projection map.

For $x\in\X$ let $\theta_x$ be the restriction of $\theta_\xi$ to $H$
for any $\xi\in p\inv(x)$.  The central technical difficulty we face
is that {\em the involution $\theta_x$ of $H$ only depends on $x$, but
  the extension $\theta_\xi$ to $G$ depends on the choice of
  representative $\xi$.} 

It is easy to check that
\begin{equation}
\label{e:wx}
\theta_x = w_x\xi_0 \in \Aut(H) \quad(w_x\in W(G,H))
\end{equation}
for some {\em twisted involution} $w_x$ with respect to $\xi_0$:
\begin{equation}\label{e:twistedinv}
w_x\xi_0(w_x) = 1.
\end{equation}
Conversely, if $w\in W$ is any twisted involution with
respect to $\xi_0$, then 
\begin{equation}\label{e:thetaw}
\theta_w =_{\text{def}} w\xi_0 \in \Aut(H)
\end{equation}
is an involutive automorphism of $H$ (or, equivalently, of
$X_*(H)$). We define
\begin{equation}\label{e:Xw}
\X^w = \{x\in \X \mid w_x = w\}, \qquad \wt\X^w = p^{-1}\X^w,
\end{equation}
so that $\X$ is the disjoint union over twisted involutions $w$ of the
various $\X^w$.

The definition \eqref{e:wx} of $w_x$ can be restated as
\begin{equation}\label{e:torusxi}
\xi = s_1\sigma_{w_x}\xi_0 \qquad \text{(some $s_1\in H$).}
\end{equation}
Here $\sigma_{w_x}$ is the Tits group representative of $w_x$ (see
\eqref{e:distrep}).  
We call $s_1$ the {\em unnormalized torus part} of $\xi$.
We compute
\begin{equation}\label{e:xisquared}
\begin{aligned}
\xi^2 &= s_1\sigma_{w_x} \xi_0 s_1\sigma_{w_x} \xi_0 \\
&= s_1\theta_{w_x}(s_1) \sigma_{w_x}\sigma_{\xi_0(w_x)} \\
&= s_1\theta_{w_x}(s_1) \sigma_{w_x}\sigma_{w_x^{-1}}\\
&= s_1\theta_{w_x}(s_1) e((\rho^\vee - \theta_x\rho^\vee)/2)\quad\text{(by
  Proposition \ref{p:bicycle})}\\
&=(s_1e(-\rho^\vee/2))\theta_{w_x}(s_1e(-\rho^\vee/2)) e(\rho^\vee).
\end{aligned}
\end{equation}
We call $s=s_1e(-\rho^\vee/2)$ the {\em normalized torus part} of $\xi$:
\begin{equation}\label{e:normtorusxi}\begin{aligned}
\xi &= se(\rho^\vee/2)\sigma_{w_x}\xi_0 = s\xi_w\qquad \text{(some $s\in H$).} \\
\xi^2 &= s\theta_{w_x}(s) z(\rho^\vee).
\end{aligned}\end{equation}
Here we have used the definition of $\xi_w$ in the following proposition.
\end{subequations} 

\begin{proposition}\label{prop:basestrong}
For every $\xi_0$-twisted involution $w\in W(G,H)$ there is a
basepoint (the one with trivial {\em normalized} torus part)
$$
\xi_w =_{\text{def}} e(\rho^\vee/2)\sigma_{w} \xi_0\in\wt\X$$
of central cocharacter $z(\rho^\vee)$ (see \eqref{e:centralcocharxi}):
$$\xi_w^2 = e(\rho^\vee)= z(\rho^\vee).$$
This basepoint is conjugate by $G$ to the large representative $\xi_1$
of \eqref{e:large}.
\end{proposition}
\begin{proof}
The formula for $\xi_w^2$ is immediate from \eqref{e:xisquared}. We
omit the argument that $\xi_w$ is conjugate to $\xi_1$. \end{proof}

Fix a  set $S$ of representatives of the set of strong real forms:
\begin{equation}
\label{e:S}
\wt \X
\supset S\ \overset{1-1}\longleftrightarrow\  \mathcal I/G.
\end{equation}

\begin{proposition}[\cite{algorithms}*{Corollary 9.9}] \label{prop:KGB}
There is a canonical bijection
$$
\X\longleftrightarrow \prod_{\xi'\in S}  K_{\xi'}\backslash G/B
$$
The bijection restricts to 
classes on both sides of any fixed central cocharacter (see
\eqref{e:centralcocharxi}), in which case both sides are finite sets.
\end{proposition}
Because of this proposition, we refer to $\X$ as the {\tt KGB}-space,
and say  $x\in \X$ is a {\tt KGB}-class. 

\begin{subequations}\label{se:torusxi}
The {\tt KGB} classes are parametrized first by a twisted involution 
$w\in W$ (see \eqref{e:wx}), and then (for each $w$) by the allowed
(twisted $H$-conjugacy classes of) normalized torus parts. Our next task is to
describe those torus parts. It is convenient to fix also a central
element $z\in Z(G)$, and to restrict attention to strong involutions
of central cocharacter $z$. According to \eqref{e:normtorusxi}, 
we are therefore seeking to solve the equation
\begin{equation}\label{e:toruspartA}
s\theta_w(s) = zz(-\rho^\vee) \qquad (\xi =
s\xi_w).
\end{equation}
Conjugation by $h\in H$ replaces the torus part $s$ by
$$s(h\theta_w(h)^{-1}),$$
so the solutions we want---elements of the {\tt KGB} space $\X$---are
cosets of the connected torus 
\begin{equation}
(1-\theta_w)H = \text{identity component of\ } H^{-\theta_w} =
  H^{-\theta_w}_0.
\end{equation} 
In order to keep track of such elements, we would like to have nice
representatives for the cosets $H/(1-\theta_w)H$. Because the Lie
algebra is the direct sum of the $+1$ and $-1$ eigenspaces of
$\theta_w$, we get 
\begin{equation}\label{e:thetacosets} 
H = [H^{\theta_w}_0][H^{-\theta_w}_0], \qquad
 [H^{\theta_w}_0]\cap
[H^{-\theta_w}_0] \subset [H^{\theta_w}_0](2).
\end{equation}

This says that every coset of $H^{-\theta_w}_0$ has a representative
in $H^{\theta_w}_0$; and that this representative is
unique up to multiplication by the finite 2-group $[H^{\theta_w}_0]
\cap [H^{-\theta_w}_0]$.  We call a coset representative in
$H^{\theta_w}_0$ {\em
  preferred}. Our immediate goal is therefore to write down all
solutions $s\in H^{\theta_w}_0$ of \eqref{e:toruspartA}.

As with many calculations in Lie theory, solving this equation
is easier on the Lie algebra. We will use the exponential map
isomorphisms 
\begin{equation}
e\colon {\mathfrak h}/X_*(H) \rightarrow H, \qquad e\colon {\mathfrak
  h}^{\theta_w}/X_*(H)^{\theta_w}  \rightarrow H^{\theta_w}_0
\end{equation}
of \eqref{e:exp}. In order to do that, we first choose a logarithm $g$
of the central cocharacter $z$:
\begin{equation}\label{e:infcochar}
z = e(g) \qquad (g\in {\mathfrak h} = X_*(H)\otimes_{\mathbb Z}
{\mathbb C}). 
\end{equation}
We say that a strong real form of central cocharacter $z$ has {\em
  infinitesimal cocharacter $g$}. It is convenient (and easy) to
arrange also
\begin{equation}
\label{e:domregdual}
 \langle\alpha,g \rangle\in\Z_{>0}\quad(\alpha\in R^+(G,H)).
\end{equation}
(Because $z$ is assumed central, roots take integer values on $g$.)

Next, we choose a logarithm $v$ for
the normalized torus part $s$:
\begin{equation}\label{e:logt}
s = e(v) \qquad (v\in {\mathfrak h}^{\theta_w}). 
\end{equation}
Now \eqref{e:toruspartA} can be written
\begin{equation}\label{e:toruspartB}
2v = v + \theta_w(v) = g -\rho^\vee - \ell \qquad 
(\text{some $\ell \in X_*(H)$}), 
\end{equation}
or
\begin{equation}\label{e:specialtoruspart}
v = (g - \rho^\vee - \ell)/2.
\end{equation}
Conversely, if $\ell\in X_*(H)$ has the property that
\begin{equation}
g - \rho^\vee + \ell \in {\mathfrak h}^{\theta_w},
\end{equation}
then $e((g-\rho^\vee - \ell)/2)$ is a preferred representative for a
normalized torus part (of some $\xi\in \wt\X$ of central cocharacter
$z$).
\end{subequations} 

We have proven the following proposition.
\begin{proposition}\label{p:rootdatumx} Fix an
infinitesimal   cocharacter $g$  
and a  $\xi_0$-twisted involution $w$. Let 
  $\theta_w = w\circ \xi_0 \in \Aut(H)$. 
The set $\X^w_g$ of {\tt KGB} classes of infinitesimal cocharacter $g$
(equivalently, of central cocharacter $z= e(g)$) with $w_x = w$
(cf.~\eqref{e:wx}) is in one-to-one correspondence with  
$$
\left\{\overline \ell\in X_*(H)/(1+\theta_w)X_*(H) \mid
  (1-\theta_w)\ell = (1-\theta_w)(g - \rho^\vee)|\right\}.
$$
This set is either empty (if $(1-\theta_w)(g-\rho^\vee)$ does not
belong to $(1-\theta_w)X_*$), or has a simply transitive action of  
$$X_*^{\theta_w}/(1+\theta_w)X_*.$$ 
This latter group is a vector space over ${\mathbb Z}/2{\mathbb Z}$,
of dimension at most the rank of $X_*$.

The corresponding $x$ has a preferred representative
(cf.~\eqref{e:thetacosets}) $\xi$ with unnormalized torus
part
$$
s_1 = e((g-\ell)/2),
$$
(see \eqref{e:torusxi}) or normalized torus part 
$$s = e((g-\rho^\vee - \ell)/2)$$
(see \eqref{e:normtorusxi}). Here $\ell\in X_*(H)$ is a representative
of $\overline\ell$. 
If we modify the element $\ell$ in its coset by adding
$(1+\theta_w)f$ (for some $f\in X_*(H)$), then
$s$ (or $s_1$) is multiplied by $e((1+\theta_w)f/2)$. That is, this
preferred choice of torus part is unique up to the image of
$1+\theta_w$ acting on the elements $H(2)$ of order 2 in $H$.  Another
formulation is that these preferred representatives $\xi$ of $x$ are a
single conjugacy class under $H(2)$.
\end{proposition}
The {\tt KGB} classes in this proposition usually represent several
{\em different} strong real forms (all of a fixed central
cocharacter); that is, they are usually {\em not} conjugate by $G$.
The parametrization of {\tt KGB} classes is so beautiful and
simple  precisely because of this inclusion of several real
forms. For example, if $G=GL(n)$, $\xi_0=1$, and $w=1$ (so that we are
talking about compact maximal tori in equal rank real forms), then the
{\tt KGB} classes amount to discrete series for strong real forms. If we
choose $g=\rho^\vee$, then the Proposition says that the {\tt KGB} classes
are indexed by $X_*(H)/2X_*(H)$, an $n$-dimensional vector space over
${\mathbb Z}/2{\mathbb Z}$. There are $n+1$ different strong real
forms appearing in this list: the various $U(p,n-p)$ with $0\le p \le
n$. Such a strong real form has $\binom{n}{p}$ discrete series; only
when we take the union over $p$ do we get something as simple as $2^n$.

We turn now to writing down Langlands parameters for representations
of real forms of $G$, in the form described in \cite{ABV}. These are
constructed in a manner roughly parallel to the strong involutions
above, but in the L-group of \eqref{e:Lgroup} rather than in the
extended group for $G$.

\begin{definition}\label{def:Lparam} A {\em Langlands parameter} for
  representations of real forms of $G$ is a pair $(\ch\xi, \gamma)$ such
  that
\begin{enumerate}[label=(\alph*)]
\item $\ch\xi\in \ch G\ch \xi_0$;
\item $\gamma \in \ch{\mathfrak g}$ is semisimple; and
\item $\ch\xi^2 = e(\gamma)$.
\end{enumerate}
Two Langlands parameters are called {\em equivalent} if they are
conjugate by $\ch G$. The semisimple group element $\ch z = \ch\xi^2 \in \ch G$
is called the {\em central character} of the Langlands parameter, and
the Lie algebra element $\gamma$ is called the {\em infinitesimal
  character}. 

Because Langlands parameters matter only up to conjugation by $\ch G$,
it is convenient to consider representatives aligned with our fixed
$\ch H \subset \ch B$.  The Langlands parameter is said to be of {\em
type $\ch H$} if   
$$\ch\xi \in \Norm_{\ch G\ch \xi_0}(\ch H),\quad \text{and}\quad
\gamma\in \ch{\mathfrak h}.$$
Finally, a Langlands parameter of type $\ch H$ is said to be {\em
integrally dominant} if it is dominant for the integral root system:
\begin{equation}
\label{e:id}
\langle\gamma,\ch\alpha\rangle\in\Z\implies \langle\gamma,\ch\alpha\rangle\ge 0 \quad(\alpha\in R^+(G,H)).
\end{equation}

\end{definition}

Harish-Chandra's theorem guarantees that representation-theoretic
infinitesimal characters---homomorphisms from the center of
$U({\mathfrak g})$ to ${\mathbb C}$---are in one-to-one correspondence
with $\ch
G$ orbits of semisimple elements in $\ch{\mathfrak g}$. The
infinitesimal character defined here of the Langlands parameter will
turn out to correspond exactly to the representation-theoretic
infinitesimal characters of the corresponding representations of real
forms of $G$. Unfortunately the central character defined here bears
no such simple relationship to the representation-theoretic central
characters.

Here is the original statement of the Langlands classification (with
the notion of Langlands parameter modified in accordance with
\cite{ABV}).

\begin{theorem}[\cite{LC}*{Proposition 4.1}]\label{thm:Langlands} In
  the setting of Definition \ref{def:Lparam}, fix a strong real form
  $\xi$ of $G$. Attached to each equivalence class of Langlands
  parameters $(\ch\xi,\gamma)$ for $G$ there is a finite set 
  $\Pi_{\ch\xi,\gamma}(\xi)$ of equivalence classes of irreducible 
  $({\mathfrak g},K_\xi)$-modules of infinitesimal character
  $\gamma$. These finite sets partition the full set of equivalence
  classes of such representations. 
\end{theorem} 

Langlands called the finite sets $\Pi_{\ch\xi,\gamma}(\xi)$ {\em
  L-packets}, because of their role in automorphic representation
theory.

Because of this theorem, we want to understand in more detail what
Langlands parameters can look like; and for a fixed Langlands
parameter, we want to understand the structure of the L-packet
$\Pi_{\ch\xi,\gamma}$.  

\begin{proposition} \label{p:Lequiv}
  Any Langlands parameter is equivalent to an integrally dominant one
  of type $\ch H$. If the infinitesimal character $\gamma\in
  \ch{\mathfrak h}$ is regular, then two Langlands parameters of type
  $\ch H$ and infinitesimal character $\gamma$ are equivalent (that
  is, conjugate by $\ch G$) if and only if they are conjugate by $\ch
  H$. In other words, a collection of all equivalent Langlands
  parameters of type $\ch H$ and infinitesimal character $\gamma$ is a
  single $\ch H$-conjugacy class.  
\end{proposition}

This is an elementary consequence of the definition, and we omit the proof.

Here is some structure theory for Langlands parameters analogous to
that given for strong involutions in \eqref{se:X}.    
\begin{subequations}\label{se:Y} 

We begin with an element---not assumed central as in
\eqref{e:centralcocharxi}, but only semisimple--- 
\begin{equation}\label{e:centralcharacter} 
{}^\vee z = e(\gamma) \in {}^\vee H.
\end{equation}
We always wish to assume that $\gamma$ is integrally dominant \eqref{e:id}.
Almost all of our results will be about
the case of {\em regular} infinitesimal character, so we will assume
\begin{equation}\label{e:domreg}
\langle\gamma,\ch\alpha\rangle\in\Z\implies \langle\gamma,\ch\alpha\rangle> 0 \quad(\alpha\in R^+(G,H))
\end{equation}
or in other words
\begin{equation}
\langle\gamma,\alpha^\vee \rangle \notin \{0,-1,-2,-3,\ldots\} \qquad
(\alpha\in R^+(G,H)).
\end{equation}

Define 
\begin{equation}\label{e:chGz}
{}^\vee G({}^\vee z) = \text{centralizer of
  ${}^\vee z$ in ${}^\vee G$} \supset \ch H.
\end{equation}
(This closed reductive subgroup of $\ch G$ may be disconnected, a
point which will require some attention; we write $\ch G(\ch z)_0$ for
its identity component.) 
 An {\em {\tt atlas} dual strong involution of central character
   ${}^\vee z$} is an element  
$$\ch \xi \in \ch G\ch\xi_0, \qquad \ch \xi^2 = {}^\vee z,$$
and an {\em {\tt atlas} dual strong real form of central character
  ${}^\vee z$} is 
a $\ch G({}^\vee z)$-conjugacy class $\ch{\mathcal C}$ of such
elements. The automorphism
\begin{equation}
\ch\theta_{\ch \xi} = \int(\ch \xi)
\end{equation}
of $\ch G$ preserves $\ch G(\ch z)$, and acts on this group ({\em not}
in general on all of $\ch G$) as an involutive automorphism. {\em It
  is therefore real forms of $\ch G(\ch z)$ that are under discussion.}
 
Keeping in mind the case $\ch z \in \ch H$, we define
the {\em dual {\tt KGB} space}---now a space of equivalence classes of
Langlands parameters---by
\begin{equation}
\label{e:dualX}\begin{aligned}
\ch\wt\X &=\{\ch\xi\in \Norm_{\ch G\ch\xi_0}(\ch H)\mid \ch\xi^2\in\ch
H\}\\
\ch\X &= \ch\wt\X /\ch H; \end{aligned}
\end{equation}
just as in \eqref{e:X}, we are dividing by the conjugation action of
$\ch H$. We also write
\begin{equation}
\label{e:dualXz}\begin{aligned}
\ch\wt\X_{\ch z} &= \ch\wt\X_{\gamma} = \{\ch\xi\in \Norm_{\ch G\ch\xi_0}(\ch
H)\mid \ch\xi^2 = \ch z = e(\gamma)\},\\
\ch\X_{\ch z} &= \ch\X_\gamma = \ch\wt\X_\gamma/\ch H.\end{aligned}
\end{equation}
According to Definition \ref{def:Lparam}, a Langlands parameter of
type $\ch H$ and infinitesimal character $\gamma$ is a pair
$(\ch\xi,\gamma)$, with $\ch\xi \in 
\ch\wt\X_\gamma$. According to Proposition
\ref{p:Lequiv}, an equivalence class of Langlands parameters of
infinitesimal character $\gamma$ is a pair $(y,\gamma)$, with $y\in
\ch\X_\gamma$. 

Associated to $y\in \ch\X$ is an involution of
$\ch H$ (conjugation by the
element $\ch\xi\in\Norm_{\ch G\ch\xi_0}(\ch H)$---that is, the
restriction to $\ch H$ of $\ch \theta_{\ch\xi}$---for any
representative $\ch\xi$ of $y$). 
We denote this involution $\ch\theta_y$:
$$
\ch\theta_y = \ch w_y\ch\xi_0 \in \Aut(\ch H)\simeq \Aut(X^*(H)) \quad
(\ch w_y\in W(\ch G,\ch H)).
$$
Write 
\begin{equation}
\label{e:wy}
\ch\xi=\ch s_1 \sigma_{\ch w_y}\ch\xi_0
\end{equation}
(compare \eqref{e:torusxi}).
The fact that $(\ch\theta_y)^2 = 1$ is equivalent to
\begin{equation}\label{e:dualtwistedinv}
\ch w_y{}^\vee\xi_0(\ch w_y) = 1,
\end{equation}
i.e., $\ch w_y$ is a twisted involution (in $W$) with respect to the
automorphism ${}^\vee\xi_0$. Just as in \eqref{e:Xw}, we define
\begin{equation}
\ch\X^{\ch w} = \{y \in \ch\X \mid \ch w_y = \ch w\}.
\end{equation}
\end{subequations} 

Exactly as in Proposition \ref{p:rootdatumx}, we can now describe
the set of Langlands parameters attached to a given twisted
involution.

\begin{proposition}\label{p:rootdatumy} Fix an
infinitesimal   character $\gamma$  
and a  $\ch\xi_0$-twisted involution $\ch w$. Let 
  $\ch\theta_{\ch w} = \ch w\circ \xi_0 \in \Aut(H)$. 
The set $\X^{\ch w}_\gamma$ of dual {\tt KGB} classes of infinitesimal
character $\gamma$ (equivalently, of central character $\ch z=
e(\gamma)$) with $\ch w_y 
= \ch w$ (cf.~\eqref{e:wy}) is in one-to-one correspondence with  
$$
\left\{\overline \lambda\in X^*(H)/(1+\ch\theta_{\ch w})X^*(H) \mid
  (1-\ch\theta_{\ch w})\lambda = (1-\ch\theta_{\ch w})(\gamma - \rho)|\right\}.
$$
This set is either empty (if $(1-\ch\theta_{\ch w})(\gamma-\rho)$ does not
belong to $(1-\ch\theta_{\ch w})X^*$), or has a simply transitive action of  
$$(X^*)^{\ch\theta_{\ch w}}/(1+\ch\theta_{\ch w})X^*.$$ 
This latter group is a vector space over ${\mathbb Z}/2{\mathbb Z}$,
of dimension at most the rank of $X^*$.

The corresponding $y$ has a preferred representative
(defined by analogy with \eqref{e:thetacosets}) $\ch\xi$ with unnormalized torus
part
$$
\ch s_1 = e((\gamma-\lambda)/2),
$$
(see \eqref{e:wy}) or normalized torus part  
$$\ch s = e((\gamma-\rho - \lambda)/2)$$
(defined by analogy with \eqref{e:normtorusxi}). Here $\lambda\in X^*(H)$
is a representative  of $\overline\lambda$. 
If we modify the element $\lambda$ in its coset by adding
$(1+\ch\theta_{\ch w})\phi$ (for some $\phi\in X^*(H)$), then
$\ch s$ (or $\ch s_1$) is multiplied by $e((1+\ch\theta_{\ch
  w})\phi/2)$. That is, this preferred choice of torus part is unique
up to the image of 
$1+\ch\theta_{\ch w}$ acting on the elements $\ch H(2)$ of order 2 in
  $\ch H$.  Another
formulation is that these preferred representatives $\ch\xi$ of $y$ are a
single conjugacy class under $\ch H(2)$.
\end{proposition}

The proof is identical to that of Proposition \ref{p:rootdatumx},
and we omit it. 

Because the set is parametrized by certain (cosets of)
characters of $H$, it is easy and useful to reformulate the result
as follows.

\begin{corollary}\label{cor:rootdatumy}
In the setting of Proposition \ref{p:rootdatumy}, the set of dual {\tt
  KGB} classes $\ch\X^{\ch w}_\gamma$ is naturally in bijection with
the set of (automatically one-dimensional) irreducible $({\mathfrak
  h},H^{\theta_w})$-modules of differential equal to $\gamma-\rho$.
\end{corollary}
\begin{proof}
By definition an $({\mathfrak h},H^{\theta_w})$-module is a vector
space carrying an algebraic action of the group $H^{\theta_w}$, and a
representation of the abelian Lie algebra ${\mathfrak h}$, so that the
differential of the former is the restriction to ${\mathfrak
  h}^{\theta_w}$ of the latter.  In the corollary, we want ${\mathfrak
  h}$ to act by $\gamma-\rho$, so there is nothing to say about
that. The characters of $H^{\theta_w}$ are the restrictions to
$H^{\theta_w}$ of characters of $H$; so they are indexed by
\begin{equation}\label{e:Hthetahat}
\overline\lambda \in X^*(H)/(1-{}^t\theta_w)X^*(H),
\end{equation}
the denominator being the characters trivial on $H^{\theta_w}$.
Now it is clear that the modules we want are indexed exactly by such
cosets $\overline\lambda$, subject to the requirement
$$(1+{}^t\theta_w)\lambda = (1+{}^t\theta_w)(\gamma - \rho)$$
(that the differential of $\lambda$ is the restriction of
$\gamma-\rho$). This is exactly the condition on $\lambda$ written in
Proposition \ref{p:rootdatumy}. \end{proof}

\begin{definition}\label{def:atlasparam} A {\tt KGB} class $x$ and a
  dual {\tt KGB} class $y$ are said to be {\em aligned} if
$$-{}^t\theta_x = \ch\theta_y \in \Aut(X^*(H))$$
(\eqref{e:wx}, \eqref{e:wy}). Equivalently, we require the twisted
involutions to satisfy
$$w_x w_0 = \ch w_y.$$
In this case we call the pair $(x,y)$ an
{\em {\tt atlas} parameter} for $G$.  We write
$$
\mathcal Z=\{(x,y)\in \X\times\ch \X\mid -{}^t\theta_x=\ch\theta_y\},
$$
for the set of all {\tt atlas} parameters. If $z=e(g) \in Z(G)$ and
$\ch z = e(\gamma) \in \ch H$ (with $g$ and $\gamma$ regular and
integrally dominant) we write
$$\mathcal Z_{z,\ch z} = \mathcal Z_{g,\gamma} = \{(x,y) \in \mathcal
Z \mid x^2 = z,\ y^2 = \ch z\}$$
for the subset of parameters of infinitesimal
cocharacter $g$ and infinitesimal character $\gamma$.  If $w\in W$ is
a $\xi_0$-twisted involution, we define $\ch w = ww_0$ (a
$\ch\xi_0$-twisted involution, and write
$$\mathcal Z^w = \X^w \times \ch\X^{\ch w},$$
so that $\mathcal Z$ is the disjoint union over $\xi_0$-twisted involutions
$w$ of the subsets $\mathcal Z^w$.
\end{definition}

We are now in a position to sharpen the Langlands classification
Theorem \ref{thm:Langlands}, by parametrizing each L-packet.

\begin{theorem}[\cite{ABV}*{Theorem 1.18}]\label{thm:LanglandsSharp} In
  the setting of Definition \ref{def:Lparam}, fix a regular and
  integrally dominant infinitesimal character $\gamma\in \ch{\mathfrak
    h}$, and a regular integral dominant $g\in {\mathfrak h}$ (so that
  $e(g) \in Z(G)$). Then there is a natural bijection between
  irreducible admissible representations of infinitesimal character
  $\gamma$ of strong real forms of $G$ having infinitesimal
  cocharacter $g$; and the set of pairs $(x,y) \in {\mathcal
    Z}_{g,\gamma}$ of {\tt atlas} parameters of infinitesimal
  cocharacter $g$ and infinitesimal character $\gamma$. In this
  bijection, the strong real form may be taken to be any
  representative $\xi$ of the first factor $x$. The Langlands
  parameter (Definition \ref{def:Lparam} may be taken to be
  $(\ch\xi,\gamma)$, with $\ch\xi$ any representative of the second
  factor $y$. We write
$$J(x,y,\gamma)$$
for the irreducible module of infinitesimal character $\gamma$
attached to $(x,y)$.
\end{theorem} 

The notation requires some explanation, because we have not even said
of what group $J(x,y,\gamma)$ is a representation. If $\xi$ is any
representative of $x$, then $\theta_\xi = \int(\xi)$ is a well-defined
involutive automorphism of $G$, with fixed point group $K_\xi$ as in
\eqref{e:thetaK}. Then $J(x,y,\gamma)$ is an irreducible $({\mathfrak
  g},K_\xi)$-module.  A different choice $\xi'$ of representative of
$x$ gives rise to a (necessarily different, because $K_{\xi'}$ is
different) $({\mathfrak g},K_{\xi'})$-module $J'(x,y,\gamma)$. Part of
what the Theorem means is that these two different modules are
identified by the canonical bijection of Proposition
\ref{prop:switchstrong}. 
 
\begin{corollary}\label{cor:Lpacket} Suppose $(\ch\xi,\gamma)$ is a
  Langlands parameter 
  of type $\ch H$ and regular infinitesimal character $\gamma$. Fix a
  dominant regular infinitesimal cocharacter $g$ as in
  \eqref{e:domregdual}. Then the union (over strong real forms of
  infinitesimal cocharacter $g$) of the L-packets
  $\Pi_{\ch\xi,\gamma}(\xi)$ (Theorem \ref{thm:Langlands}) may be
  identified with the set $\X^w_g$ of Proposition
  \ref{p:rootdatumx}. (Here $w$ is the twisted involution dual to the
  one for $\ch\xi$.) In particular, this L-packet is either empty (if
  $(1-\theta_w)(g-\rho^\vee)$ does not belong to $(1-\theta_w)X_*$),
  or has a simply transitive action of 
$$X_*^{\theta_w}/(1+\theta_w)X_*.$$ 
\end{corollary}
Notice that if we consider real forms of infinitesimal cocharacter
$\rho^\vee$ (which includes the quasisplit real form) then this union
of L-packets is never empty. This is consistent with Langlands' result
that every L-packet is nonempty for the quasisplit real form.

This classical corollary is the most familiar way of thinking about
ambiguity in the Langlands classification: starting with a Langlands
parameter, and enumerating the various (strong) real forms where it can give a
representation. We are in fact going to be interested mostly in the
{\em dual} problem: starting with a strong real form of type $H$ and
and an infinitesimal character $\gamma$, and enumerating the various
Langlands parameters giving a representation. For example, if we start
with a split maximal torus, then the Langlands parameters in question
just index the characters of the split maximal torus of differential
$\gamma$. These admit a simply transitive action of the group
$({\mathbb Z}/2{\mathbb Z})^n$ of characters of the component group of
the split torus.

\begin{corollary}\label{cor:Rpacket} Suppose $\xi$ is a strong real
  form of type $H$ and dominant regular infinitesimal cocharacter
  $g$. Fix a dominant regular infinitesimal character $\gamma\in
  {\mathfrak h}^*$.  Then the collection of Langlands parameters
  $(\ch\xi,\gamma)$ of type $\ch H$ aligned with $\xi$ (Definition
  \ref{def:atlasparam}) may be identified with the set
  $\ch\X^{\ch w}_\gamma$ of Proposition
  \ref{p:rootdatumx}. (Here $\ch w = ww_0$ is the twisted involution
  dual to $w = w_\xi$.) In particular, this set of parameters is
  either empty (if $(1-\ch\theta_{\ch w})(\gamma-\rho)$ does not
belong to $(1-\ch\theta_{\ch w})X^*$), or has a simply transitive action of  
$$(X^*)^{\ch\theta_{\ch w}}/(1+\ch\theta_{\ch w})X^*.$$ 
\end{corollary}

We are going to need a slightly more precise understanding of how the
parametrization of representations in Theorem \ref{thm:LanglandsSharp}
actually works.
\begin{subequations}\label{se:Langlandsconstruction}
So let us fix an {\tt atlas} parameter $(x,y)$ (Definition
\ref{def:atlasparam}) of (integrally dominant regular) infinitesimal
character $\gamma \in {\mathfrak h}^*$. Choose a strong real form
representative $\xi$ for $x$, so that what we are seeking to construct
is an irreducible $({\mathfrak g},K_\xi)$-module $J(x,y,\gamma)$. The
construction begins with the $\theta_\xi$-stable Cartan subgroup
$H$. The Cartan involution $\theta_\xi$ acts on $H$ by $\theta_w$, so
\begin{equation}
H \cap K_\xi = H^{\theta_w}.
\end{equation}
By definition of {\tt atlas} parameter, $y\in \ch\X^{\ch
  w}_\gamma$; so by Corollary \ref{cor:rootdatumy}, $y$ defines
\begin{equation}\label{e:Cy}
{\mathbb C}(y,\gamma) = \text{irreducible $({\mathfrak h},H\cap K_\xi)$-module
  of differential $\gamma-\rho$.}
\end{equation}
We want to construct a $({\mathfrak g},K_\xi)$-module using the
character ${\mathbb C}(y,\gamma)$. This is a large and complicated problem,
solved by work of Zuckerman reported in \cite{green}, but here
is a sketch. (Shorthand for this construction is {\em cohomological
  induction}, and we will use that phrase to refer to it.) 

Now we extend ${\mathbb C}(y,\gamma)$ to a $({\mathfrak
  b},H\cap K_\xi)$-module by making
${\mathfrak n}$ 
act trivially, and then form the (dual to Verma) $({\mathfrak g}, H\cap
K_\xi)$-module
\begin{equation}\label{e:Vermay}
M(y,\gamma) = \Hom_{{\mathfrak b}}(U({\mathfrak
  g}),{\mathbb C}(y,\gamma)\otimes{\mathbb C}(2\rho)).
\end{equation} 
Here ${\mathbb C}(2\rho)$ is the representation of $B$ on the top
exterior power of the Lie algebra: the sum of the positive roots. The
weight of ${\mathfrak h}$ by which we are ``producing'' is
$\gamma+\rho$, so this is the lowest weight of $M(y,\gamma)$.
By the theory of Verma modules, $M(y,\gamma)$ has infinitesimal character
$\gamma$. Now we apply the Zuckerman right derived functor
(\cite{KV}*{(2.113)}) 
\begin{equation}\left(\Gamma_{{\mathfrak g},H\cap K_\xi}^{{\mathfrak
      g},K_\xi}\right)^S\colon ({\mathfrak g},H\cap K_\xi)\text{-modules}
  \rightarrow ({\mathfrak g},K_\xi)\text{-modules},
\end{equation}
with $S=\dim{\mathfrak n}\cap {\mathfrak k}_\xi$, obtaining what is called the
{\em standard $({\mathfrak g},K_\xi)$-module}
\begin{equation}
I^{\text{quo}}(x,y,\gamma)= \left(\Gamma_{{\mathfrak g},H\cap
    K_\xi}^{{\mathfrak g},K_\xi}\right)^S(M(y,\gamma)).
\end{equation}
This module has finite length, and has a unique irreducible quotient
$J(x,y,\gamma)$.
Proofs may be found in \cite{KV}*{Theorem 11.129}.
\end{subequations}

\section{Twisting parameters}\label{sec:twist}
\setcounter{equation}{0}

We want to consider the action of $\delta_0$ on representations. 
In terms of parameters, we need to study the action 
of $\delta_0$ on $\X$ (and $\ch\delta_0$ on $\ch\X$). 
We do this in the setting of Propositions \ref{p:rootdatumy} and
\ref{p:rootdatumx}. 

Of course a $\delta_0$-fixed representation of infinitesimal character
$\gamma$ can exist only if the infinitesimal character $\gamma$ is
itself fixed by $\delta_0$; that is, if and only if
\begin{equation}
\ch\delta_0(\gamma) = z\cdot\gamma\qquad \text{(some $z\in W$).}
\end{equation}
Because of the integrally dominant condition \eqref{e:domreg} that we
impose on $\gamma$ (and which is automatically inherited by
$\ch\delta_0(\gamma)$), it follows that
\begin{equation}
z\cdot R^+(\gamma) = R^+(\ch\delta_0(\gamma)) \subset R^+;
\end{equation}
here $R^+(\gamma)$ is the set of positive integral roots defined by
$\gamma$. If $\gamma$ is integral (so that $R^+(\gamma) = R^+$), such
a condition forces $z=1$, i.e., $\ch\delta_0(\gamma) = \gamma$. If
$\gamma$ is {\em not} integral, however, we can draw no such
conclusion. Here is a convenient substitute.

\begin{lemma}\label{lemma:deltagamma} Suppose $\gamma \in {\mathfrak
    h}^*$, and that $\ch\delta_0$ preserves the $W$ orbit of $\gamma$.
  Then this orbit has an integrally dominant representative $\gamma'$
  with the property that
$$\ch\delta_0(\gamma') = \gamma'.$$
\end{lemma}

We omit the (elementary) proof.  Because of this lemma, it is
sufficient to study representations of infinitesimal character
represented by a $\ch\delta_0$-fixed integrally dominant weight
\begin{equation}\label{e:fixeddomreg}
\ch\delta_0(\gamma) = \gamma, \quad \langle\gamma,\beta^\vee \rangle
\notin \{0,-1,-2,\dots\} \quad (\beta \in R^+(G,H)).
\end{equation}

The situation for real forms is a bit more subtle, because the
infinitesimal cocharacter is not an invariant of a real form, but
merely a useful extra parameter that we attach to the real form.
Here is what we would like to know.

\begin{conjecture}\label{conj:deltag} Suppose $\xi$ is a strong
  involution for $G$ (see 
  \eqref{e:calI}), of (dominant regular integral) infinitesimal
  cocharacter $g$. Assume that the involution
  $\delta_0\circ\theta_\xi\circ\delta_0$ is equivalent (that is,
  conjugate by $G$) to
  $\theta_\xi$. Then there is a $\delta_0$-fixed regular integral
  $g'$, and a strong involution $\xi'$ of infinitesimal cocharacter
  $g'$, such that $\theta_\xi = \theta_{\xi'}$.
\end{conjecture}

Unfortunately this statement is false.

\begin{example}\label{ex:badrealform} Suppose $G = SL(4)$, endowed
  with the trivial 
  distinguished automorphism $\xi_0$ (so that we are considering equal
  rank real forms) and the nontrivial distinguished automorphism
  $\delta_0$. On the diagonal torus,
$$\delta_0(a_1,a_2,a_3,a_4) = (a_4^{-1},a_3^{-1},a_2^{-1},a_1^{-1}).$$
Let $\omega = \exp(2\pi i/8)$ be a primitive eighth root of $1$, and define
$$\xi = \diag(\omega,\omega,\omega,\omega^5)
=\omega[\diag(1,1,1,-1)].$$
Then $\xi^2 = \omega^2I=iI$, a central element of order $4$; and
$$K_\xi = S(GL(3) \times GL(1)),$$
the complexified maximal compact subgroup for the real form $SU(3,1)$
of $G$. The infinitesimal cocharacter of this strong real form is any
weight of the form
$$\begin{aligned} g &= (g_1,g_2,g_3,g_4) \in {\mathbb Q}^4, \\ 
\sum g_j &= 0,\qquad g_1 > g_2 > g_3 > g_4, \\
\exp(2\pi i g_j) &= \xi_j;\end{aligned}$$
that is,
$$g_j \cong \begin{cases} 1/8 \pmod{\mathbb Z} & (g=1,2,3)\\
5/8 \pmod{\mathbb Z} & (g = 4),\end{cases}$$
These conditions are easily satisfied (for example by $(17/8,9/8,
1/8,-27/8)$); but it is easy to see that they {\em cannot} be
satisfied by a $\delta_0$-fixed $g$. The reason is that
$$\delta_0(g) = (-g_4,-g_3,-g_2,-g_1).$$
If this is equal to $g$, then $g_3 = -g_2$, which contradicts the
requirements $g_3 \cong g_2 \cong 1/8 \pmod{\mathbb Z}$.  On the
other hand, $\delta_0\circ\theta_\xi\circ\delta_0$ {\em is} conjugate
to $\theta_\xi$ (by a cyclic permutation matrix).

One might hope that in this example none of the representations of
$SU(3,1)$ is fixed by $\delta_0$, and indeed none of the four discrete
series representations is fixed; but there is a spherical principal
series representation (of infinitesimal character $\rho$) which {\em
  is} fixed.

\end{example}

In any case, we are going to consider only cases when Conjecture
\ref{conj:deltag}  is
true; that is, we are going to consider only real forms of $G$ of
infinitesimal cocharacter $g$ satisfying
\begin{equation}\label{e:fixedcodomreg}
\delta_0(g) = g, \quad \langle\beta,g \rangle
\notin \{0,-1,-2,\dots\} \quad (\beta \in R^+(G,H)).
\end{equation}
In general there will be an extra twist 
by the central element $z$, ($-I$ in the example),
satisfying $\delta_0(\xi^2)=z\xi^2$.

Suppose $x\in\X$, and $\xi\in p\inv(x)$. 
Then $\delta_0(x)=x$ if and only if 
$\delta_0(\xi)=h\inv\xi h$ for some $h\in H$ (we cannot necessarily
choose $\xi$ so that $h=1$).  
This is equivalent to 
$$
\xi(h\delta_0)\xi\inv=h\delta_0,
$$
i.e., 
$$
({}^{\delta_0} H)^{\theta_\xi}=\langle H^{\theta_x},h\delta_0\rangle
$$
Suppose $\xi$ corresponds to $\overline\ell\in X_*(H)/(1+\theta_x)X_*(H)$ by
Proposition \ref{p:rootdatumx}, and $\ell\in X_*(H)$ is a
representative. Then $x$ is $\delta_0$-fixed if and only if 
$$\delta_0\ell\in\ell+(1+\theta_x)X_*(H).$$ 
Here is a precise statement.

\begin{proposition}\label{prop:xtwist} 
\addtocounter{equation}{-1}
\begin{subequations}
Suppose $g$ is an infinitesimal cocharacter 
as in \eqref{e:domregdual}.
Suppose  $x\in \X$ has infinitesimal cocharacter $g$, and let $w=w_x$
be the underlying twisted involution 
\eqref{e:wx}.
Assume
\begin{equation}
\delta_0(g) = g,\quad  \delta_0(w)=w.
\end{equation}
Suppose that $x$ corresponds via Proposition \ref{p:rootdatumx} to
$\overline \ell\in X_*(H)/(1+\theta_x)X_*(H)$.
Choose $\ell\in X_*(H)$ representing $\overline\ell$, so 
$x$ has a representative with unnormalized torus part
$s_1=e((g-\ell)/2)$, or normalized torus part $s=e((g-\ell-\rho^\vee)/2)$. 
\begin{enumerate}
\item The class $x$ is fixed by $\delta_0$ if and only if
\begin{equation}\label{e:elltwist}
(\delta_0-1)\ell =(1+\theta_x)t \qquad \text{(some $t\in
  X_*(H)$)}.
\end{equation}
\item The element $t$ is uniquely defined by $\ell$ up to adding
$X_*(H)^{-\theta_x}$; if also $\ell$ is modified in its coset,
then $t$ changes by $(1-\delta_0)X_*(H)$.

\item The corresponding special representative
\begin{equation}
\xi = e((g-\ell)/2)\sigma_w\xi_0
\end{equation}
satisfies
\begin{equation}\label{e:xitwister}
\delta_0 \xi \delta_0^{-1} = e((1-\theta_x)t/2) \xi = e(t/2)\xi e(-t/2);
\end{equation}
that is, $\xi$ is conjugate to its $\delta_0$ twist using
the element $e(t/2)$ of $H(2)$.  
\item Condition \eqref{e:xitwister} is equivalent to 
\begin{equation}
\label{e:extcartan}
({}^{\delta_0}H)^{\theta_\xi}=\langle H^{\theta_x},e(-t/2)\delta_0\rangle
\end{equation}
\end{enumerate}
\end{subequations}
\end{proposition} 

Here is the version for $\ch G$.

\begin{proposition}\label{prop:ytwist} 
\addtocounter{equation}{-1}
\begin{subequations}
Suppose $\gamma$ is an infinitesimal character as in \eqref{e:domreg}.
Suppose  $y\in\ch\X$ has infinitesimal character $\gamma$, and let
$ w=w_y$ be the underlying twisted involution \eqref{e:wy}. Assume
\begin{equation}
\ch\delta_0(\gamma) = \gamma,\quad  \ch\delta_0( w)=w.
\end{equation}
Suppose that $y$ corresponds via Proposition \ref{p:rootdatumy} to
$\overline \lambda\in X^*(H)/(1+\theta_x)X^*(H)$. Choose $\lambda\in X^*(H)$ representing $\overline\lambda$,
so $y$ has a representative with unnormalized torus part
$e((\gamma-\lambda)/2)$, and normalized torus part $e((\gamma-\lambda-\rho)/2)$. 
\begin{enumerate}
\item The class $y$ is fixed by $\ch \delta_0$ if and only if
\begin{equation}\label{e:mutwist}
(\ch\delta_0-1)(\lambda) = (1+\ch\theta_y)\tau \qquad \text{(some $\tau\in
  X^*(H)$)}.
\end{equation}
\item The element $\tau$ is uniquely defined by $\lambda$ up to adding
$X^*(H)^{-\ch\theta_y}$; if also $\ell$ is modified in its coset,
then $t$ changes by $(1-\delta_0)X^*(H)$.
\item The corresponding special representative
\begin{equation}
\ch\xi = e((\gamma-\lambda)/2)\ch\sigma_w\ch \xi_0
\end{equation}
satisfies
\begin{equation}
\ch\delta_0 (\ch\xi)\ch\delta_0^{-1}=e((1-\ch\theta_y)\tau/2)
 = e(\tau/2)\ch\xi e(-\tau/2) \xi;  
\end{equation}
that is, $\ch\xi$ is conjugate to its $\ch\delta_0$ twist using
the element $e(\tau/2)$ of $\ch H(2)$.  
\item Condition \eqref{e:xitwister} is equivalent to 
\begin{equation}
\label{e:extcartandual}
({}^{\ch\delta_0}\ch H)^{\ch\theta_{\ch\xi}}=\langle
H^{\ch\theta_y},e(-\tau/2)\ch\delta_0\rangle 
\end{equation}
\end{enumerate}
\end{subequations}
\end{proposition}

\section{Extended parameters}\label{sec:extparam}
\setcounter{equation}{0}

We now define parameters for $(\g,\Kext)$-modules.  Suppose
$(x,\overline\lambda,\gamma)$ is a $\delta_0$-fixed parameter.  If
$\xi\in p\inv(x)\in \wt\X$ then $J(x,\overline\lambda,\gamma)$ is a
$\delta_0$-fixed $(\g,K_\xi)$-module.  As discussed in the
Introduction this can be extended in two ways to give a
$(\g,{}^{\delta_0} K_\xi)$-module.

\begin{lemma}
\label{l:extparam}
Suppose $(x,\overline\lambda,\gamma)$ is a $\delta_0$-fixed parameter.
Choose $h\delta_0\in ({}^{\delta_0} H)^{\theta_\xi}$ as in
Proposition \ref{prop:xtwist}. The two extensions of
$J(x,\overline\lambda,\nu)$ to a $(\g,{}^{\delta_0} K_\xi)$-module are
parametrized by the two extensions of the character $\overline\lambda$
of $H^{\theta_x}$ to 
$$({}^{\delta_0} H)^{\theta_\xi} = \langle H^{\theta_x},h\delta_0\rangle,$$ 
whose values at $h\delta_0$ are the
two square roots of $\overline\lambda(h\delta_0(h))$.
\end{lemma}

\begin{subequations}\label{se:extparam}
We now begin to assemble the data---the {\em extended parameters} of
Definition \ref{d:extparam}---that we will use to construct one of the
square roots required in Lemma \ref{l:extparam}.  We will be
considering representations with a fixed regular infinitesimal
character, for real forms with a fixed infinitesimal cocharacter.  So
fix an integrally dominant infinitesimal character $\gamma$:
\begin{equation}\label{e:inf}
\gamma\in X^*(H)_{\mathbb C} \subset {\mathfrak h}^*,\quad
\langle\gamma,\ch\alpha\rangle \notin \Z_{<0} \quad (\alpha\in
  R^+(G,H)) 
\end{equation}
and an integral dominant infinitesimal cocharacter  $g$:
\begin{equation}
g\in X_*(H)_\Q\subset {\mathfrak h},\quad  \langle g,\alpha\rangle\in
\Z_{>0} \quad (\alpha\in R^+(G,H)).
\end{equation}
\end{subequations}
 We require (see Lemma \ref{lemma:deltagamma} and Conjecture
 \ref{conj:deltag}) 
\begin{equation}
\label{e:inffixed}
\delta_0(g)=g, \quad{}^t\delta_0(\gamma)=\gamma
\end{equation}

\begin{definition}
\label{d:epsilon}\addtocounter{equation}{-1}
\begin{subequations}            
Suppose $(x,y,\gamma)$ is a  parameter for a $\delta_0$-fixed representation.
Define $\overline\lambda\in X^*(H)/(1-\theta_x)X^*(H)\simeq X^*(H^{\theta_x})$ (from $y$) by Proposition \ref{p:rootdatumy}.
Choose a good representative $\xi$ for $x$, and
define $\overline\ell\in X_*(H)/(1+\theta_x)X_*(H)$ corresponding to $\xi$,
by Proposition \ref{p:rootdatumx}.
Choose a representative $\ell\in X_*(H)$ for $\overline\ell$, and choose 
$t\in X_*(H)$ satisfying \eqref{e:elltwist}. Set
$h=e(t/2)$ so $h\delta_0\in ({}^{\delta_0} H)^{\theta_\xi}$ and $(h\delta_0)^2=h\delta_0(h)\in H^{\theta_x}$. 
Define
\begin{equation}
\label{e:epsilon}
\epsilon(x,y)=\overline\lambda(h\delta_0(h)).
\end{equation}
\end{subequations}
\end{definition}

\begin{lemma}
\label{l:epsilon}
\hfil
\begin{enumerate}
\item $\epsilon(x,y)=(-1)^{\langle\overline\lambda,(1+\delta_0)t\rangle}$,
\item  $\epsilon(x,y)$ is independent of the choices of $\xi$,
  $\overline\ell$, $\ell$, and $t$ (for fixed $g$ and  $\gamma$).  
\end{enumerate}
\end{lemma}

\begin{proof}
The first statment is immediate.
By \eqref{e:elltwist}, $t$ is determined by $\ell$ up to adding elements of 
$X_*(H)^{-\theta_x}$, and by $\overline\ell$ up to
$(1-\delta_0)X^*(H)$. Therefore 
$$
t\text{ is determined by }x\text{ up to adding }
X^*(H)^{-\theta_x}+(1-\delta_0)X^*(H)
$$
We have
\begin{equation}
\label{e:epsilons}
\begin{aligned}
(-1)^{\langle \lambda,(1+\delta_0)t\rangle}&=
(-1)^{\langle \lambda,(1\pm\delta_0)t\rangle}\\
&=
(-1)^{\langle (1\pm{}^t\delta_0)\lambda,t\rangle}\\
&=(-1)^{\langle (1\pm\ch\theta_y)\tau,t\rangle}\\
&=(-1)^{\langle \tau,(1\pm \theta_x)t\rangle}\\
\end{aligned}
\end{equation}
The second equality shows this sign is unchanged by adding to $t$ an
element of $(1-\delta_0)X^*(H)$, and the 
last one shows it is unaffected by adding elements of
$X^*(H)^{-\theta_x}$. 
\end{proof}

We need to choose a square root of $\epsilon(x,y)$. Just as for the
parameters $(x,y)$ for representations of real forms of $G$, 
it is helpful to symmetrize the picture with respect to $G$ and $\ch G$.

\begin{definition}
\label{d:extparam}
Fix $\gamma,\ g$ as in \eqref{se:extparam}, and a $\xi_0$-twisted
involution $w\in W$. 
Let $\theta=\theta_w = w\xi_0 \in \Aut(H)$ and $\ch\theta=
\ch\theta_{ww_0} = -{}^t\theta$. 

An {\it extended parameter} (for the twisted involution $w$ and the
specified infinitesimal character and cocharacter) is a set
$$
E=(\lambda,\tau,\ell,t)
$$
where 
\begin{enumerate}
\item $\lambda\in X^*(H)$ satisfies $(1-\ch\theta)\lambda =
  (1-\ch\theta)(\gamma-\rho)$;  
\item $\ell\in X_*(H)$ satisfies $(1-\theta)\ell =
  (1-\theta)(g-\ch\rho)$; 
\item $\tau\in X^*(H)$ satisfies $(\ch\delta_0-1)\lambda =
  (1+\ch\theta)\tau$; 
\item $t\in X_*(H)$ satisfies $(\delta_0-1)\ell=(1+\theta)t$.
\end{enumerate}
\end{definition}
Associated to an extended parameter $E=(\lambda,\tau,\ell,t)$ are the
following elements: 
\begin{enumerate}[label=(\alph*)]
\item $\xi(E)\in \wt\X$ corresponds to $\lambda$ by Proposition
  \ref{p:rootdatumy}; 
\item $\ch\xi(E)\in \wt{\ch\X}$ corresponds to $\ell$ by Proposition
  \ref{p:rootdatumx}; 
\item $x(E)=_{\text{def}}p(\xi(E))\in\X$, $x(E)^2=e(g)$;
\item $y(E)=_{\text{def}}p(\ch\xi(E))\in\ch\X$, $y(E)^2=e(\gamma)$;
\item $h(E)\delta_0=e(t/2)\delta_0\in ({}^{\delta_0} H)^{\theta_\xi}$
  (cf. \eqref{e:extcartan}); 
\item $\ch h(E)\ch\delta_0=e(\tau/2)\ch\delta_0\in
  ({}^{\ch\delta_0\vee} H)^{\ch\theta_{\ch\xi}}$
  (cf.~\eqref{e:extcartandual}). 
\end{enumerate}

We say $E$ is an {\em extended parameter for $(x(E),y(E))$}.

\begin{definition}\label{d:extparamz}
Suppose $(\lambda,\tau,\ell,t)$ is an extended parameter for $(x,y)$. 
Define
\begin{equation}
\label{e:extparamz}
z(\lambda,\tau,\ell,t) = i^{\langle\tau, (1+\theta_x)t\rangle}
(-1)^{\langle \lambda,t\rangle} \\
\end{equation}
\end{definition}

By \eqref{e:epsilons} we have:
\begin{equation}
z(\lambda,\tau,\ell,t)^2=\epsilon(x,y).
\end{equation}

Associated to $(\lambda,\tau,\ell,t)$ is an extension of
$J(x,y,\gamma)$ defined as follows.

\begin{definition}
\label{d:extparam2}
Suppose $(\lambda,\tau,\ell,t)$ is an extended parameter for $(x,y)$. 
Set $\xi=\xi(\lambda,\tau,\ell,t)$ and $h=h(\lambda,\tau,\ell,t)=e(t/2)$. 
Define an extension of $\overline\lambda$ to $({}^{\delta_0}
H)^{\theta_\xi}$  (see Lemma \ref{l:extparam}) by having
it take the value $z(\lambda,\tau,\ell,t)$ at 
$h\delta_0$. This defines  an extension of $J(x,y,\gamma)$ to a
$(\g,{}^\delta_0 K_\xi)$-module, denoted
$J_z(\lambda,\tau,\ell,t)$. (The subscript $z$ refers to the
particular formula chosen in Definition \ref{d:extparamz}.)
\end{definition}

We deal with the question of equivalence of parameters in the next Section.

For later use we record precisely how  these elements depend on the various choices.
Suppose we are given $(x,y)\in\mathcal Z$.
Choose representatives $\xi$ for $x$ and $\ch\xi$ for $y$ by Propositions 
\ref{p:rootdatumx} and \ref{p:rootdatumy}, respectively.
That is
\begin{equation}\label{e:lambdam}
\begin{aligned}
\xi &= e((g-\ell)/2)\sigma_w \xi_0\\
\ch\xi &=e((\gamma-\lambda)/2){}^\vee\sigma_{ww_0} {}^\vee\xi_0.
\end{aligned}
\end{equation}
Then
\begin{equation}\label{e:lambdamunique}
\begin{aligned}
\ell &\text{\ is determined by $\xi$ up to $2X_*^{\theta_x}$}\\
\ell &\text{\ is determined by $x$ up to $(1+\theta_x)X_*$}\\
\lambda &\text{\ is determined by $\ch\xi$ up to $2(X^*)^{{}^\vee\theta_y}$}\\
\lambda &\text{\  is determined by $y$ up to $(1+{}^\vee\theta_y)X^*$}\\
\end{aligned}
\end{equation}
It is helpful to write in addition 
\begin{equation}\label{e:ftphitau}
\begin{aligned}
f &= (\delta_0-1)\ell = (1+\theta_x)t\\
\phi &= ({}^\vee\delta_0-1)\lambda = (1+{}^\vee\theta_y)\tau\\
\end{aligned}
\end{equation}
Because (for example) $t$ is evidently determined by $f$ up to
$X_*^{-\theta_x}$, the corresponding uniqueness statements are
\begin{equation}\label{e:ftphitauunique}
\begin{aligned}
f &\text{\ is determined by $\xi$ up to $2(1-\delta_0)X_*^{\theta_x}$}\\
f &\text{\ is determined by $x$ up to $(1-\delta_0)(1+\theta_x)X_*$}\\
t &\text{\ is determined by $\xi$ up to $(1-\delta_0)X_*^{\theta_x} +
  X_*^{-\theta_x}$}\\ 
t &\text{\ is determined by $x$ up to $(1-\delta_0)
X_* +
  X_*^{-\theta_x}$}\\
\phi &\text{\ is determined by $\ch\xi$ up to
  $2(1-{}^\vee\delta_0)(X^*)^{{}^\vee\theta_y}$}\\
\phi &\text{\ is determined by $y$ up to $(1-{}^\vee\delta_0)(1+{}^\vee\theta_y)X^*$}\\
\tau &\text{\ is determined by $\ch\xi$ up to $(1-{}^\vee\delta_0)(X^*)^{{}^\vee\theta_y} +
  (X^*)^{-{}^\vee\theta_y}$}\\ 
\tau &\text{\ is determined by $y$ up to
  $(1-{}^\vee\delta_0)
X^* + (X^*)^{-{}^\vee\theta_y}$}\\
\end{aligned}
\end{equation}
Eventually we will want a parallel choice of square root of $\epsilon$
related to the dual group ${}^\vee G$. This is
\begin{equation}\label{e:extparamzeta}
\begin{aligned}
\zeta(\lambda,\tau,\ell,t) &=_{\text{def}} i^{\langle\tau,f\rangle}
(-1)^{\langle\tau ,\ell\rangle}\\ 
&=z(\lambda,\tau,\ell,t) (-1)^{\langle\lambda,t\rangle}
(-1)^{\langle\tau,\ell \rangle}. 
\end{aligned}
\end{equation}

\section{Equivalences of extended parameters}\label{sec:equiv}
\setcounter{equation}{0}
In this section we record how to tell when two of the extended
modules defined in Definition \ref{d:extparam2} are equivalent.

Fix $\gamma$ and $g$ as usual, and 
suppose $(x,y)$ is a $\delta_0$-fixed parameter.
Choose two extended parameters 
\begin{subequations}\label{se:equivextended}
\begin{equation}\label{e:twoextended}
E=(\lambda,\tau,\ell,t),\qquad E'=(\lambda',\tau',\ell',t')
\end{equation}
for $(x,y)$ (Definition \ref{d:extparam}).
Set $\xi=\xi(E)$, $\xi'=\xi(E')$, 
and define
\begin{equation}
K_\xi = \Cent_G(\xi), \qquad {}^{\delta_0} K_\xi = \Cent_{\langle
  G,\delta_0\rangle}(\xi),
\end{equation}
and similarly with primes. Because $\xi$ and $\xi'$ are assumed to be
conjugate by $G$, Proposition \ref{prop:switchstrong} provides a 
{\em canonical} identification
\begin{equation}\label{e:twistequiv}
\text{irreducible $({\mathfrak g},K_\xi)$-modules}\simeq
\text{irreducible $({\mathfrak g},K_{\xi'})$-modules}
\end{equation}
(by twisting the action by $\Ad(g)$). Exactly the same argument applies to irreducible $({\mathfrak g},{}^{\delta_0} K_\xi)$-modules. 
\end{subequations}

\begin{definition}
\label{d:sgn} \addtocounter{equation}{-1}
\begin{subequations}            
We say {\em $E$ is equivalent to $E'$} if 
$J_z(E)$ and $J_z(E')$ correspond by this canonical identification.

Define $\sgn(E,E')=1$ if $J_z(E)\sim J_z(E')$, or $-1$ otherwise.

In other words, if $[\ ]$ denotes the image of a representation of an
extended group in the module $\mathcal M$ (see the Introduction or
Section \ref{sec:hecke}) then
\begin{equation}
\label{e:sgn}
[J_z(E)]=\sgn(E,E')[J_z(E')].
\end{equation}
\end{subequations}
\end{definition}
\begin{subequations}

The Langlands classification attaches to $(\xi,y)$ an irreducible
$({\mathfrak g},K_\xi)$-module $J(\xi,y)$. The construction of
$J(\xi,y,\gamma)$ begins
with a one-dimensional $({\mathfrak
  h},H^{\theta_\xi})$-module ${\mathbb C}_{y,\gamma}$. Cohomological
induction produces a ``standard'' $({\mathfrak g},K_\xi)$-module
$I(\xi,y,\gamma)$, with unique irreducible
quotient $J(\xi,y,\gamma)$. The nature of this construction makes it obvious that the
identification of \eqref{e:twistequiv} carries $I(\xi,y,\gamma)$ to
$I(\xi',y,\gamma)$, and consequently $J(\xi,y,\gamma)$ to $J(\xi',y,\gamma)$. 

Here is more detail on  
 how the extended group representation of Definition \ref{d:extparam2} is
constructed. First, the element $e(t/2)\delta_0$ is a generator for
the extended Cartan:
\begin{equation}
({}^{\delta_0}H)^{\theta_\xi} = \langle e(t/2)\delta_0, H^{\theta_\xi}\rangle
\end{equation}
The one-dimensional module ${\mathbb C}_{y,\gamma}$ extends to a 
one-dimensional $({\mathfrak h},({}^{\delta_0}H)^{\theta_\xi})$-module
by declaring
\begin{equation}\label{e:extparamz2}
e(t/2)\delta_0 \text{ acts by the scalar }z(\lambda,\tau,\ell,t).
\end{equation}
Cohomological induction from this one-dimensional provides an
extension of $I(\xi,y,\gamma)$ to a $({\mathfrak g},{}^{\delta_0}
K_\xi)$-module $I_z(\lambda,\tau,\ell,t)$, and then
$J_z(\lambda,\tau,\ell,t)$ is its unique irreducible quotient. Of course
exactly the same words describe $J_z(\lambda',\tau',\ell',t')$.

So how do we decide whether these two modules are equivalent?
According to \eqref{e:lambdamunique}, we can find $u\in X_*(H)$ so
that 
\begin{equation}\label{e:changeell}
\begin{aligned}
\ell' &= \ell + (\theta_x+1)u\\
f' &= f + (\theta_x+1)(\delta_0 -1)u
\end{aligned}
\end{equation}
It follows that
\begin{equation}
e(u/2)\cdot\xi\cdot e(-u/2) = \xi'.
\end{equation}
If we define 
\begin{equation}
t_2 = t + (\delta_0 - 1)u, 
\end{equation}
then $(\ell',t_2)$ is another choice of representative for $x$ as in
\eqref{se:extparam}, and in fact conjugate to $(\ell,t)$ by
$e(u/2)$:
\begin{equation}
e(u/2)\cdot e(t/2)\delta_0\cdot e(-u/2) = e(t_2/2)\delta_0.
\end{equation}
Consequently 
\begin{equation}\label{e:changet}
i =_{\text{def}} t'-t_2 \in X_*^{-\theta_x}, \qquad t'= t + (\delta_0
- 1)u +i. 
\end{equation}
In exactly the same way, we find
\begin{equation}\label{e:changelambdatau}
\begin{aligned}
\lambda' &= \lambda+ ({}^\vee\theta_y + 1)\omega \qquad \text{(some
  $\omega\in X^*(H)$)}\\
\tau' &= \tau + ({}^\vee\delta_0 - 1)\omega +\iota \qquad \text{(some
  $\iota\in X^*(H)^{-\theta_y}$)}\\
\phi' &= \phi + ({}^\vee\delta_0 - 1)({}^\vee\theta_y + 1)\omega
\end{aligned}
\end{equation}
\end{subequations}

\begin{proposition}\label{prop:equivextended}
Suppose $E=(\lambda,\tau,\ell,t)$ and $E'=(\lambda',\ell',\tau',t')$
are extended parameters for $(x,y)$.  
Then 
$$
\sgn(E,E')=(-1)^{\langle(1+{}^\vee\delta_0)\tau,w\rangle}(-1)^{\langle\iota,t'\rangle}.
$$
Here $u$ and $\iota$ are 
defined in \eqref{e:changeell}, \eqref{e:changet}, and
\eqref{e:changelambdatau}.
\end{proposition}
\begin{proof}
\begin{subequations}\label{se:proofequiv}
We change the parameter $(\lambda,\tau,\ell,t)$ to
$(\lambda',\ell',\tau',t)$ in three steps:
\begin{equation}
\begin{aligned}
\label{e:path}
E=(\lambda,\tau,\ell,t) \rightarrow F=(\lambda,\tau,\ell',t_2) &\rightarrow\\
G=(\lambda,\tau,\ell',t') \rightarrow E'=&(\lambda',\tau',\ell',t').
\end{aligned}
\end{equation}

In the first step we have conjugated by
$e(u/2)$. It follows easily that the extended representations
correspond if and only if the scalars chosen for the actions of
$e(t/2)\delta_0$ and $e(t_2/2)\delta_0$ agree. That is,
\begin{equation}\label{e:step1}
\sgn(E,F)=z(E)/z(F).
\end{equation}
At the second step of \eqref{e:path}, we are keeping the group
${}^{\delta_0} K_{\xi'}$ the same, but changing the representative of
the extended Cartan from $e(t_2/2)\delta_0$ to
$e(t'/2)\delta_0$. This gives an equivalent extended parameter exactly
if we multiply the scalar by 
$$
(-1)^{\langle\lambda,t'-t_2\rangle} =
(-1)^{\langle\lambda,i\rangle}.$$
Therefore
\begin{equation}\label{e:step2}
\sgn(F,G)=\frac{z(F)}{z(G)}(-1)^{\langle\lambda,i\rangle}.
\end{equation}

Finally, in the last step of \eqref{e:path} the group and the extended
Cartan representative remain the same; all that may change is the
scalar $z$. Therefore
\begin{equation}\label{e:step3}
\sgn(G,E')=z(G)/z(E').
\end{equation}
Combining \eqref{e:step1}--\eqref{e:step3}, we find
\begin{equation}
\label{e:sgnEE'}
\sgn(E,E')=\frac{z(E)}{z(F)}\frac{z(F)}{z(G)}(-1)^{\langle \lambda, i\rangle}\frac{z(G)}{z(E')}=\frac{z(E)}{z(E')}(-1)^{\langle \lambda, i\rangle}.
\end{equation}
It remains to compute $z(E)/z(E')$. We do this in two steps. 
First of all we have from \eqref{e:extparamz}
\begin{equation}
z(E)/z(G)=i^{\langle \tau,(1+\theta_x)(t-t')\rangle}(-1)^{\langle \lambda, t-t'\rangle}.
\end{equation}
With $u$ and $i$ given by \eqref{e:changet} this gives
\begin{equation}
z(E)/z(G)=i^{\langle \tau,(1+\theta_x)[(1-\delta_0)w+i]\rangle}
(-1)^{\langle\lambda,(\delta_0-1)u +i\rangle}\\  
\end{equation}
and a short computation using the identities gives
\begin{equation}
\begin{aligned}
z(E)/z(G)&=(-1)^{\langle (\ch\delta_0+\ch\theta_y)\tau, u\rangle}
(-1)^{\langle (1+\ch\theta_y)\tau, u\rangle}(-1)^{\langle\lambda,
  i\rangle}\\ 
&=(-1)^{\langle(1+\ch\delta_0)\tau,u\rangle}(-1)^{\langle \lambda,i\rangle}
\end{aligned}
\end{equation}
Next we compute
\begin{equation}
z(G)/z(E')=i^{\langle \tau-\tau',(1+\theta_x)t'\rangle}(-1)^{\langle
  \lambda-\lambda',t'\rangle} 
\end{equation}
Using \eqref{e:changelambdatau} this gives
\begin{equation}
\begin{aligned}
z(G)/z(E')&=i^{\langle (\ch\delta_0-1)\omega+\iota,
  (1+\theta_x)t'\rangle}(-1)^{\langle
  (1+\ch\theta_y)\omega,t'\rangle}\\ 
&=(-1)^{\langle
  \omega,(1+\theta_x)t'\rangle}(-1)^{\langle\iota,t'\rangle}
(-1)^{\langle \omega,(1+\theta_x)t'\rangle}\\ 
&=(-1)^{\langle\iota,t'\rangle}
\end{aligned}
\end{equation}
Multiplying (h) and (i) gives
\begin{equation}
z(E)/z(E')=
(-1)^{\langle(1+\ch\delta_0)\tau, u\rangle}(-1)^{\langle
  \lambda,i\rangle}(-1)^{\langle\iota,t'\rangle} 
\end{equation}
Multiplying both sides by $(-1)^{\langle\lambda,i\rangle\rangle}$ and
using \eqref{e:sgnEE'} gives the result. 

\end{subequations} 
\end{proof}

\subsection{Duality for extended parameters}

We offer some remarks about duality in the sense
of \cite{IC4}. Define a group
\begin{subequations}\label{se:dualreps}
(called ${}^\vee G(e(\gamma))_0$ in \eqref{e:chGz})
\begin{equation}\label{e:veeGgamma}
{}^\vee G(\gamma) = [\Cent_{{}^\vee G}(e(\gamma))]_0 \supset {}^\vee H,
\end{equation} 
a connected reductive group with root system
\begin{equation}\label{e:veeR(gamma)}
{}^\vee R(\gamma) = \{\alpha^\vee \in R^\vee \mid \langle
\gamma,\alpha^\vee\rangle \in {\mathbb Z}\},
\end{equation}
the integral roots for the infinitesimal character $\gamma$. The 
adjoint action of the representative
\begin{equation}
\ch\xi = e((\gamma-\lambda)/2){}^\vee\sigma_y {}^\vee\xi_0
\end{equation}
defines an involutive automorphism of ${}^\vee G(\gamma)$, so 
\begin{equation}
{}^\vee K_{\ch\xi} = \Cent_{{}^\vee G(\gamma)}(\ch\xi)  
\end{equation}
is a symmetric subgroup of $\ch G(\gamma)$.
By symmetry, the  parameter
$(y,x)$ defines an irreducible $({}^\vee{\mathfrak g}(\gamma),{}^\vee
K_{\ch\xi})$-module ${}^\vee J(x,y)$, with infinitesimal character $g$.
(To be precise, we need to introduce
a covering group related to the difference in $\rho$-shifts between
${}^\vee G$ and ${}^\vee G(\gamma)$, but we will overlook this
technicality.)  As in  \eqref{e:extcartan}, we find that 
\begin{equation}
\ch h\ch\delta_0=e(\tau/2){}^\vee\delta_0
\end{equation}
is a representative for an extended Cartan. As in Definition \ref{d:epsilon} 
we need to take a square root of 
\begin{equation}
\begin{aligned}
\overline\ell((\ch h\ch\delta_0)^2)=(-1)^{\langle\ell,
  (1+\ch\delta_0)\tau\rangle}
\end{aligned}
\end{equation}
which by \eqref{e:epsilons} is precisely the sign $\epsilon(x,y)$ of
Definition \ref{d:epsilon}. 

Therefore we may define an extended representation by making
$e(\tau/2){}^\vee\delta_0$ act by any desired square root of
$\epsilon$. 
It turns out that duality dictates choosing a different square root 
than we did earlier; we  choose $\zeta$ as in \eqref{e:extparamzeta}:
\begin{equation}\label{e:extparamzeta2}
\zeta(\lambda,\tau,\ell,t) = i^{\langle\tau,f\rangle}(-1)^{\langle\tau ,\ell\rangle}
=z(\lambda,\tau,\ell,t)(-1)^{\langle\lambda,t\rangle}(-1)^{\langle\ell,\tau\rangle}.
\end{equation}
Then define an extended representation $\ch J_\zeta(\lambda,\tau,\ell,t)$ by
\begin{equation}\label{e:extrepzeta}
e(\tau/2){}^\vee\delta_0 \mapsto \zeta(\lambda,\tau,\ell,t).
\end{equation}
\end{subequations}

The point of this choice of sign is that it makes the next result hold.
Recall  if  $E,E'$ are parameters for $(x,y)$ then $\sgn(E,E')$ is defined by 
the identity
$$
[J_z(E)]=\sgn(E,E')[J_z(E')],
$$
and a formula for it is given in Proposition \ref{prop:equivextended}.

\begin{proposition}\label{prop:dualequivextended}
Suppose $E,E'$ are extended parameters for $(x,y)$. 
Then 
$$
[\ch J_\zeta(E)]=\sgn(E,E')[\ch J_\zeta(E')],
$$
where $\sgn(E,E')$ is defined in Defintion \ref{d:sgn}.
Equivalently, 
\begin{equation}
\label{e:samesgn}
J_z(E)\simeq J_z(E')\text{ if and only if }
\ch J_\zeta(E)\simeq \ch J_\zeta(E').
\end{equation}
\end{proposition}

The proof is identical to that of Proposition \ref{prop:equivextended}.
What matters for us, and what is by no means automatic, is that the
sign is  the same as the sign in Definition \ref{d:sgn}.
We deduce

\begin{corollary}\label{cor:dual} In the setting \eqref{se:extparam},
  there is a natural bijection from $\delta_0$-fixed extended
  representations (of strong real forms of infinitesimal cocharacter
  $g$) of $G$, of infinitesimal character $\gamma$; to
  ${}^\vee\delta_0$-fixed representations of (strong real forms of
  infinitesimal cocharacter $\gamma$) of $\ch G(\gamma)$, of infinitesimal
  character $g$. The bijection sends $J_z(\lambda,\tau,\ell,t)$ to
  ${}^\vee J_\zeta(\lambda,\tau,\ell,t)$.
\end{corollary}
The fact that this map is well defined on equivalence classes
is precisely \eqref{e:samesgn}.
In Section \ref{sec:dual} we use this to extend the duality 
of \cite{IC4} to the twisted setting.

\bigskip

The formulations of these results are designed to allow a theoretical
analysis of all possible parameters for extended representations. For
computational purposes, one may simply want to ask when two given
parameters are equivalent. To answer that question using the results
above requires calculating elements $u$ and $\omega$ by solving their
defining equations \eqref{e:changeell} and
\eqref{e:changelambdatau}. This is not enormously difficult, but it is
not necessary. We therefore conclude this section with a simpler
formula for $\sgn(E,E')$. 

\begin{proposition}\label{prop:equivextendedprime}
Suppose $E$ and $E'$ are extended parameters for $(x,y)$. 
Then 
$$\begin{aligned}
\sgn(E,E') &= i^{\langle ({}^\vee\delta_0 -1)\lambda,t'-t\rangle + \langle
  \tau'-\tau,(\delta_0 -1)\ell'\rangle}
  (-1)^{\langle\tau,\ell'-\ell\rangle} (-1)^{\langle
    \lambda'-\lambda,t'\rangle} (-1)^{\langle\tau,t'-t\rangle}\\
&= i^{\langle\tau',(\delta_0-1)\ell'\rangle - 
    \langle\tau,(\delta_0-1)\ell\rangle}
  (-1)^{\langle\tau,\ell'-\ell\rangle} (-1)^{\langle
    \lambda'-\lambda,t'\rangle}.    
\end{aligned}$$
Here the two expressions on the right are automatically equal, and
the powers of $i$ appearing are automatically even.
\end{proposition}
The proof is similar to the proofs of Propositions
\ref{prop:equivextended} and \ref{prop:dualequivextended}. We omit the details.

\section{Hecke Algebra Action}\label{sec:hecke}
\setcounter{equation}{0}

Our goal is to compute the Hecke algebra action defined in \cite{LVq}.
We begin by summarizing the definition of this Hecke algebra module.
We then explain what extra information is needed, beyond the formulas
of \cite{LVq}*{Sections 7.5--7.7}, to carry out the computation.

In Sections \ref{sec:hecke} through \ref{sec:2Ci} we consider the case
of integral infinitesimal character. In Section \ref{sec:nonintegral}  
we discuss the modification necessary to treat the general case.

We start with our group $G$ and a pair of commuting involutions
$\delta_0,\xi_0$ as in Section \ref{sec:setting}.   
Fix a regular, integral infinitesimal character $\gamma\in\mathfrak
h^*$.  As always we assume $\gamma$ is integrally dominant as in
Definition \ref{def:Lparam}; since $\gamma$ is integral this means
$\gamma$ is dominant: $\ch\alpha(\gamma)\in\Z_{>0}$ for all
$\alpha>0$.
Let $\mathcal H$ be the twisted Hecke algebra of \cite{LVq}*{Section 4},
and set $\mathcal A=\Z[q^{\frac12},q^{-\frac12}]$
Fix a strong involution $\xi$ inner to $\xi_0$,
and set $K=K_\xi$. Associated to $\gamma$
is an $\mathcal H$-module $M$, defined in \cite{LVq}*{Section 2.3}.
In our setting this is a quotient of the Grothendieck group over
$\mathcal A$ of $(\g,{}^{\delta_0} K)$-modules with infinitesimal
character $\gamma$.  Write $[X]$ for the image in $M$ of a 
$(\g,{}^{\delta_0} K)$-module $X$.  
Let $\chi$ be the non-trivial
extension of the trivial representation of a 
one-dimensional $(\g,{}^{\delta_0} K)$-module.
In $M$ we have the relation
$$
[X]+[X\otimes\chi]\equiv 0
$$
Therefore $M$ has a basis consisting of one extension to
$(\g,{}^{\delta_0} K)$ of each irreducible  $\delta_0$-fixed
$(\g,K)$-module with infinitesimal character $\gamma$.
Furthermore if $J$ is irreducible, and 
is not the extension of an irreducible $(\g,K)$-module, then 
$[J]\equiv 0$.

Associated to a $\delta_0$-orbit $\kappa$ of 
simple roots is a generator $T_\kappa$ of $\mathcal H$.
Suppose $I$ is a standard, $\delta_0$-fixed $(\g,K)$-module with
infinitesimal character $\gamma$, and $\wt I$ is an extension of $I$
to a $(\g,{}^{\delta_0} K)$-module.  
Then formulas for $T_\kappa([\wt I])$ given in
\cite{LVq}*{Sections 7.5--7.7} are of the following form.  
There is a set $\{I_i\mid 1\le i\le n\}$ (with $n\le 3$) of standard,
$\delta_0$-fixed $(\g,K)$-modules,  
such that the appropriate formula 
\begin{equation}
\label{e:heckeformula}
T_\kappa([\wt I])=\sum_i a_i [\wt I_i]
\end{equation}
of \cite{LVq} holds  \emph{for some choices of extension of each $I_i$
  to a $(\g,{}^{\delta_0} K)$-module $\wt I_i$}.  
If we choose each extension $\wt I_i$ arbitrarily, then
\eqref{e:heckeformula} holds with a factor of $\pm1$ in front of each
term on the right.

It is natural to ask if it is possible to choose the $\wt I_i$
uniformly, so that the formulas \eqref{e:heckeformula} hold for all
$I$ and $\kappa$.
The fact that in the {\tt 2i12} and {\tt 2r21} cases there is a term with
a negative sign is a hint that this might not be the case, and it
turns out not to be possible in general.

Instead, we carry over the Hecke module structure to our extended
parameters, and compute the Hecke operators in this setting, keeping
the extra information of which extensions (i.e., signs) appear in the
formulas.  This is straightforward except when $\kappa$ is of type
{\tt 2i12, 2r21, 2Ci} or {\tt 2Cr}.

\begin{definition}
\label{d:heckemodule}
Let $\mathcal M$ be the $\mathcal A$-module spanned by 
the extended parameters of infinitesimal character $\gamma$, modulo
the relation 
$$
[E] \equiv \sgn(E,E')[E'].
$$
By \eqref{e:sgn} the map 
$[J_z(E)]\rightarrow [E]$ is a well-defined $\mathcal A$-module isomorphism.
Using this we carry over the $\mathcal H$-module structure on $M$ 
to define $\mathcal M$ as an $\mathcal H$-module.
\end{definition}

To interpret the formulas of \cite{LVq}*{Sections 7.5--7.7} in terms of
$\mathcal M$ we need the notion of Cayley transform (defined only for
certain particular $\kappa$) and cross action (defined for {\em
  every} $\kappa$) of extended parameters (defined in that
reference). The rows of Tables
\ref{t:shortcayleycross1}--\ref{t:shortcayleycross3} corresponding to
Cayley transforms are labeled {\em Cay}, and those for cross action
{\em crx}.

In addition, when $\kappa$ is of type {\tt 2i12, 2r21, 2Ci} or {\tt
  2Cr}, the formulas in \cite{LVq} make use of one more transform,
given (on the level of parameters for $G$) by the cross action of just
{\em one} of the two simple roots comprising $\kappa$. On most
parameters, this cross action will not give a $\delta_0$-fixed
parameter; like the Cayley transforms, the definition makes sense only
when $\kappa$ is of one of these four special types. The corresponding
rows of Table \ref{t:shortcayleycross2} are labeled {\em cr1x}.

These formulas are given in Tables
\ref{t:shortcayleycross1}--\ref{t:shortcayleycross3}. 
Except in the cases noted above this gives the formulas for the Hecke
algebra action (see Proposition \ref{p:heckeaction}).

Here are some notes for interpreting the tables. 
\begin{subequations}\label{se:tablenotation}
Always we start with a $\delta_0$-fixed representation
of (${}^\vee\delta_0$-fixed) infinitesimal character $\gamma$, for a
strong real form of $\delta_0$-fixed infinitesimal cocharacter $g$,
with ${\tt atlas}$ parameter $(x,y)$. Let $(\lambda,\tau,\ell,t)$ be
an extended parameter for $(x,y)$ (Definition \ref{d:extparam}).  

We also fix a ${}^\vee\delta_0$-orbit $\kappa$ on
the set of simple roots, consisting of either
\begin{equation}\label{e:kappa}
\begin{aligned}
\text{one root\ } &\{\alpha = {}^\vee\delta_0(\alpha)\} \quad &\text{\
  (type $1$); or}\\ 
\text{two roots\ } &\{\alpha,\beta={}^\vee\delta_0(\alpha)\},\quad
\langle\alpha,\beta^\vee\rangle=0 \quad &\text{\ (type {\tt 2}); or}\\
\text{two roots\ } &\{\alpha,\beta={}^\vee\delta_0(\alpha)\},\quad
\langle\alpha,\beta^\vee\rangle=-1 \quad &\text{\ (type {\tt
    3}).\phantom{ or}}\\ 
\end{aligned}
\end{equation}
We will sometimes write
\begin{equation}\label{e:kappaweight}
\kappa=_{\text{def}} \alpha+\beta \in X^*,\qquad \kappa^\vee =
\alpha^\vee + \beta^\vee \in X_*
\end{equation}
in types  {\tt 2} and {\tt 3}. (The weight $\kappa$ is a root in type {\tt 3}, but not in type {\tt 2}.)
Let 
\begin{equation}
w_\kappa=
\begin{cases}
s_\alpha&\text{type 1}\\
s_\alpha s_\beta&\text{type 2}\\
s_\alpha s_\beta s_\alpha = s_\kappa &\text{type 3}\\
\end{cases}
\end{equation}
Then $W^{\delta_0}$ is a Coxeter group with these elements as Coxeter
generators.

We will write $(x_1,y_1)$ for the {\tt atlas} parameters defining (one
of) the other $\delta_0$-fixed representations appearing in the action
of the Hecke algebra generator $T_\kappa$ on $(x,y)$, given by a
(possibly iterated) cross action or Cayley transform.  The point of
the tables is to calculate new extended parameters, denoted $E_1 =
(\lambda_1,\tau_1,\ell_1,t_1)$, for $(x_1,y_1)$ in terms of $E =
(\lambda,\tau,\ell,t)$.

Write $w_\kappa\times E$ for the cross action on $E$ (described in
the crx rows of Tables
\ref{t:shortcayleycross1}--\ref{t:shortcayleycross3}). Write 
\begin{equation}\label{e:cr1x}
w_\kappa \times_1 E \qquad (\text{$\kappa$ of type  {\tt 2i12} or {\tt 2r21}})
\end{equation}
for the element extending $s_\alpha\times (\lambda,\ell)$ defined in the cr1x
rows of Table \ref{t:shortcayleycross2}. Finally, write
\begin{equation}
c_\kappa(E) = E_\kappa\quad \text{or}\quad c_\kappa(E) = \{E_\kappa,
E_\kappa'\} 
\end{equation} 
for the (possibly multi-valued) Cayley transform defined by the Cay
rows of Tables \ref{t:shortcayleycross1}--\ref{t:shortcayleycross3}. 

 We will write
\begin{equation}\label{e:paramintegers}
\gamma_\alpha =_{\text{def}} \langle\gamma,\alpha^\vee\rangle, \quad
g_\alpha =_{\text{def}} \langle\alpha,g\rangle,
\end{equation}
and similarly for $\lambda$ and $\ell$; these quantities are all
integers. The $\delta_0$-fixed requirement means that
\begin{equation}\label{e:deltaintegers}
\gamma_\alpha = \gamma_\beta, \quad g_\alpha = g_\beta \qquad
\text{(types {\tt 2} and {\tt 3})}. 
\end{equation}
The $\delta_0$-fixed requirement on $\lambda$ and $\ell$ is more subtle, and
with the details depending on the case. For example, we have
\begin{equation}
\lambda_\alpha + \lambda_\beta = 2(\gamma_\alpha-1), \quad \ell_\alpha=\ell_\beta
\qquad \text{(type {\tt 2Ci})}
\end{equation}
A few (but not many) such conditions are recorded in the notes column.

The notes column of the tables includes additional notation
peculiar to some cases. For example, the case {\tt 1i1} corresponds to
a discrete series in a block for $A_1$ with two discrete series and
just one principal series. This turns out to mean that the root
$\alpha$ must be trivial on the fixed points $H^{\theta_{x_1}}$ for the
 more split Cartan; and this in turn is equivalent to the existence of
  $\sigma\in X^*(H)$ so that
\begin{equation}\label{e:zeta}
\alpha = (1+{}^\vee\theta_{y_1})\sigma.
\end{equation} 
That is the meaning of the note in the {\tt 1i1} row; the weight
$\sigma$ (which one needs to find by solving \eqref{e:zeta} to
implement the algorithm) appears in the formula for $\tau_1$.

\end{subequations} 

The terminology here is more compact than that of \cite{LVq}. See Table 1.

\begin{proposition}
\label{p:heckeaction}
Suppose $(x,y)$ is a $\delta_0$-fixed parameter, and $E$ is an extended parameter for $(x,y)$. 
Let $\kappa$ be a $\delta_0$-orbit of simple roots. 

Suppose $\kappa$ is {\it not} of type {\tt 2i12, 2r21, 2Ci} or {\tt
  2Cr}. Then  the formulas for the action of the Hecke operator
$T_\kappa$ from \cite{LVq} apply, using the Cayley transforms and
cross actions from Tables 2-4, to give a formula for $T_\kappa([E])$.
\end{proposition}

This is a direct translation of the calculations of
\cite{LVq}*{Sections 7.5--7.7} to our setting. 
We treat the excluded cases in the next two sections.

\begin{example}
  Suppose $\kappa$ is a \emph{two-imaginary noncompact type I-I
    ascent} for $(x,y)$ (in the terminogy of \cite{LVq}), i.e., of
  type {\tt 2i11} (in our terminology).  Suppose $E_1$ is an extended
  parameter for $(x,y)$, and set $E_2=w_\kappa\times E_1$, and set
  $E'=c_\kappa(E_1)$, as defined by Table 3.  Then formula
  \cite{LVq}*{(7.6)(e$'$)} gives:
$$
\begin{aligned}
T_\kappa([E_1])&=[E_2]+[E']\\
T_\kappa([E_2])&=[E_1]+[E']\\
T_\kappa([E'])&=(q-1)([E_1]+[E_2])+(q-2)[E']
\end{aligned}
$$
\end{example}

\section{The {\tt 2i12} case}\label{sec:2i12}
\setcounter{equation}{0}

If $\kappa$ is of type {\tt 2i12} or {\tt 2r21} then the formulas
for the Hecke operator $T_\kappa$ do not carry over directly from \cite{LVq}*{(7.6)(i$''$) and (j$''$)}.   
We start with a special case.

\begin{lemma}
Suppose $\kappa$ is of type {\tt 2i12} for $E_0=(\lambda,\tau,\ell,t)$.
Assume
\begin{equation}
\begin{aligned}
\tau_\alpha=\tau_\beta&=0\\
t_\alpha=t_\beta&=0\\
g_\alpha-\ell_\alpha&=g_\beta-\ell_\beta=1\\
\gamma_\alpha-\lambda_\alpha&=\lambda_\beta-\lambda_\beta=1
\end{aligned}
\end{equation}
Let $E_0'=w_\kappa\times_1 E_0$ as given by Table 3. (Recall that this
is a certain extension of the parameter $s_\alpha \times (\lambda,\ell)$ for
$G$ 
The Cayley transform $c_\kappa(E_0)$ is double-valued; write $c_\kappa(E_0)=\{F_0,F_0'\}$,
where the parameter for $F_0$ is  $(\lambda,\tau,\ell,t)$, and  $F_0'=w_\kappa\times_1 F_0$.

The action of $T_\kappa$ on the space spanned by $E_0,E_0',F_0,F_0'$ is
\begin{equation}
\begin{aligned}
T_\kappa(E_0) &= E_0 + F_0 + F_0'\\
T_\kappa(E_0') &= E_0' + F_0-F_0'\\
T_\kappa(F_0) &= (q^2-1)(E_0+E_0') + (q^2-2)F_0\\
T_\kappa(F_0') &= (q^2-1)(E_0-E_0') + (q^2-2)F_0'
\end{aligned}
\end{equation}
\end{lemma}

Explicitly the extended parameters are:
\begin{equation}
\label{e:2i12standard}
\begin{aligned}
&E_0:(\lambda,\tau,\ell,t)\quad &E_0':(\lambda,\tau,\ell+\alpha^\vee,t-s)\\
&F_0:(\lambda,\tau,\ell,t)\quad &F_0':(\lambda+\alpha,\tau-\sigma,\ell,t)
\end{aligned}
\end{equation}
with $\sigma$ and $s$ given in Table 3.

When the parameters are in this form, this is simply a 
direct translation of the proof of \cite{LVq}*{(7.6)(i$''$)}.

\begin{lemma}
Suppose $E$ is an extended parameter for $(x,y)$, and 
$\kappa$ is of type {\tt 2i12} for $E$. Set $E'=w_\kappa\times_1 E$. 
Write $c_\kappa(E)=c_\kappa(E')=\{F,F'\}$. Possibly after switching $E$ and $E'$, 
and possibly also switching $F$ and $F'$, we can find $E_0,E_0',F_0
,F_0'$ as in the previous Lemma,
such that $E$ and $E_0$ are extensions of the same parameter, and similarly
$(E',E'_0),(F,F_0)$ and $(F',F'_0)$. 
\end{lemma}

\begin{proof}
Write $E=(\lambda,\tau,\ell,t)$, so $E'=(\lambda,\tau,\ell+\alpha,t-s)$.
After replacing $\tau$ with a different solution of its defining equation:
$$
\tau\rightarrow\tau+ \tau_\beta\sigma+\frac12(\tau_\alpha+\tau_\beta)\alpha
$$
where $\alpha-\beta=(1+\ch\theta_{y_1})\sigma$, we can assume 
$\tau_\alpha=\tau_\beta=0$.

Since $2\alpha^\vee,2\beta^\vee$ and $\alpha^\vee-\beta^\vee$ are all in $(1+\theta_x)X_*$,
$a\alpha^\vee+b\beta^\vee$ is in $(1+\theta_x)X_*$ provided $a+b\in 2\Z$.
By adding such a term to $\ell$ we can arrange that
$g_\alpha-\ell_\alpha-1=0$ and $g_\beta-\ell_\beta-1=0$ or $2$. 
Make the corresponding change $t\rightarrow t+\frac12(b-a)(\alpha^\vee-\beta^\vee)$. 
If $g_\beta-\ell_\beta-1=2$, 
replace $E$ with 
$E' = w_\kappa\times_1 E$, and now we have
$$
g_\alpha-\ell_\alpha-1=g_\beta-\ell_\beta-1=0.
$$
Since $g_\alpha=g_\beta$ this implies $\ell_\alpha=\ell_\beta$.
Then  Conditions (a) and (d) of  Definition \ref{d:extparam} imply
$\lambda_\alpha-\gamma_\alpha-1=\lambda_\beta-\gamma_\beta-1=0$ 
and $t_\alpha=t_\beta=0$.
Table 3
then says that $F=(\lambda,\tau,\ell,t)$ is one of the two Cayley transforms of $E$,
and that our parameters now have the form \eqref{e:2i12standard}.
\end{proof}

\begin{proposition}
\label{p:heckeaction2i12}
In the setting of the previous Lemma we have
$$
\begin{aligned}
T_\kappa(E)&=E+\sgn(E,E_0)(\sgn(F,F_0)F+\sgn(F',F'_0)F')\\
T_\kappa(E')&=E'+\sgn(E',E'_0)(\sgn(F,F_0)F-\sgn(F',F'_0)F')\\
T_\kappa(F) &= (q^2-1)\sgn(F,F_0)(\sgn(E,E_0)E+\sgn(E',E'_0)E') + (q^2-2)F\\
T_\kappa(F') &= (q^2-1)\sgn(F',F'_0)(\sgn(E,E_0)E-\sgn(E',E'_0)E') + (q^2-2)F'\\
\end{aligned}
$$
This formula is independent of the choice of $E_0,E'_0,F_0,F'_0$.  
\end{proposition}

This is immediate.

\section{The {\tt 2Ci} case}\label{sec:2Ci}
\setcounter{equation}{0}

Now we describe the Hecke algebra action in the {\tt 2Ci} case.
So fix a type $2$
root $\kappa = \{\alpha,\beta\}$, and an extended parameter
\begin{subequations}\label{se:2CiHecke} 
\begin{equation}\label{e:2Ci1}
 E = (\lambda,\ \tau,\ \ell,\ t)
\end{equation}
as in \ref{d:extparam}. Assume that $\kappa$ is of type {\tt
  2Ci} for $E$: that is, that $\alpha$ and $\beta$ are complex roots
interchanged by 
\begin{equation}\label{e:2Ci2}
\theta_x = \Ad(e((g -\ell)/2)\sigma_w\xi_0).
\end{equation}
This means in turn that
\begin{equation}
w\xi_0\alpha = \beta,\qquad w\xi_0\beta = \alpha.
\end{equation}
Proposition \ref{prop:signs} says that
\begin{equation}
\sigma_w\xi_0 X_\alpha = X_\beta, \qquad \sigma_w\xi_0 X_\beta = X_\alpha,
\end{equation}
and therefore that
\begin{equation}\label{e:theta2Cisign}
\theta_x(X_\alpha) = (-1)^{g_\beta - \ell_\beta}X_\beta, \quad
\theta_x(X_\beta) = (-1)^{g_\alpha - \ell_\alpha}X_\beta. 
\end{equation} 
The requirement \eqref{e:inffixed} implies that
\begin{equation}
\gamma_\alpha = \gamma_\beta, \qquad g_\alpha = g_\beta.
\end{equation}
Similarly, the requirements for an extended parameter to be
$\delta_0$-fixed imply among other things that
\begin{equation}\label{e:deltafixed2Ci}
\lambda_\alpha + \lambda_\beta = 2(\gamma_\alpha - 1), \quad
\lambda_\alpha - \lambda_\beta = \tau_\beta - \tau_\alpha,\quad
\ell_\alpha =\ell_\beta, \quad t_\alpha = 
-t_\beta. 
\end{equation}
In particular, we can define a sign
\begin{equation}\label{e:epsilon2Ci}
\epsilon = \epsilon(E) = (-1)^{g_\alpha - \ell_\alpha} = (-1)^{g_\beta
  - \ell_\beta}.
\end{equation}
Writing 
\begin{equation}
{\mathfrak g} = {\mathfrak k} \oplus {\mathfrak s}
\end{equation}
for the eigenspace decomposition under $\theta_x$, we get from
\eqref{e:theta2Cisign} 
\begin{equation}
X_\alpha + \epsilon X_\beta=_{\text{def}} X_{\mathfrak k} \in
{\mathfrak k}, \qquad  X_\alpha - \epsilon X_\beta =_{\text{def}}
X_{\mathfrak s} \in {\mathfrak s}.
\end{equation}
and also
\begin{equation}
\theta_x(\sigma_\alpha) = \sigma_\beta^\epsilon.
\end{equation}
The Weyl group element $s_\alpha s_\beta$ is represented by
\begin{equation}\label{e:sigmaE}
\sigma_E = \sigma_\alpha \sigma_\beta^\epsilon \in G^{\theta_x} = K.
\end{equation}
Finally, the extended group ${}^{\delta_0} K$ is generated by $K$ and
the element
\begin{equation}\label{e:h2Ci}
h=e(t/2)\delta_0
\end{equation}
(Definition \ref{d:extparam}(e)).

Table \ref{t:shortcayleycross2} constructs from $E$ a second extended
parameter
\begin{equation}\label{e:2Cr1}
\begin{aligned}
E_1 = (&\lambda_1,\tau_1,\ell_1,t_1) \\ = (&s_\alpha\lambda +
(\gamma_\alpha - 1)\alpha, s_\alpha\tau + (\lambda_\alpha -
\gamma_\alpha + 1)\alpha, \\
&s_\alpha\ell + (g_\alpha - 1)\alpha^\vee, s_\alpha t + (\ell_\alpha -
g_\alpha + 1)\alpha^\vee).
\end{aligned} 
\end{equation}
The root $\kappa$ is of type {\tt 2Cr} for the parameter $E_1$. The
element $\ell_1$ is chosen so that the corresponding Cartan involution
(on all of $G$, not just $H$) is 
\begin{equation}
\theta_{x_1} = \sigma_\alpha^{-1}\theta_x \sigma_\alpha.
\end{equation}
\end{subequations} 

\begin{proposition}\label{prop:2Cicorr}
Suppose we are in the setting \eqref{se:2CiHecke}. 
\begin{enumerate}
\item Applying the formula in Table \ref{t:shortcayleycross2} to the
  {\tt 2Cr} parameter $E_1$ gives exactly the same parameter $E$ with
  which we started. 
\item The action of the Hecke algebra
  generator $T_\kappa$ (\cite{LVq}*{7.6($c''$)} is
$$T_\kappa(E) = qE + (-1)^{[(\tau_\alpha + \tau_\beta)/2](g_\alpha -
  \ell_\alpha -1)}(q+1)E_1$$
Here the sign may be regarded as specifying a renormalization
of $E_1$ (whose existence is asserted in \cite{LVq}). 
\item The corresponding formula for the case {\tt 2Cr} is
$$T_\kappa(E_1) = (q^2-q-1)E_1 + (-1)^{(\gamma_{1,\alpha} -
  \lambda_{1,\alpha} + \tau_{1,\alpha} -1)[(t_{1,\beta} -
  t_{1,\alpha})/2]}(q^2-q)E.$$
The sign is exactly the same as the one for $T_\kappa(E)$, written in
terms of the parameter $E_1$. 
\end{enumerate}
\end{proposition}
\begin{proof}
The first assertion can be verified by applying the formulas for
passing from {\tt 2Ci} to {\tt 2Cr} and from {\tt 2Cr} to {\tt 2Ci} in
succession, then simplifying; we omit the details. 

For the second assertion, we need to understand
representation-theoretically the relationship between the extended parameter
parameters $E$ and $E_1$, and how this relates to the Hecke algebra
action. For this question it is easiest to think of $({\mathfrak
  g},K)$-modules with a fixed $K$; that is, to conjugate
$\theta_{x_1}$ back to $\theta_x$, and to correspondingly change $E_1$
into a parameter 
\begin{subequations}\label{se:2Ci2Cr}
\begin{equation}
\begin{aligned}
E_2 = (&\lambda_2,\tau_2,\ell_2,t_2) \\
= (&\lambda - (\gamma_\alpha - 1)\alpha, \tau - (\lambda_\alpha -
\gamma_\alpha + 1)\alpha, \\
&\ell - (g_\alpha - 1)\alpha^\vee, t - (\ell_\alpha -
g_\alpha + 1)\alpha^\vee) 
\end{aligned}
\end{equation}
related to the Borel subgroup
\begin{equation}\label{e:Balpha}
B' = \sigma_\alpha B \sigma_\alpha^{-1}.
\end{equation}
(The {\tt atlas} decision to prefer $E_1$ to $E_2$ is just a
bookkeeping convenience. Everything about representation theory, and
also most things about perverse sheaves, are calculated with a fixed
Cartan involution, and so refer to the relationship between $E$ and
$E_2$. The {\tt atlas} formulas for Hecke algebra actions index bases
by $E_1$ rather than $E_2$, so we will occasionally mention $E_1$
below; but mostly we will be concerned about $E$ and $E_2$.)

The distinguished automorphism corresponding to $\delta_0$ for $B'$ is
\begin{equation}
\delta'_0 = \sigma_\alpha\delta_0 \sigma_\alpha^{-1} = \sigma_\alpha
\sigma_\beta^{-1} \delta_0.
\end{equation}
The generator for the extended Cartan defined by $E_2$ is
\begin{equation}\label{e:2Cih2}
h_2 = e(t_2/2)\delta'_0 = e(t/2)m_\alpha^{\ell_\alpha - g_\alpha + 1}
  \sigma_\alpha \sigma_\beta^{-1}\delta_0 = e(t/2)\sigma_E^{-\epsilon}\delta_0
\end{equation}
(with $\sigma_E \in K$ as in \eqref{e:sigmaE}.
Now we can start to talk about representation theory: that is, about
$({\mathfrak g}, K)$-modules $M$ and their extensions to $({\mathfrak
g},{}^{\delta_0} K)$-modules $M'$. Write 
\begin{equation}
P = LU \supset B,\ B'
\end{equation}
 for the parabolic subgroup with $L$ generated by $H$ and the simple
 roots $\alpha$ and $\beta$.  Then
\begin{equation}
M_j = H_j({\mathfrak u},M), \qquad M'_j = H_j({\mathfrak u},M')
\end{equation}
are $({\mathfrak l},L\cap K)$- and $({\mathfrak l},{}^{\delta_0}(L\cap
K))$-modules respectively; and the relationship between
representations and parameters (which uses ${\mathfrak n}$-homology)
factors through this construction by means of the Hochschild-Serre
spectral sequence. In this way (omitting details) one can reduce the
questions we are studying to the case
\begin{equation}\label{e:so31}
G=L,\quad R= \{\pm\alpha,\pm\beta\}.
\end{equation}
In the setting \eqref{e:so31}, here is what the representation theory
looks like. The group $L$ is locally $SL(2) \times SL(2)$, and $K$ is
approximately a ``diagonal'' copy of $SL(2)$. (More precisely, the
``diagonal'' copy is 
\begin{equation}\label{e:SL2K}
SL(2)_K = \left\{\left(\begin{pmatrix} a&b\\ c&d\end{pmatrix},
 \begin{pmatrix} a&\epsilon b\\ \epsilon c&d\end{pmatrix}\right)  \middle|
\begin{pmatrix} a&b\\ c&d\end{pmatrix} \in SL(2,{\mathbb C})
\right\},
\end{equation}
with $\epsilon$ as in \eqref{e:epsilon2Ci}. Furthermore $K$, and even
its intersection with the derived group of $L$, may be disconnected.

Attached to $E$ is an irreducible principal series  $({\mathfrak
  g},{}^{\delta_0}K)$-module $I(E)$. The restriction of $I(E)$ to
${}^{\delta_0} K$ is 
\begin{equation}
I(E) = \Ind_{{}^{\delta_0} (H\cap K)}^{{}^{\delta_0} K}(\Lambda_E + 2\rho_n).
\end{equation}
Here $\Lambda_E$ is the character of $H\cap K$ defined by the first
term $\lambda$ in $E$, extended to ${}^{\delta_0} (H\cap K)$ by making
$h$ (from \eqref{e:h2Ci}) act by \eqref{e:extparamz}. The twist
$2\rho_n$ is the character by which ${}^{\delta_0} (H\cap K)$ acts on
(the top exterior power of) ${\mathfrak n}\cap {\mathfrak s}$; that
is, on the vector $X_{\mathfrak s}$ from \eqref{e:theta2Cisign}.

The reason for the last twist is that for fundamental series
modules $M$, the character of $H\cap K$ in the parameter
is a weight on $H_{\dim({\mathfrak n}^{\op}\cap{\mathfrak s})}({\mathfrak
  n}^{\op},M)$, specifically appearing in the image of a natural map
$$H_0({\mathfrak n}^{\op}\cap {\mathfrak k}, M) \otimes
{\textstyle\bigwedge^{\dim({\mathfrak n}^{\op}\cap{\mathfrak s})}} \left({\mathfrak
    n}^{\op}\cap{\mathfrak s} \right) 
\rightarrow H_{\dim({\mathfrak n}^{\op}\cap{\mathfrak s})}({\mathfrak
  n}^{\op},M).$$ 
The conclusion is that {\em the weight of $H\cap K$ on the parameter is
equal to the ${\mathfrak n}\cap{\mathfrak k}$-highest weight of the
lowest $K$-type, minus $2\rho_n$}. The coroot for $K$ is
\begin{equation}\label{e:2CiKcoroot}
\alpha^\vee + \beta^\vee,
\end{equation}
which acts on $X_{\mathfrak s}$ by $2$. The dimension of the lowest $K$-type is
is therefore
\begin{equation}
\lambda_\alpha + \lambda_\beta + 2 + 1 = 2\gamma_\alpha+ 1;
\end{equation}
the $1$ comes from the $\rho$-shift in the Weyl dimension formula, and
we have used \eqref{e:deltafixed2Ci} to convert $\lambda$ to
$\gamma$. In particular, we find that
\begin{equation}
I(E)|_{SL(2)_K} = \text{sum of irreducibles of dimensions
  $2\gamma_\alpha+1$, $2\gamma_\alpha + 3$,\dots}.
\end{equation}

In the same fashion, attached to $E_2$ is a reducible principal series
$({\mathfrak g},{}^{\delta_0}K)$-module $I(E_2)$. The restriction of
$I(E_2)$ to ${}^{\delta_0} K$ is 
\begin{equation}
I(E_2) = \Ind_{{}^{\delta_0} (H\cap K)}^{{}^{\delta_0} K}(\Lambda_{E_2}).
\end{equation}
The reason for the absence of a twist on $\Lambda_{E_2}$ is that for
principal series modules $M_2$ for quasisplit groups, the parameter
appears as a weight on $H_0({\mathfrak n}^{\op},M_2)$; and for
(almost) spherical representations, this weight space is precisely the
image of the (almost) spherical vector.  In particular,
\begin{equation}
I(E_2)|_{SL(2)_K} = \text{sum of irreducibles of dims
  $1$, $3$,\dots}.
\end{equation}

The principal series representation $I(E_2)$ has a unique irreducible
quotient representation $J(E_2)$:
\begin{equation}
\begin{aligned}
J(E_2)|_{SL(2)_K} &= \text{sum of irreducibles of dims
  $1$, $3$,\dots, $2\gamma_\alpha -1$},\\
\dim J(E_2) &= \gamma_\alpha^2.
\end{aligned}
\end{equation}
We get a short exact sequence
\begin{equation}
\begin{aligned}
0 \rightarrow &I(E') \rightarrow I(E_2) \rightarrow J(E_2) \rightarrow 0,\\
&I(E')|_{SL(2)_K} = \text{sum of irreducibles of dims $2\gamma_\alpha+1$,
  $2\gamma_\alpha + 3$,\dots}.
\end{aligned}
\end{equation}
\end{subequations} 
{\em This extended parameter $E'$ with $I(E')$ appearing as a
  composition factor of $I(E_2)$ is the one on which the Hecke algebra
  action gives $E_1$ (remember that this is essentially just another
  label for $E_2$) with positive coefficient.} (This is a
consequence of the Beilinson-Bernstein localization theory relating
perverse sheaves to representations, and the perverse sheaf definition
of the Hecke algebra action in \cite{LVq}.)  So we need to understand
the relationship between the extended parameters $E$ and $E'$.

\begin{subequations}\label{se:2Cicorrection}
Because the spherical composition factor $J(E_2)$ is a unique
quotient of $I(E_2)$, the spherical vector in $I(E_2)$ is cyclic. The
action of $X_{\mathfrak s}$ carries highest weight vectors for $K$ to
highest weight vectors for $K$; so we deduce
\begin{multline}\label{e:2Cichange}
X_{\mathfrak s}^{\gamma_\alpha}(\text{spherical
    vector in $I(E_2)$}) \\
= \text{highest weight vector for lowest $K$-type of
  $I(E')$}.\end{multline}

Because $\sigma_E \in SL(2)_K$ acts trivially on the (one-dimensional)
lowest $K$-type of $J(E_2)$, the formula \eqref{e:2Cih2} shows that
\begin{equation}\label{e:E2LKT}
\Lambda_{E_2}(h_2) = \text{action of $e(t/2)\delta_0$ on $J(E_2)$ lowest
  $K$-type}.
\end{equation}
It is easy to calculate
$$\Ad(e(t/2)\delta_0)(X_{\mathfrak s}) = -\epsilon(-1)^{t_\alpha} =
(-1)^{g_\alpha - \ell_\alpha -1 + t_\alpha}.$$
Combining \eqref{e:E2LKT} with \eqref{e:2Cichange}, we find
\begin{equation}\label{e:EprimeLKT}
\begin{aligned}
&\ \text{action of $h = e(t/2)\delta_0$ on $I(E')$ lowest $K$-type} \\
&= \Lambda_{E_2}(h_2)(-\epsilon)^{\gamma_\alpha}(-1)^{\gamma_\alpha
  t_\alpha}\\
&= \Lambda_{E_2}(h_2)(-1)^{\gamma_\alpha((g_\alpha - \ell_\alpha
  -1)+t_\alpha)}.
\end{aligned}
\end{equation}
Using the description of the parameter for $E'$ given before
\eqref{e:2CiKcoroot}, we get
\begin{equation}\label{e:2Cidesiredchange}
\Lambda_{E'}(h) = \Lambda_{E_2}(h_2)(-1)^{(\gamma_\alpha-1)((g_\alpha
  - \ell_\alpha -1)+t_\alpha)} 
\end{equation}

Now we need to compare this ``desired'' relationship between
$\Lambda_{E'}(h)$ and $\Lambda_{E_2}(h_2)$ with the actual
relationship between $\Lambda_E(h)$ and $\Lambda_{E_2}(h_2)$. We find
(using \eqref{e:extparamz} and the formulas in Table
\ref{t:shortcayleycross2} for $E_1$) 
\begin{equation}\label{e:2Ciactualchange}
\begin{aligned}
\Lambda_E(h) \Lambda^{-1}_{E_2}(h_2) &= \Lambda_E(h)
\Lambda^{-1}_{E_1}(h_1) \\[.2ex]
&=\  i^{\langle\tau,(\delta_0 -
  1)\ell\rangle}(-1)^{\langle\lambda,t\rangle}\\
&\qquad i^{-\langle\tau-[(\tau_\alpha +\tau_\beta)/2]]\alpha,(\delta_0 -
  1)(\ell+(g_\alpha-\ell_\alpha -1)\alpha^\vee\rangle}\\
&\qquad (-1)^{\langle(\lambda +(\gamma_\alpha - \lambda_\alpha -1)\alpha ,t +
  (\ell_\alpha - g_\alpha - t_\alpha + 1)\alpha^\vee\rangle} \\[.2ex]
&=\ i^{\langle[(\tau_\alpha +\tau_\beta)/2]\alpha, (\delta_0 -
  1)(\ell+(g_\alpha-\ell_\alpha -1)\alpha^\vee\rangle}\\
&\qquad i^{-\langle\tau, (\delta_0 - 1)(\ell+(g_\alpha-\ell_\alpha
  -1)\alpha)\rangle}\\
&\qquad (-1)^{\langle(\gamma_\alpha - \lambda_\alpha -1)\alpha, t +
  (\ell_\alpha - g_\alpha - t_\alpha + 1)\alpha^\vee\rangle} \\
&\qquad (-1)^{\langle \lambda,(\ell_\alpha - g_\alpha - t_\alpha + 1)\alpha^\vee\rangle} 
\end{aligned}
\end{equation}
There are four factors on the right. In the first,
$$\langle\alpha,(\delta_0-1)\ell\rangle = \ell_\alpha - \ell_\beta = 0$$
by \eqref{e:deltafixed2Ci}. In the third,
$\langle\alpha,\alpha^\vee\rangle = 2$ contributes an even power of
$(-1)$, so can be dropped. We are left with
\begin{equation}\label{e:2Ciactualchange2}
\begin{aligned}
\Lambda_E(h) \Lambda^{-1}_{E_2}(h_2) &=\ i^{\langle[(\tau_\alpha
  +\tau_\beta)/2]\alpha, (\delta_0 -  1)((g_\alpha-\ell_\alpha
  -1)\alpha^\vee)\rangle}
i^{-\langle\tau, (\delta_0 - 1)((g_\alpha-\ell_\alpha
  -1)\alpha)\rangle}\\
&\qquad (-1)^{\langle(\gamma_\alpha - \lambda_\alpha -1)\alpha,
  t\rangle} (-1)^{\langle \lambda,(\ell_\alpha - g_\alpha - t_\alpha +
    1)\alpha^\vee\rangle}\\[.2ex]
&=\ (-1)^{[(\tau_\alpha +\tau_\beta)/2](g_\alpha-\ell_\alpha -1)]}
(-1)^{[(\tau_\alpha - \tau_\beta)/2](g_\alpha-\ell_\alpha -1)}\\
&\qquad (-1)^{(\gamma_\alpha - \lambda_\alpha -1)t_\alpha} (-1)^{
  \lambda_\alpha(\ell_\alpha - g_\alpha - t_\alpha + 1)}\\[.2ex]
&=\ (-1)^{\tau_\alpha(g_\alpha-\ell_\alpha -1)} (-1)^{(\gamma_\alpha -
  \lambda_\alpha -1)t_\alpha} (-1)^{\lambda_\alpha(\ell_\alpha -
  g_\alpha - t_\alpha +  1)}\\ 
\end{aligned}
\end{equation}
Splitting the last factor between the first two gives
\begin{equation}\label{e:2Ciactualchange3}
\begin{aligned}
\Lambda_E(h) \Lambda^{-1}_{E_2}(h_2) &=\ (-1)^{(\lambda_\alpha +
  \tau_\alpha)(g_\alpha-\ell_\alpha -1)}\\ 
&\qquad (-1)^{(\gamma_\alpha -1)t_\alpha}
\end{aligned}
\end{equation}
Now use the first two formulas from \eqref{e:deltafixed2Ci} to write
$\lambda_\alpha = (\gamma_\alpha - 1) + (\tau_\beta -
\tau_\alpha)/2$. We get
\begin{equation}\label{e:2Ciactualchange4}
\begin{aligned}
\Lambda_E(h) \Lambda^{-1}_{E_2}(h_2) &= (-1)^{[(\gamma_\alpha -
  1)+(\tau_\alpha + \tau_\beta)/2](g_\alpha-\ell_\alpha -1)}
(-1)^{(\gamma_\alpha -1)t_\alpha}\\ 
&= (-1)^{(\gamma_\alpha - 1)(g_\alpha-\ell_\alpha +t_\alpha -1)}
(-1)^{[(\tau_\alpha + \tau_\beta)/2](g_\alpha-\ell_\alpha -1)}
\end{aligned}
\end{equation}
The first factor here is exactly the one from
\eqref{e:2Cidesiredchange}, so we deduce
\begin{equation}
\Lambda_{E}(h) = \Lambda_{E'}(h) (-1)^{[(\tau_\alpha +
  \tau_\beta)/2][g_\alpha - \ell_\alpha -1]}.
\end{equation}
The sign on the right has to appear in front of the \cite{LVq} Hecke
algebra formula for the coefficient of $E_1$ in $T_\kappa E$. This
proves the second assertion of the proposition. For the third, we just
rewrite exactly the same formula in terms of the parameter $E_1$; by
the first assertion of the proposition, the formulas in Table
\ref{t:shortcayleycross2} tell us how to do that. We omit the
algebraic details.
\end{subequations} 
\end{proof}

We summarize the results of Sections \ref{sec:hecke}--\ref{sec:2Ci}.

\begin{theorem}
  If the infinitesimal character $\gamma$ is integral, then the action
  of $\mathcal H$ on $\mathcal M$ (Definition \ref{d:heckemodule}) is
  given by Propositions \ref{p:heckeaction}, \ref{p:heckeaction2i12}
  and \ref{prop:2Cicorr}.
\end{theorem}

\section{Nonintegral Infinitesimal character}
\label{sec:nonintegral}
\setcounter{equation}{0}

Suppose the infinitesimal character $\gamma$ is not necessarily
integral. As always we assume it is integrally dominant \eqref{e:id}.
Set
\begin{subequations}
\renewcommand{\theequation}{\theparentequation)(\alph{equation}}  
\begin{equation}
\ch R(\gamma)=\{\ch\alpha\in \ch R\mid \langle \gamma,\ch\alpha\rangle\in\Z\}
\end{equation}
as in \eqref{e:veeR(gamma)}, and set 
\begin{equation}
\begin{aligned}
R(\gamma)&=\{\alpha\in R\mid \ch\alpha\in \ch R(\gamma)\}\\
R(\gamma)^+&=R^+\cap R(\gamma).
\end{aligned}
\end{equation}
We say  $\alpha\in R$ is {\em integral} if 
$\alpha\in R(\gamma)$.
We say an integral root is {\em simple} (respectively {\em integral-simple})
if it is simple for $R^+$ (respectively $R(\gamma)^+$).

The Weyl group $W(\gamma)$ of $R(\gamma)$ satisfies:
\begin{equation}
W(\gamma)=\{w\in W\mid w\gamma-\gamma\in {\mathbb Z}R\}.
\end{equation}

We now assume $\ch\delta_0(\gamma)=\gamma$ (see Lemma
\ref{lemma:deltagamma}), so $\ch\delta_0$ acts on $R(\gamma)$.
Then  $\ch\delta_0$ preserves both the simple and integral-simple roots, 
so the notions of integral and
integral-simple apply to a $\ch\delta_0$-orbit
$\kappa=\{\alpha,\ch\delta_0(\alpha)\}$ of roots.
Let $\mathcal H(\gamma)$ be the Hecke algebra 
of \cite{LVq}*{(4.7)} applied to $(R(\gamma),\delta_0)$. 

Let $\mathcal M_\gamma$ be the module of Definition
\ref{d:heckemodule}. The construction of \cite{LVq} gives a
representation of $\mathcal H(\gamma)$ on $\mathcal M_\gamma$. (More
precisely, the construction of \cite{LVq} concerns geometry related by
base change (to compare a base field of finite characteristic with
${\mathbb C}$) and Beilinson-Bernstein localization (to relate
$K$-equivariant perverse sheaves to $({\mathfrak g},K)$-modules) to
the module of Definition \ref{d:heckemodule}.  In order to make a
parallel identification in the case of nonintegral infinitesimal
character, one needs a discussion like that in \cite{ABV}*{Chapter
 17}. We omit the details.)

\end{subequations}

Suppose $\kappa$ is a $\ch\delta_0$-orbit of roots that are
integral (for $\gamma$) and simple (for $G$).  Then the
formulas of Tables 2--4 apply to give a formula for the action
$T_\gamma$ on $\mathcal M$.  The technical issue we have to deal with
here is what to do if $\kappa$ is integral-simple (for $\gamma$) but
not simple (for $G$).

\begin{definition}
\label{d:id}
Let $\ID$ be the set of integrally dominant elements of $\mathfrak
h^*$:
$$
\ID=\{\gamma\in\mathfrak h^*\mid\alpha\in  R(\gamma)^+\implies
\langle\gamma,\ch\alpha\rangle \ge 0\}
$$

If $\gamma\in \mathfrak h^*$ then $\gamma$ is $W(\gamma)$-conjugate to
a unique element of $\ID$.
If $\gamma\in \ID$ and $w\in W$ let $w*\gamma$ be the unique element
of $\ID$ 
which is $W(w\gamma)$ conjugate to $w\gamma$.
\end{definition}

It is easy to see 
that $w*\gamma$ is the unique element satisfying
\begin{enumerate}[label=(\alph*)]
\item $w*\gamma\in\ID$
\item $w*\gamma$ is $W$-conjugate to $\gamma$
\item $w*\gamma\in w\gamma+{\mathbb Z}R$.
\end{enumerate}
Condition (c) is equivalent (in the presence of (a) and (b)) to  
\begin{enumerate}
\item[(c$'$)] $w*\lambda=xw\lambda\text{ for some }x\in W(w\lambda)$.
\end{enumerate}

\begin{lemma}
\label{l:id2}
The map $(w,\gamma)\rightarrow w*\gamma$ is an action of $W$
on $\ID$. 
It satisfies:
\begin{enumerate}
\item $\text{Stab}_W(\gamma)=W(\gamma)$ ;
\item The $W$-orbit of $\gamma$ under $*$ is in bijection with $W/W(\gamma)$;
\item $w*\gamma=xw\gamma$ for some  $x\in W(w\gamma)$;
\item Suppose $\alpha$ is simple for $R^+$. Then 
$$
s_\alpha*\gamma=
\begin{cases}
\gamma&\alpha\in R(\gamma)\\
s_\alpha(\gamma)&\alpha\not\in R(\gamma)  
\end{cases}
$$

\end{enumerate}
\end{lemma}

\begin{proof}
If $x,y\in W$ then 
$(xy)*\lambda$ is the unique element satisfying conditions (a--c)
above with respect to $xy$. 
On the other hand $x*(y*\lambda)$ obviously satisfies (a) and (b). 
Condition (c) holds as well:
$$
\begin{aligned}
x*(y*\gamma)&\in x(y*\gamma)+{\mathbb Z}R\\
&           \in x(y\gamma+{\mathbb Z}R)+{\mathbb Z}R
&          =(xy)\gamma+{\mathbb Z}R 
\end{aligned}
$$
Assertions (1--3) are straightforward, and (4) is clear if $\alpha$ is
integral for $\gamma$, so assume this is not the case.
Obviously  $s_\alpha(\gamma)$ satisfies
the 
conditions (b) and (c) for $s_\alpha*\gamma$, and (a) follows from the
fact that $s_\alpha$ permutes $R^+-\{\alpha\}$. We leave the details
to the reader.
\end{proof}

If $\gamma$ is integral the formulas for the cross action in  Tables 2--4 define an action of
$W^{\delta_0}$ on $\mathcal M$. With a small change the same holds in general. 
To indicate the role of $\gamma$, write 
 $(\lambda,\tau,\ell,t,\gamma)$ for an extended parameter. 

\begin{definition}
\label{d:nonintegralcross}
Suppose $\kappa$ is a $\ch\delta_0$-orbit of simple roots.
Suppose $\gamma\in\ID$ and 
$(\lambda,\tau,\ell,t,\gamma)$ is an extended parameter.
Use the formulas for the cross action of $\kappa$ from Tables 2--4,
applied to $(\lambda,\tau,\ell,t)$, to define 
$(\lambda_1,\tau_1,\ell_1,t_1)$. 
Then define
$$
w_\kappa\times(\lambda,\tau,\ell,t,\gamma)=(\lambda_1+(w_\kappa*\gamma-
\gamma),\tau_1,\ell_1,t_1,w_\kappa*\gamma).
$$
Define the cross action of any element of $W^{\delta_0}$ by writing it
as a product of $w_\kappa$s. 
\end{definition}

If $\kappa$ is integral then $w_\kappa*\gamma=\gamma$, and Definition
\ref{d:nonintegralcross} agrees  with the definition of the cross
action in Tables 2--4. The main point is that even  if $\kappa$ is not
integral, the formula in Definition \ref{d:nonintegralcross} gives a
valid extended parameter. In particular the relation 
$$
(1-\ch\theta_1)(\lambda_1)=(1-\ch\theta_1)(\gamma-\rho)
$$
holds exactly as in the integral case. What we need to know is that  
$$
(1-\ch\theta_1)(\lambda_1+(w_\kappa*\gamma-\gamma))=(1-\ch\theta_1)
(w_\kappa*\gamma-\rho) 
$$
which follows immediately.
Furthermore in the integral case $\lambda_1\in X^*$. In the nonintegral
case it follows readily from the definitions that
$\lambda_1+(w_\kappa*\gamma-\gamma)\in X^*$ (even though 
this doesn't hold separately for $\lambda_1$ and $w_\kappa*\gamma-\gamma$).

\begin{proposition}\label{prop:nonintegralCayley}
Suppose $\gamma\in\ID$ is not necessarily integral. 
Then the action of $\mathcal H(\gamma)$ on $\mathcal M_\gamma$ is
given by the formulas in Table 5, with the following changes.

Suppose $\kappa$ is integral-simple, 
$w\in W^{\delta_0}$, and these satisfy:
$w\kappa$ is integral,  simple (for $R^+$), and the Cayley
transform
$c_{w\kappa}(w\times E)$ is defined by Tables 2--4.
Then define $c_\kappa(E)$ to be:
$$
c_\kappa(E)=w\inv\times c_{w\kappa}(w\times E)
$$
where the cross action is that of Definition \ref{d:nonintegralcross}.

On the other hand suppose $w\kappa$ is of type {\tt 2i11,2i12,2r11,2r12} for $w\times E$,
so $w\kappa\times(w\times E)$ is defined by Table 3.
Define
$$
w_\kappa\times E=w\inv\times[ w\kappa\times(w\times E)].
$$
\end{proposition}

It is helpful to reformulate the action of $W$.

\begin{lemma}
Suppose $E=(\lambda,\tau,\ell,t,\gamma)$ is an extended parameter, and
$w\in W^{\delta_0}$. Then 
$w\times E=(\lambda',\tau',\ell',t',w*\gamma)$ where:
$$
\begin{aligned}
\lambda'&=w*\gamma-w(\gamma-\lambda)+(w\rho-\rho)-(w\rho_{r}(x)-\rho_{r}(wxw\inv))\\
\tau'&=w\tau-(\ch\delta_0-1)(w\rho_r(x)-\rho_r(wxw\inv))/2\\
\ell'&=g-w(g-\ell)+(w\ch\rho-\ch\rho)-(w\rho_r(y)-\rho_r(wyw\inv)\\
t'&=wt-(\delta_0-1)(w\rho_r(y)-\rho_r(wy{w}\inv))/2
\end{aligned}
$$
\end{lemma}

The proof is that these formulas agree with those Definition
\ref{d:nonintegralcross} when $w=w_\kappa$. We omit the details.

To apply the Proposition we need the following Lemma.

\begin{lemma}
Suppose $\kappa$ is a $\ch\delta_0$-orbit of integral-simple roots. Then there
exists $w\in W^{\delta_0}$ such that $w\kappa$ is simple, unless
$\kappa=\{\alpha\}$ is of length $1$ and (the simple factor of) $G$ is
locally isomorphic to $SL(2n+1,\mathbb R)$. 
\end{lemma}

This follows from the facts that the ``quotient'' root system
$R/\delta_0$ \cite{steinberg} consisting of the restrictions of roots
to $H^{\delta_0}$, is a (possibly non-reduced) root system, with Weyl
group $W^{\delta_0}$; and in a reduced root system every root is
conjugate to a simple root.  The excluded case in the lemma is type
$A_{2n}$, in 
which case $R/\delta_0$ is the non-reduced system of type $BC_n$, and a
$\delta_0$-fixed root restricts to twice a root.

Extending Proposition \ref{prop:nonintegralCayley} to this excluded
case requires just a calculation in $SL(3,\mathbb R)$, which we omit.

\section{Duality}\label{sec:dual}
\setcounter{equation}{0}
\begin{definition}
Let $\tau$ be the anti-automorphism of $\mathcal H$ given by
\begin{equation}
q\tau(T_\kappa)=-q^\ell T_\kappa\inv=-T_\kappa+(q^\ell-1)\quad (\ell=\text{length}(\kappa)).
\end{equation}

Suppose $\pi$ is a representation of $\mathcal H$ on an $\mathcal A$-module
$V$. The dual representation $\pi^*$, on $\Hom_{\mathcal A}(V,\mathcal A)$ is given by 
$$
\pi^*(T_\kappa)(\lambda)(v)=\lambda(\pi(\tau(T_\kappa)v).
$$
\end{definition}

In the setting of Section \ref{sec:setting} let $\mathcal H$ be the
Hecke algebra for $(G,\delta_0)$ (see \cite{LVq} and Section
\ref{sec:hecke}).  Let $\ch\mathcal H$ be the algebra given by the
same construction applied to $(\ch G,\ch\delta_0)$.  If $\kappa$ is a
$\delta_0$-orbit of simple roots for $G$, then $\ch\kappa$ is a
$\ch\delta_0$-orbit of simple roots for $\ch G$, and the map
$T_\kappa\rightarrow T_{\ch\kappa}$ induces a Hecke algebra  isomorphism.

Fix (regular, rational) infinitesimal character $\gamma$ and (regular,
integral) infinitesimal character cocharacter $g$ as in
\eqref{se:extparam}. 

We now assume that $\gamma$ is integral.  Let $\mathcal M$ be
$\mathcal H$-module of Definition \ref{d:heckemodule}, applied to
$G$, $\delta_0$, and $\gamma$.  Recall $\mathcal M$ is spanned by
equivalence classes $[E]$ for $E$ an extended parameter with
infinitesimal character $\gamma$, and $[I_z(E)]\rightarrow [E]$ is an
isomorphism of Hecke modules.

Let $\ch\mathcal M$ be the $\ch\mathcal H$-module obtained by applying
the same construction to $\ch G,\ch\delta_0$ and $g$.
If $E$ is an extended parameter, write $\ch E$ for the same parameter,
viewed as an extended parameter for $\ch G$. 
The  map $[I_\zeta(\ch E)]\rightarrow [\ch E]$ is an isomorphism of
$\ch\mathcal H$-modules. Write $[E]'\in \Hom_\mathcal A(V,\mathcal A)$ for the dual basis vector.

\begin{proposition}
The map $[E]'\rightarrow (-1)^{\text{length}(E)} [\ch E]$ is an isomorphism
of $\mathcal H\simeq \ch\mathcal H$-modules.
\end{proposition}

\begin{proof}
The statement is equivalent to the following assertion. For all
$\kappa$, and extended parameters $E,F$:
\begin{subequations}
\renewcommand{\theequation}{\theparentequation)(\alph{equation}}  
\begin{equation}
\text{the coefficient of }[E]\text{ in }-T_\kappa([F])+(u^{\ell(\kappa)}-1)\sgn(E,F)
\end{equation}
is equal to 
\begin{equation}
(-1)^{\ell(E)-\ell(F)}*\text{the coefficient of }[\ch F]\text{ in
}T_{\ch\kappa}([\ch E]).
\end{equation}
\end{subequations}
In (a) $\sgn(E,F)$ is defined to be $0$ if $E,F$ are not extensions of
the same parameter.

Up to signs, all of these formulas can be read off easily from the
formulas for the Hecke algebra action on parameters. See Table
5. The fact that the signs are correct is due to the symmetry of Table 5. 
This is best illustrated by an example.

\begin{example}
Suppose $\kappa$ is type {\tt 1i1} for an extended parameter $F$.
Then  $\kappa$ is also of  type {\tt 1i1} for $w_\kappa\times F$.
\begin{subequations}
\renewcommand{\theequation}{\theparentequation)(\alph{equation}}  
According to Table 5,
\begin{equation}
\text{the coefficient of }[w_\kappa\times F]\text{ in }-T_\kappa([F])\text{ is }-1
\end{equation}
We need to show this equals
\begin{equation}
-1(\text{the coefficient of }[\ch F]\text{ in
}T_{\ch\kappa}([\ch(w_\kappa\times F)])
\end{equation}
From the same line in Table 5, applied to $\ch G$, we know
\begin{equation}
-(\text{the coefficient of }[\ch F]\text{ in }
  T_{\ch\kappa}([w_\kappa\times \ch F])=-1
\end{equation}
So we need to know
\begin{equation}
(w_\kappa\times F)^\vee\equiv w_\kappa\times\ch F.
\end{equation}
\end{subequations}
This identity reflects a symmetry of the tables.
Here $w_\kappa\times F$ is a cross action of type {\tt 1i1},
$w_\kappa\times\ch F$ is of type {\tt 1r1}.
Switching the roles of $\lambda\leftrightarrow\ell$, and
$\tau\leftrightarrow t$ interchanges these two formulas. 
\end{example}

The necessary symmetry holds for all Cayley transforms and cross
actions; in Table 5 the dual operations are listed on the same line.
This completes the proof of the Proposition.
\end{proof}

\section{Appendix}
\label{sec:paramnotes}
\setcounter{equation}{0}
\begin{subequations} \label{se:distW}
We collect a few  technical results about the Tits group \cite{Tits}, which
will be needed for our study of parameters for representations in
Section \ref{sec:atlasparameters}. We continue with the notation
of \eqref{se:notation}. 
For each simple root
$\alpha$, the pinning defines a canonical homomorphism
\begin{equation}
\phi_\alpha\colon (SL(2),\text{diag}) \rightarrow (G,H) \qquad
d\phi_\alpha\begin{pmatrix} 0&1 \\ 0 & 0\end{pmatrix} = X_\alpha.
\end{equation}
Similarly,
\begin{equation}
\phi_{\alpha^\vee}\colon (SL(2),\text{diag}) \rightarrow ({}^\vee
G,{}^\vee H) \qquad d\phi_{\alpha^\vee}\begin{pmatrix} 0&1 \\ 0 & 0\end{pmatrix} = X_{\alpha^\vee}.
\end{equation}

It is sometimes convenient to define also
\begin{equation}
H_\alpha = d\phi_\alpha\begin{pmatrix} 1&0 \\ 0 & -1\end{pmatrix},\qquad
X_{-\alpha} = d\phi_\alpha\begin{pmatrix} 0&0 \\ 1 & 0\end{pmatrix};
\end{equation}
the first element (because $\alpha(H_\alpha) = 2$) ``is'' the coroot
$\alpha^\vee$. The second is a preferred root vector for $-\alpha$,
characterized by the last of the three relations
\begin{equation}
\left[H_\alpha,X_\alpha\right] = 2X_\alpha,\quad
\left[H_\alpha,X_{-\alpha}\right] = -2X_{-\alpha},\quad
\left[X_\alpha,X_{-\alpha}\right] = H_\alpha.
\end{equation}
In this way we get a distinguished representative
\begin{equation}\label{e:sigmaalpha}
\begin{aligned}
\sigma_\alpha &=_{\text{def}} \phi_\alpha\begin{pmatrix} 0 & 1\\ -1&
  0\end{pmatrix}=\exp(\frac\pi2(X_\alpha-X_{-\alpha}))\\
 \qquad \sigma_\alpha^2 &= m_\alpha =_{\text{def}}
\alpha^\vee(-1)
\end{aligned}
\end{equation}
for the simple reflection $s_\alpha$. These representatives satisfy
the braid relations (see \cite{Tits}) and therefore define distinguished
representatives
\begin{equation}\label{e:distrep}
\sigma_w =_{\text{def}} \sigma_{\alpha_1}\sigma_{\alpha_2}\cdots
\sigma_{\alpha_r} \qquad (w = s_{\alpha_1}s_{\alpha_2}\cdots s_{\alpha_r} \
\text{reduced}) 
\end{equation}
for each Weyl group element $w$. (That $\sigma_w$ is independent of
the choice of reduced decomposition is a consequence of the fact that
the $\sigma_\alpha$ satisfy the braid relations.)  If $\gamma$ is any
distinguished (that is, pinning-preserving) automorphism of $(G,B,H)$,
then 
\begin{equation}
\gamma(\sigma_w) = \sigma_{\gamma(w)}.
\end{equation}
The braid relations imply, for any $w\in W$ and simple root $\alpha$:
\begin{equation}
\label{e:braid2}
\sigma_w\sigma_\alpha=
\begin{cases}
  \sigma_{ws_\alpha}&\text{length}(ws_\alpha)=\text{length}(w)+1\\
  \sigma_{ws_\alpha}m_\alpha&\text{length}(ws_\alpha)=\text{length}(w)-1
\end{cases}
\end{equation}
and a similar result  for $\sigma_\alpha\sigma_w$ (with $m_\alpha$ on the left)

In exactly the same way, we get a distinguished representative in
${}^\vee G$
\begin{equation}\label{e:dualdistrep}
{}^\vee\sigma_w =_{\text{def}}
\sigma_{\alpha_1^\vee}\sigma_{\alpha_2^\vee}\cdots 
\sigma_{\alpha_r^\vee} \qquad (w = s_{\alpha_1}s_{\alpha_2}\cdots
s_{\alpha_r} \ 
\text{reduced}) 
\end{equation}
\end{subequations} 

The main fact we need about these representatives is
\begin{proposition}
\label{p:bicycle}
In the setting of \eqref{se:distW}, 
$$\begin{aligned} \sigma_w \sigma_{w^{-1}} &
=(w\rho^\vee - \rho^\vee)(-1)\\
&= e((\rho^\vee -w\rho^\vee)/2)\\
&= \prod_{\substack{\beta\in R^+(G,H)\\ w^{-1}\beta\notin R^+(G,H)}} m_\beta.
\end{aligned}$$
\end{proposition}
Proof is an easy induction on $\ell(w)$. See \cite{contragredient}*{Lemma 5.4}.

\begin{proposition}\label{prop:signs}
In the setting \eqref{se:notation}, suppose $w\in W$, $\alpha,
\beta\in \Pi$ are simple roots, and 
$w\alpha = \beta$. Write $X_\alpha$ and $X_\beta$ for the simple root
vectors given by the pinning, and $\sigma_w\in N(H)$ for the Tits
representative of $w$ defined in \eqref{e:distrep}. Then
$$
\sigma_w\sigma_\alpha\sigma_w^{-1}=\sigma_\beta
$$
and
$$
\Ad(\sigma_w)(X_\alpha) = X_\beta,\quad
\Ad(\sigma_w)(X_{-\alpha}) = X_{-\beta}
$$
\end{proposition}
\begin{proof}
Since $\beta=w\alpha$, $s_\beta w=ws_\alpha$. 
If $\text{length}(ws_\alpha)=\text{length}(s_\beta w)=\text{length}(w)+1$ then
the first case of \eqref{e:braid2} implies
$$
\sigma_w\sigma_\alpha=\sigma_{ws_\alpha}=\sigma_{s_\beta w}=\sigma_\beta\sigma_w.
$$
If the lengths are decreasing we see
$$
\begin{aligned}
\sigma_w\sigma_\alpha&=\sigma_{ws_\alpha}m_\alpha\\
&=\sigma_{s_\beta w}m_\alpha\\
&=m_\beta\sigma_\beta \sigma_w m_\alpha\\
&=m_\beta m_{s_\beta w\alpha}\sigma_\beta\sigma_w\\
&=\sigma_\beta\sigma_w\quad\text{(since $s_\beta w\alpha=-\beta$)}.
\end{aligned}
$$
For the second statement we observe that $\Ad(\sigma_w)(X_\alpha)$ is some multiple of 
$X_{\beta}$. The Tits group preserves the $\Z$-form of $\mathfrak g$ generated by the various $X_{\pm\alpha}$, so this scalar is $\pm1$; 
we need to show it is $1$. We compute
$$
\begin{aligned}
\sigma_w\sigma_\alpha \sigma_w^{-1}&=
\sigma_w(\exp\frac\pi2(X_\alpha-X_{-\alpha}))\sigma_w\inv\\
&=\exp(\frac\pi2\Ad(\sigma_w)(X_\alpha-X_{-\alpha}))
\end{aligned}
$$
On the other hand by what we just proved this equals
$$
\sigma_\beta=\exp(\frac\pi2(X_\beta-X_{-\beta}))
$$
Setting these equal gives the two equalities in the second statement.
\end{proof}

\begin{corollary}\label{cor:signs}
In the setting \eqref{se:notation}, suppose $w\in W$, $\alpha,
\beta\in \Pi$ are simple roots, and 
$w\alpha = -\beta$. Write $X_\alpha$ and $X_\beta$ for the simple root
vectors given by the pinning, and $\sigma_w\in N(H)$ for the Tits
representative of $w$ defined in \eqref{e:distrep}. Then
$$
\sigma_w\sigma_\alpha\sigma_w\inv=\sigma_\beta
$$
and
$$
\Ad(\sigma_w)(X_{\alpha}) = -X_{-\beta}, \qquad
\Ad(\sigma_w)(X_{-\alpha}) = -X_{\beta}
$$
\end{corollary}
\begin{proof}
Let $w'=ws_\alpha$. 
The first assertion follows from the previous Lemma applied to $\sigma_{w'}$, using the fact that 
$\sigma_{w'}=\sigma_w\sigma_\alpha$ (since 
$w'=ws_\alpha$ is a reduced expression).
As in the proof of the previous Proposition we conclude
$$
\exp(\frac\pi 2(\Ad(\sigma_w)(X_\alpha-X_{-\alpha}))
=\exp(\frac\pi 2(X_\beta-X_{-\beta})),
$$
and in this case  this implies $\Ad(\sigma_w)(X_\alpha)=-X_{-\beta}$ 
and $\Ad(\sigma_w)(X_{-\alpha})=-X_\beta$.
\end{proof}

\begin{table}
\caption{\bf Types of roots, and their associated Cayley transforms}
\noindent\makebox[\textwidth]{%
\begin{tabular}{|c|c|c|c|}
\hline 
type &terminology of \cite{LVq}& definition & Cayley transform\\
\hline
{\tt 1C+}&complex ascent&$\alpha$ complex, $\theta\alpha>0$&\\
\hline
{\tt 1C-}&complex descent&$\alpha$ complex, $\theta\alpha<0$&\\
\hline
{\tt 1i1}&imaginary noncpt type I ascent&$\alpha$ imaginary, noncpt, type 1&$\gamma^\kappa=\gamma^\alpha$\\
\hline
{\tt 1i2f}&imaginary noncpt type II ascent&
\begin{tabular}{cc}$\alpha$ imaginary, noncpt, type 2\\
$\sigma$ fixes both terms of $\gamma^\alpha$
\end{tabular}
&$\gamma^\kappa=\gamma^\alpha=\{\gamma^\kappa_1,\gamma^\kappa_2\}$
\\\hline
{\tt 1i2s}&imaginary noncpt type II ascent&
\begin{tabular}{cc}$\alpha$ imaginary, noncpt, type 2\\ 
$\sigma$ switches the two terms of $\gamma^\alpha$
\end{tabular}
&\\
\hline
{\tt 1ic}&cpt imaginary descent&$\alpha$ cpt imaginary&\\
\hline
{\tt 1r1f}&real type I descent&
\begin{tabular}{cc}$\alpha$ real, parity, type 1\\ 
$\sigma$ fixes both terms of $\gamma_\alpha$
\end{tabular}
&$\gamma_\kappa=\gamma_\alpha=\{\gamma_\kappa^1,\gamma_\kappa^2\}$
\\
\hline
{\tt 1r1s}&real type I  descent&
\begin{tabular}{cc}$\alpha$ real, parity, type 1\\ 
$\sigma$ switches the two terms of $\gamma_\alpha$
\end{tabular}&
\\
\hline
{\tt 1r2}&real type II descent&
$\alpha$ real, parity, type 2&$\gamma_\kappa=\gamma_\alpha$\\
\hline
{\tt 1rn}&real nonparity ascent&
$\alpha$ real, non-parity&\\
\hline
{\tt 2C+}&two-complex ascent&
\begin{tabular}{c}
$\alpha,\beta$ complex $\theta\alpha>0$\\ $\theta\alpha\ne\beta$  
\end{tabular}
&\\
\hline
{\tt 2C-}&two-complex descent&
\begin{tabular}{c}
$\alpha,\beta$ complex $\theta\alpha<0$\\$\theta\alpha\ne\beta$
\end{tabular}&\\
\hline
{\tt 2Ci}&two-semiimaginary ascent&$\alpha,\beta$ complex, $\theta\alpha=\beta$&
$\gamma^\kappa=s_\alpha\times\gamma=s_\beta\times\gamma$\\
\hline
{\tt 2Cr}&two-semireal descent&$\alpha,\beta$ complex,$\theta\alpha=-\beta$&
$\gamma_\kappa=s_\alpha\times\gamma=s_\beta\times\gamma$
\\
\hline
{\tt 2i11}&
\begin{tabular}{c}
two-imaginary noncpt \\type I-I ascent
\end{tabular}&
\begin{tabular}{c}
$\alpha,\beta$ noncpt imaginary, type 1\\
$(\gamma^\alpha)^\beta$ single valued
\end{tabular}&$\gamma^\kappa=(\gamma^\alpha)^\beta$
\\
\hline
{\tt 2i12}&
\begin{tabular}{c}
two-imaginary noncpt \\type I-II ascent
\end{tabular}&
\begin{tabular}{c}
$\alpha,\beta$ noncpt imaginary, type 1\\
$(\gamma^\alpha)^\beta$ double valued
\end{tabular}&
$\gamma^\kappa=\{\gamma^{\kappa}_1,\gamma^{\kappa}_2\}=(\gamma^\alpha)^\beta$\\
\hline
{\tt 2i22}&
\begin{tabular}{c}
two-imaginary noncpt \\type II-II ascent
\end{tabular}&
\begin{tabular}{c}
$\alpha,\beta$ noncpt imaginary, type 1\\
$(\gamma^\alpha)^\beta$ has $4$ values
\end{tabular}&
$\gamma^\kappa=\{\gamma^\kappa_1,\gamma^\kappa_2\}=\{\gamma^{\alpha,\beta}\}^\sigma$\\
\hline
{\tt 2r22}&
\begin{tabular}{c}
two-real \\type II-II descent
\end{tabular}&
\begin{tabular}{c}
$\alpha,\beta$ real, parity, type 2\\
$(\gamma_\alpha)_\beta$ single valued
\end{tabular}&
$\gamma_\kappa=(\gamma_{\alpha})_\beta$
\\
\hline
{\tt 2r21}&
\begin{tabular}{c}
two-real \\type II-I descent
\end{tabular}&
\begin{tabular}{c}
$\alpha,\beta$ real, parity, type 2\\
$(\gamma_\alpha)_\beta$ double valued
\end{tabular}&
$\gamma_\kappa=\{\gamma_{\kappa}^1,\gamma_{\kappa}^2\}=(\gamma_\alpha)_\beta$\\
\hline
{\tt 2r11}&
\begin{tabular}{c}
two-real \\type I-I descent
\end{tabular}&
\begin{tabular}{c}
$\alpha,\beta$ real, parity, type 2\\
$(\gamma_\alpha)_\beta$ has $4$ values
\end{tabular}&
$\gamma_\kappa=\{\gamma_\kappa^1,\gamma_\kappa^2\}=\{(\gamma_{\alpha})_\beta\}^\sigma$\\
\hline
{\tt 2rn}&two-real nonparity ascent&$\alpha,\beta$ real, nonparity&\\
\hline
{\tt 2ic}&two-imaginary cpt  descent&$\alpha,\beta$ cpt imaginary&\\
\hline
{\tt 3C+}&three-complex ascent&$\alpha,\beta$ complex $\theta\alpha>0$, $\theta\alpha\ne\beta$&\\
\hline
{\tt 3C-}&three-complex descent&$\alpha,\beta$ complex $\theta\alpha<0$, $\theta\alpha\ne\beta$&\\
\hline
{\tt 3Ci}&three-semiimaginary ascent&$\alpha,\beta$ complex, $\theta\alpha=\beta$&
$\gamma^\kappa=(s_\alpha\times\gamma)^\beta\cap(s_\beta\times\gamma)^\alpha$\\
\hline
{\tt 3Cr}&three-semireal descent&$\alpha,\beta$ complex, $\theta\alpha=-\beta$&
$\gamma_\kappa=(s_\alpha\times\gamma)_\beta\cap(s_\beta\times\gamma)_\alpha$\\
\hline
{\tt 3i}&
three imaginary noncpt ascent&
$\alpha,\beta$ noncpt imaginary, type 1
&$\gamma^\kappa=s_\alpha\times\gamma^\beta=s_\beta\times\gamma^\alpha$
\\
\hline
{\tt 3r}&
three-real descent&
$\alpha,\beta$ real, parity, type 2
&$\gamma_\kappa=s_\alpha\times\gamma_\beta=s_\beta\times\gamma_\alpha$
\\
\hline
{\tt 3rn}&three-real non-parity ascent&
$\alpha,\beta$ real, nonparity&\\
\hline
{\tt 3ic}&three-imaginary cpt descent&
$\alpha,\beta$ noncpt imaginary&\\
\hline
\end{tabular}
}
\end{table}

\begin{sidewaystable}
\caption{\bf Cayley and cross actions on extended parameters: type {\tt 1}}
\label{t:shortcayleycross1} 
\begin{tabular}
{>{\small}l >{\small}c   >{\small}c >{\small}c >{\small}c >{\small}c} 
 \\[-1ex]
 & $\lambda|_{\mathfrak t} = (\gamma-\rho)|_{\mathfrak t}$ 
 & $({}^\vee\delta_0-1)\lambda=(1+{}^\vee\theta_y)\tau$ 
 & $\ell|_{\mathfrak a} = (g-\rho^\vee)|_{\mathfrak a}$ 
 & $(\delta_0-1)\ell = (1+\theta_x)t$ 
\\[.3ex] 
type & \normalsize{$\lambda_1$} &  \normalsize{$\tau_1$} &
\normalsize{$\ell_1$} & \normalsize$t_1$ & \normalsize{notes} 
\\[.5ex]
\hline\hline

{\tt 1C} crx& $s_\alpha\lambda+(\gamma_\alpha-1)\alpha$ & $s_\alpha\tau$ &
$s_\alpha \ell+(g_\alpha-1)\alpha^\vee$ & $s_\alpha t$ \\[.3ex] \hline

{\tt 1i1} crx & $\lambda$ & $\tau$ & $\ell+\alpha^\vee$ & $t$ & 
 \\[.3ex] \hline

{\tt 1i1} Cay & $\lambda$ & $\tau - \tau_\alpha\sigma$ &
$\ell+\frac{g_\alpha -\ell_\alpha - 1}{2} \alpha^\vee$ & $t$ 
&
$\alpha = (1+ {}^\vee\theta_{y_1})\sigma$
\\[.3ex] \hline 
\\[-2.2ex]

{\tt 1i2f} Cay & $\lambda$, $\lambda+\alpha$ &
$\tau-\frac{\tau_\alpha}{2} \alpha$ & $\ell+ \frac{g_\alpha
  -\ell_\alpha - 1}{2} \alpha^\vee$ & $t$ & $\tau_\alpha$ even \\[.3ex] \hline
\\[-2.2ex]

{\tt 1i2s} Cay & $\lambda$, $\lambda+\alpha$ &  [none] & $\ell+
\frac{g_\alpha -\ell_\alpha - 1}{2} \alpha^\vee$ & $t$ & $\tau_\alpha$
odd \\[.3ex] \hline 
\\[-2.2ex]

{\tt 1r1f} Cay & $\lambda + \frac{\gamma_\alpha-\lambda_\alpha-1}{2}
\alpha$ & $\tau$ & $\ell$, $\ell+\alpha^\vee$ & $t -
\frac{t_\alpha}{2}\alpha^\vee$ & $t_\alpha$ 
even \\[.3ex] \hline  
\\[-2.2ex]

{\tt 1r1s} Cay & $\lambda + \frac{\gamma_\alpha - \lambda_\alpha-1}{2}
\alpha$ &  $\tau$& $\ell$, $\ell+\alpha^\vee$ & [none] & $t_\alpha$
odd  \\[.3ex] \hline 
\\[-2.2ex]

{\tt 1r2} crx & $\lambda+\alpha$ &  $\tau$ & $\ell$ & $t$ &
\\[.6ex] \hline

{\tt 1r2} Cay & $\lambda+\frac{\gamma_\alpha-\lambda_\alpha-1}2\alpha$ &  $\tau$ & $\ell$ & $t - t_\alpha s$ 
& $\alpha^\vee = s + \theta_{x_1}s$ \\[.6ex] \hline

\hline\hline 

\end{tabular}
\end{sidewaystable} 

\begin{sidewaystable}
\caption{\bf Cayley and cross actions on extended parameters: type {\tt 2}}
\label{t:shortcayleycross2} 
\begin{tabular}
{>{\small}l >{\small}c >{\small}c  >{\small}c  >{\small}c >{\small}c } \\[-1ex]
 & $\lambda|_{\mathfrak t} = (\gamma-\rho)|_{\mathfrak t}$ & $
 ({}^\vee\delta_0-1)\lambda = (1+{}^\vee\theta_y)\tau$ & $\ell|_{\mathfrak
   a} = (g-\rho^\vee)|_{\mathfrak a}$ & $(\delta_0-1)\ell =
 (1+\theta_x)t$ \\[.3ex]  

type & \normalsize$\lambda_1$ & \normalsize{$\tau_1$} & \normalsize$\ell_1$ &
\normalsize$t_1$ & \normalsize{notes}
\\[.5ex]
\hline\hline

{\tt 2C} crx & $w_\kappa\lambda + (\gamma_\alpha-1)\kappa$ &
$w_\kappa\tau$ & $w_\kappa \ell  + (g_\alpha-1)\kappa^\vee$ &
$w_\kappa t$  \\[.3ex] \hline
\\[-2.2ex]

{\tt 2Ci} Cay& $s_\alpha\lambda + (\gamma_\alpha-1)\alpha $ &
$\tau-\frac{\tau_\alpha + \tau_\beta}{2}\alpha $ & $s_\alpha \ell +
(g_\alpha -1)\alpha^\vee$ & \small{$\begin{matrix}s_\alpha  t \, +
    \\[.9ex] (\ell_\alpha - g_\alpha+1)\alpha^\vee\end{matrix}$} & 
$\ell_\alpha = \ell_\beta$  \\[2ex] \hline 
\\[-2.2ex]

{\tt 2Cr} Cay & $s_\alpha\lambda + (\gamma_\alpha-1)\alpha $ &
\small{$\begin{matrix} s_\alpha \tau\, + \\[.9ex] (\lambda_\alpha - 
    \gamma_\alpha+1)\alpha \end{matrix}$} &  $s_\alpha \ell +
(g_\alpha -1)\alpha^\vee$ & $t-\frac{t_\alpha + t_\beta}{2}\alpha^\vee $ &
$\lambda_\alpha=\lambda_\beta$ \\[3ex] \hline  

{\tt 2i11} crx& $\lambda$ & $\tau$ & $\ell+\kappa^\vee$ & $t$\\[.3ex] \hline 
\\[-2.2ex]

{\tt 2i11} Cay& $\lambda$ &
$\begin{matrix}{\tau-\tau_\alpha\sigma(\alpha)}\\[.9ex]
  {-\,\tau_\beta\sigma(\beta)} \end{matrix}$ &
\small{$\begin{matrix}{\ell+\frac{g_\alpha - \ell_\alpha 
  -1}{2}\alpha^\vee}\\[.9ex]{+\,\frac{g_\beta - \ell_\beta 
  -1}{2}\beta^\vee}\end{matrix}$} & $t$ &
$ \begin{matrix}\alpha= \\[.9ex] (1+{}^\vee\theta_{y_1})\sigma(\alpha)
\end{matrix}$ \\[3ex] \hline
\\[-2.2ex]

{\tt 2i12} cr1x & $\lambda$ &  $\tau$ & $\ell+\alpha^\vee$ &  $t-s$ &
$\begin{matrix}\alpha^\vee - \beta^\vee\\ = s+\theta_x s\end{matrix}$
\\[1ex] \hline 
\\[-2.2ex]

{\tt 2i12f} Cay & $\lambda$, $\lambda + \alpha$ & $\begin{matrix}\tau  +
\tau_\beta\sigma\\[.9ex] -\, \frac{\tau_\alpha +  \tau_\beta}{2}\alpha
\end{matrix}$,\ \  $\tau_1-\sigma$  &
\small{$\begin{matrix}{\ell+\frac{g_\alpha - \ell_\alpha  
  -1}{2}\alpha^\vee}\\[.9ex]{+\,\frac{g_\beta - \ell_\beta
  -1}{2}\beta^\vee}\end{matrix}$} &  $t$ & $\begin{matrix}\tau_\alpha
+ \tau_\beta \text{\ even} \\ \alpha -  \beta\\ =
\sigma+{}^\vee\theta_{y_1} \sigma \end{matrix}$ \\[3ex] \hline  
\\[-2.2ex]

{\tt 2i12s} Cay & $\lambda$, $\lambda + \alpha$ & [none] &
\small{$\begin{matrix}{\ell+\frac{g_\alpha - \ell_\alpha 
  -1}{2}\alpha^\vee}\\[.9ex]{+\,\frac{g_\beta - \ell_\beta
  -1}{2}\beta^\vee}\end{matrix}$} &  $t$ & $\tau_\alpha + \tau_\beta$ odd
\\[3ex] \hline 
\\[-2.2ex] 

{\tt 2i22} Cay& $\begin{matrix}\lambda,\ \lambda+\kappa \text{\ \ \bf{OR}}\\
 \lambda+\alpha,\ \lambda+\beta \end{matrix}$ & $\begin{matrix} \tau -
\frac{\tau_\alpha}{2}\alpha - \frac{\tau_\beta}{2}\beta  \text{\ \
  \bf{OR}}\\ \tau -
\frac{\tau_\alpha\pm 1}{2}\alpha - \frac{\tau_\beta \mp
  1}{2}\beta \end{matrix}$ & \small{$\begin{matrix}{\ell+\frac{g_\alpha
      - \ell_\alpha - 1}{2}\alpha^\vee}\\[.9ex]{+\,\frac{g_\beta - \ell_\beta
  -1}{2}\beta^\vee}\end{matrix}$} & $t$
& $\begin{matrix}\text{$\tau_\alpha$, $\tau_\beta$ even \bf{OR}}\\
  \text{$\tau_\alpha$, $\tau_\beta$ odd} 
\end{matrix} $
\\[3ex] \hline 
\\[-2.2ex] 

{\tt 2r22} crx& $\lambda+\kappa$ &  $\tau$ & $\ell$ & $t$ \\[3ex] \hline 
\\[-2.2ex]

{\tt 2r22} Cay& \small{$\begin{matrix}{\lambda+\frac{\gamma_\alpha -
        \lambda_\alpha  - 1}{2}\alpha}\\[.9ex]{+\,\frac{\gamma_\beta -
        \lambda_\beta  - 1}{2}\beta}\end{matrix}$} &  $\tau$ & $\ell$
& $\begin{matrix}{t - t_\alpha s(\alpha^\vee)}\\[.9ex] 
  {-\,t_\beta s(\beta^\vee)} \end{matrix}$ & $\begin{matrix}\alpha^\vee
= \\[.9ex] (1+\theta_{x_1})s(\alpha^\vee) \end{matrix}$\\[3ex] \hline
\\[-2.2ex]
  
{\tt 2r21} cr1x& $\lambda+\alpha$ &  $\tau-\sigma$ & $\ell$ & $t$ &
$\begin{matrix}\alpha -   \beta\\ = \sigma +{}^\vee\theta_y
  \sigma\end{matrix}$ \\[3ex] \hline 
\\[-2.2ex]

{\tt 2r21f} Cay& \small{$\begin{matrix}{\lambda+\frac{\gamma_\alpha -
        \lambda_\alpha  - 1}{2}\alpha}\\[.9ex]{+\,\frac{\gamma_\beta -
        \lambda_\beta  -1}{2}\beta}\end{matrix}$} &  $\tau$ & $\ell$,
$\ell+\alpha^\vee$ & $\begin{matrix}t  + t_\beta  s\\[.9ex] -\, \frac{t_\alpha +
 t_\beta}{2}\alpha^\vee \end{matrix}$,\ \  $t_1-s$ 
& $\begin{matrix}t_\alpha + t_\beta \text{\ even} \\ \alpha^\vee - 
    \beta^\vee \\ =  s +\theta_{x_1} s \end{matrix}$ \\[3ex] \hline 
\\[-2.2ex]

{\tt 2r21s} Cay& \small{$\begin{matrix}{\lambda+\frac{\gamma_\alpha -
        \lambda_\alpha  - 1}{2}\alpha}\\[.9ex]{+\,\frac{\gamma_\beta -
        \lambda_\beta  - 1}{2}\beta}\end{matrix}$} & $\tau$ & $\ell$,
$\ell+\alpha^\vee$ & [none] &  $t_\alpha + t_\beta$ odd \\[3ex] \hline
\\[-2.2ex]

{\tt 2r11} Cay& \small{$\begin{matrix}{\lambda+\frac{\gamma_\alpha -
        \lambda_\alpha  - 1}{2}\alpha}\\[.9ex]{+\,\frac{\gamma_\beta -
        \lambda_\beta  - 1}{2}\beta}\end{matrix}$} & $\tau$ &
$\begin{matrix}\ell,\ \ell+\kappa^\vee \text{\ \ \bf{OR}}\\  \ell
  +\alpha^\vee,\ \ell+\beta^\vee \end{matrix}$ & \small{$\begin{matrix} t - \frac{t_\alpha}{2}\alpha^\vee -
    \frac{t_\beta}{2}\beta^\vee  \text{\ \ \bf{OR}}\\ t -
    \frac{t_\alpha\pm 1}{2}\alpha^\vee - \frac{t_\beta \mp 
  1}{2}\beta^\vee \end{matrix}$} & $\begin{matrix}\text{$t_\alpha$,
  $t_\beta$ even \bf{OR}}\\ \text{$t_\alpha$, $t_\beta$ odd}
\end{matrix} $ \\[4ex] \hline\hline 

\end{tabular}
\end{sidewaystable} 

\begin{sidewaystable}
\caption{\bf Cayley and cross actions on extended parameters: type {\tt 3}}
\label{t:shortcayleycross3} 
\begin{tabular}
{>{\small}l >{\small}c >{\small}c  >{\small}c  >{\small}c >{\small}c } \\[-1ex]

 & $\lambda|_{\mathfrak t} = (\gamma-\rho)|_{\mathfrak t}$ & $
 ({}^\vee\delta_0-1)\lambda = (1+{}^\vee\theta_y)\tau$ & $\ell|_{\mathfrak
   a} = (g-\rho^\vee)|_{\mathfrak a}$ & $(\delta_0-1)\ell =
 (1+\theta_x)t$ \\[.3ex]  

type & \normalsize$\lambda_1$ & \normalsize{$\tau_1$} & \normalsize$\ell_1$ &
\normalsize$t_1$ & \normalsize{notes}
\\[.5ex]
\hline\hline

{\tt 3C} crx & $w_\kappa\lambda + (\gamma_\kappa-2)\kappa$ &  $w_\kappa\tau$ & 
$w_\kappa \ell  + (g_\kappa-2)\kappa^\vee$ &  $w_\kappa t$  \\[.3ex]
\hline 
\\[-2.2ex]
 
{\tt 3Ci} Cay&  $\begin{matrix} \lambda\text{\ \  \bf{OR}}\\[.9ex] 
\lambda + \kappa \end{matrix}$ & $\tau - \frac{\tau_\kappa} {2}\kappa$
& \small{$\ell + (g_\alpha - 1 - \ell_\alpha)\kappa^\vee$}  &  $t$ &
$\begin{matrix}\text{$\gamma_\alpha -1 - \lambda_\alpha$ even 
    \bf{OR}}\\ \text{$\gamma_\alpha - 1- \lambda_\alpha$ odd}
\end{matrix} $
\\[.8ex]\hline
\\[-2.2ex]

{\tt 3Cr} Cay&$\lambda +\left(\gamma_\alpha - 1 - \frac{\lambda_\kappa}
  {2} \right)\kappa$ & $\tau$ & $\ell$ &  $t$ \\[.3ex]\hline 
\\[-2.2ex]

{\tt 3i} Cay& $\lambda$ &  $\tau$ & $\ell+\left(g_\alpha - 1 -
  \frac{\ell_\kappa}{2} \right)\kappa^\vee$ &  $t$
\\[.3ex]\hline \\[-2.2ex]

{\tt 3r} Cay& $\lambda + (\gamma_\alpha - 1 - \lambda_\alpha)\kappa$
&  $\tau$ & $\begin{matrix} \ell \text{\ \  \bf{OR}}\\[.9ex] 
 \ell + \kappa^\vee \end{matrix} $ &  $t - \frac{t_\kappa}{2}\kappa^\vee$
 & $\begin{matrix}\text{$g_\alpha -1- \ell_\alpha$ even
    \bf{OR}}\\ \text{$g_\alpha -1 - \ell_\alpha$ odd} \end{matrix}$
 \\[2ex]\hline\hline
\end{tabular}
\end{sidewaystable}

\begin{table}
\label{table:actionHecke}
\caption{\bf Action of Hecke operators}
\begin{tabular}{|c|c|c|c|}
\hline
\small$\kappa$-type($E$)&$T_\kappa(E)$&\small$\kappa$-type($E$)&$T_\kappa(E)$
\\
\hline
{\tt 1C+}&$w_\kappa\times E$&{\tt1C-}&$(q-1)E+q(w_\kappa\times E)$\\
\hline
{\tt 1i1}&$w_\kappa\times E+E_\kappa$&{\tt 1r2}& \small
$(q-1)E-w_\kappa\times E+(q-1)E_\kappa$\\ 
\hline
{\tt 1i2f}&$E+E_\kappa^1+E_\kappa^2$&{\tt1r1f}& \small
$(q-2)E+(q-1)(E_\kappa^1+E_\kappa^2)$\\ 
\hline
{\tt 1i2s}&$-E$&{\tt1r1f}&$qE$\\
\hline
{\tt 1ic}&$qE$&{\tt1rn}&$-E$\\
\hline
{\tt 2C+}&$w_\kappa\times E$&{\tt 2C-}&$(q^2-1)E+q^2(w_\kappa\times E)$\\
\hline
{\tt 2Ci}&$qE+(q+1)E_\kappa$&{\tt 2Cr}&$(q^2-q-1)E+(q^2-1)E_\kappa$\\
&(see Section \ref{sec:2Ci})&&(see Section \ref{sec:2Ci})\\

\hline
{\tt 2i22}&$w_\kappa\times E+E_\kappa^1+E_\kappa^2$&{\tt 2r11}& \small
$(q^2-2)E+(q^2-1)(E_\kappa^1+E_\kappa^2)$\\ 

\hline
{\tt 2i11}&$w_\kappa\times E+E_\kappa$&{\tt 2r22}&\small
$(q^2-1)E-w_\kappa\times E-(q^2-1)E_\kappa$\\  
 
\hline
{\tt 2i12}&$w_\kappa\times E+E_\kappa^1-E_\kappa^2$&{\tt 2r21}& \small $(q^2-2)E+(q^2-1)(E_\kappa^1-E_\kappa^2)$\\
&(see Section \ref{sec:2i12})&&(see Section \ref{sec:2i12})\\
\hline
{\tt 2ic}&$q^2E$&{\tt2rn}&$-E$\\

\hline
{\tt 3C+}&$w_\kappa\times E$&{\tt 3C-}&$(q^3-1)E+q^3(w_\kappa\times E)$\\
\hline
{\tt 3Ci}&$qE
+(q+1)E_\kappa$&{\tt 3Cr}&$(q^3-q-1)E+(q^3-q)E_\kappa$\\
\hline
{\tt 3i}&$qE+(q+1)E_\kappa$&{\tt 3r}&$(q^3-q-1)E+(q^3-q)E_\kappa$\\
\hline
{\tt 3i}&$q^3E$&{\tt3r}&$-E$\\
\hline
\end{tabular}
\end{table}

\end{document}